\documentclass[11pt, reqno]{amsart}
\usepackage{amssymb}
\usepackage{graphicx}
\usepackage{tensor}

\usepackage{fullpage} 

\usepackage[english]{babel}
\usepackage[utf8]{inputenc}

\usepackage{amsfonts}
\usepackage{amssymb}
\usepackage{amsmath}
\usepackage{amsthm}

\usepackage{pdfsync}

\usepackage{amssymb}
\usepackage{epstopdf}
\usepackage{hyperref}

 \usepackage{accents}

\usepackage{enumerate}
\usepackage{enumitem}

\usepackage{cases}

\newtheorem{theorem}{Theorem}[section]
\newtheorem{corollary}[theorem]{Corollary}
\newtheorem{lemma}[theorem]{Lemma}
\newtheorem{proposition}[theorem]{Proposition}
\theoremstyle{definition}
\newtheorem{definition}[theorem]{Definition}

\theoremstyle{remark}
\newtheorem{remark}[theorem]{Remark}
\numberwithin{equation}{section}

\newcommand{\tr}{\textrm{tr}}
\newcommand{\supp}{{\mathrm{supp}}}

\newcommand{\mb}{\mathbb}
\newcommand{\be}{\begin{equation}}
\newcommand{\ee}{ \end{equation}}

\renewcommand{\Re}{\mathrm{Re}}
\renewcommand{\Im}{\mathrm{Im}}

\newcommand{\ep}{\varepsilon}

\newcommand{\sg}{\sigma}
\newcommand{\al}{\alpha}

\newcommand{\dd}{\,\mathrm{d}}

\newcommand{\vn}[1]{\|#1\|}
\newcommand{\vm}[1]{\left|#1\right|}

\newcommand{\tht}{\theta}

\newcommand{\lpr}{ \left( }
\newcommand{\rpr}{ \right) } 

\newcommand{\lng}{\langle}
\newcommand{\rng}{\rangle}

\newcommand{\lmd}{\lambda}
\newcommand{\lpp}{\left[}
\newcommand{\rpp}{\right]}

\newcommand{\ls}{\lesssim}

\newcommand{\dbtilde}[1]{\accentset{\approx}{#1}}

\newcommand*\jb[1]{\left\langle #1 \right\rangle} 

\newcommand{\pt}{\partial}

\usepackage{mathtools}
\newcommand{\defeq}{\vcentcolon=}

\usepackage{enumitem}
\usepackage{verbatim}

\begin{document}

\title[Yang-Mills]{The gauge-invariant I-method for Yang-Mills}
\author{Cristian Gavrus}%
\email{gcd2006@gmail.com}%

\thanks{The author thanks Ben Dodson for suggesting the paper \cite{keel2011global} and for discussions on the I-method, and thanks Sung-Jin Oh for discussions on his Yang-Mills papers. 
The author expresses gratitude to the John Hopkins University, where a large part of this work was carried.}

\begin{abstract}   
We prove global well-posedness of the $ 3d $ Yang-Mills equation in the temporal gauge in $ H^{\sg} $ for $ \sg > \frac{5}{6} $. 

Unlike related equations, Yang-Mills is not directly amenable to the method of almost conservation laws (I-method) in its Fourier and global version. 
We propose a modified energy which: 
\begin{enumerate}
\item Is gauge-invariant and easy to localize
\item Provides local gauges which give control of local Sobolev norms (through an Uhlenbeck-type lemma for fractional regularities)
\item Is slightly smoother in time compared to the classical I-method energy for related systems. 
\end{enumerate}

The spatial smoothing is realized via the Yang-Mills heat flow instead of the multiplier $ I $.

Due to the temporal condition and its finite speed of propagation, the local gauge selection is compatible with recent initial data extension results. Therefore, smoothened energy differences $  \mathcal{E}(t_1,s) -  \mathcal{E}(t_0,s) $ can be partitioned into local pieces whose (appropriately extended) bounds can be square summed. After revealing the null structure within the trilinear integrals, these can be estimated using known methods.   

In an appendix we show how an invariant modified energy for Maxwell-Klein-Gordon can extend previous results to regularities $ \sg > \frac{5}{6} $. 
\end{abstract}



\maketitle

\setcounter{tocdepth}{1}
\tableofcontents

\section{Introduction}

In this paper we consider the problem of low regularity global existence for the initial value problem for the Yang-Mills equations on $ \mb{R}^{3+1} $ with a non-abelian structural group $ G $. 
We say that the connection $ A_{\al}:  \mb{R}^{3+1} \to  \mathfrak{g} $ solves the \emph{Yang-Mills equation} if 
\be \label{YM} \tag{YM}
\mathbf{D}^{\al} F_{\al \beta} = 0 
\ee 
where 
$$ F_{\al \beta} \defeq \pt_{\al} F_{\beta} - \pt_{\beta} F_{\alpha} + [A_{\al}, A_{\beta}] \qquad \text{and} \qquad  \mathbf{D}_{\al} B \defeq \pt_{\al} B + [A_{\al}, B]   $$        
denote the \emph{curvature} and the \emph{covariant derivative}, while  $ \mathfrak{g} $ is the associated Lie algebra. The background on Lie theory notions used is presented in section \ref{LieDefinitions}. 

Expanding, \eqref{YM} takes the form 
$$ \partial^{\alpha} F_{\alpha \beta}+\left[A^{\alpha},\partial_{\alpha} A_{\beta}-\partial_{\beta} A_{\alpha}+\left[A_{\alpha}, A_{\beta}\right]\right]=0
$$
An important feature of the \eqref{YM} equations is its \emph{gauge invariance}. Given a $ G $ -valued  potential $ U $, \eqref{YM} is invariant under the following gauge transformations 
\be \label{GaugeTransformation}
A_{\al} \mapsto U A_{\al} U^{-1} - \partial_{\al} U U^{-1}, \quad F_{\al \beta} \mapsto U F_{\al \beta}U^{-1}.
\ee 
For this reason, the Cauchy problem cannot be well-posed until a choice of gauge is made. If we impose the \emph{Temporal gauge} condition $ A_0 = 0 $ the equations become 
\be \label{YMTemporal}
\begin{aligned}
 \pt_t \text{div} A & = [ \pt_t A_j, A_j ] \\
 \Box A_i - \pt_i \text{div} A & = \pt_j [A_i, A_j] +  [  \pt_j A_i, A_j] + [A_j, \pt_i A_j ] - [A_j, [A_j, A_i ]] 
\end{aligned}
\ee
These equations can be rewritten as \eqref{YMTemporal2} or \eqref{YMTemporal3}. 

The initial data sets consist of $ (\bar{A}_{i}, \bar{E}_i ) = (A_i, F_{0i})(t=0) $ for $ i=1,2,3 $. Note that in the temporal gauge $ F_{0i}= \pt_t A_i $. Due to the first equation in \eqref{YMTemporal}, i.e. \eqref{YM} for $ \beta=0 $, initial data sets have to satisfy the \emph{constraint} (or Gauss) \emph{equation}
\be \label{Constraint}
\bar{\mathbf{D}}^{\ell} \bar{E}_{\ell} \defeq  \partial^{\ell} \bar{E}_{\ell}+\left[\bar{A}^{\ell}, \bar{E}_{\ell}\right]=0.
\ee

The energy of a connection $ A_{\al} $ at $ t $ is defined by 
\be \label{Energy}
{\bf E}[F_{\al \beta}](t) \defeq \frac{1}{2} \sum_{\alpha<\beta}  \int_{\mathbb{R}^{3}} (F_{\alpha \beta} (t), F_{\alpha \beta}(t))  \dd x
\ee
and it is conserved for sufficiently smooth solutions of \eqref{YM}.

\subsection{Prior Yang-Mills results on  $ \mb{R}^{3+1} $} The early works of Segal \cite{segal1979cauchy}, Choquet-Bruhat and Christodoulou \cite{choquet1981existence}, and Eardley-Moncrief \cite{eardley1982global} establish local and global well-posedness with high regularity (e.g. $ H^2 $).   

Finite energy global well-posedness for \eqref{YM} was first proved by Klainerman and Machedon in \cite{klainerman1995finite}:

\begin{theorem}[\cite{klainerman1995finite}] \label{GWPKM} The Yang-Mills equation in the temporal gauge is globally well-posed for finite energy initial data which is locally in $ H^1 \times L^2 $.  
\end{theorem} 

The proof involves local Coulomb gauges based on Uhlenbeck's lemma. Two other proofs appear in \cite{oh2015finite} and \cite{oh2019hyperbolic}: one is based on the Yang-Mills heat flow for gauge selection, the other proof rests on initial data surgery techniques. In this work we need to combine these three methods together with some new ones.

Going below the energy regularity the first result was due to Tao \cite{tao2003local}:

\begin{theorem}[\cite{tao2003local}] \label{LWPTao} Let $ \sg> \frac{3}{4} $. The Yang-Mills equations in the temporal gauge \eqref{YMTemporal} are locally well-posed on  $ [-1,1] \times \mb{R}^3 $ for sufficiently small initial data in $ H^{\sg} (\mb{R}^3) \times H^{\sg-1}(\mb{R}^3) $ satisfying \eqref{Constraint}. 
\end{theorem} 

Removing the smallness assumption by shrinking the time interval is a delicate matter. This was obtained by Oh and T\u{a}taru in \cite{oh2019hyperbolic} as a consequence of their initial data excision and extension techniques for the Gauss equation \eqref{Constraint}:

\begin{theorem}[\cite{oh2019hyperbolic}] \label{LWPHloc}
Let $ \sg> \frac{3}{4} $.
The Yang-Mills equation in the temporal gauge is locally well-posed for initial data  satisfying \eqref{Constraint} and in $ H^{\sg}_{loc} (\mb{R}^3) \times H^{\sg-1}_{loc}(\mb{R}^3) $, on a sufficiently small time interval. 
\end{theorem} 

By a square summability argument one can formulate this result for global Sobolev spaces as well (see section \ref{SecProofCorLWP}):

\begin{corollary}[Local well-posedness in $  H^{\sg} \times H^{\sg-1} $]   \label{CorLWP}
Let $ \sg> \frac{3}{4} $. Suppose the initial data $  (\bar{A}_{x}, \bar{E}_{x}) $ is in $ H^{\sg}(\mb{R}^3) \times H^{\sg-1}(\mb{R}^3) $ and satisfies \eqref{Constraint}. Then the local in time solution from Theorem \ref{LWPHloc} is in $ C_t H^{\sg} \cap C_t^1 H^{\sg-1} $. For any $ t $ in the time interval the solution map $ (\bar{A}_{x}, \bar{E}_{x}) \mapsto A[t] $ is locally Lipschitz continuous on $ H^{\sg} \times H^{\sg-1} $. 
\end{corollary}

Persistence of higher regularity in $ H^{\bar{\sg}} \times H^{\bar{\sg}-1} $ holds too ($ \bar{\sg} \geq \sg$).
Solutions solve \eqref{YMTemporal} in the sense of distributions. Note that $ [ A,\partial A ] \in \dot{H}_{x}^{-1} \subset \mathcal{S}^{*} $ due to \eqref{Sbvone}. 

We remark that extending Theorem \ref{LWPTao} (and thus also \ref{LWPHloc}- \ref{CorLWP}) to $ \sg > \frac{1}{2} $ (which is the critical regularity) as in the case of Maxwell-Klein-Gordon \cite{machedon2004almost} remains an open question. 



\subsection{Main result} \

The purpose of this article is to obtain global solutions with regularities below the energy: 

\begin{theorem} \label{MainThm}
Let  $ \sg> \frac{5}{6} $. The Yang-Mills equation in the temporal gauge is globally well-posed in $ H^{\sg}(\mb{R}^3) \times H^{\sg-1}(\mb{R}^3) $: the solutions in Corollary \ref{CorLWP} extend to $ C_t H^{\sg} \cap C_t^1 H^{\sg-1} (\mb{R}_t \times \mb{R}^3).  $
\end{theorem} 

\

This result was motivated by \cite{keel2011global} which proves a global result for the Maxwell-Klein-Gordon (MKG) equations in the Coulomb gauge for regularities $ \sg > \frac{\sqrt{3}}{2} $ . Indeed, results on Yang-Mills usually follow in the footsteps of similar results for MKG, although the transition can often be highly non-trivial. The paper \cite{keel2011global} uses the I-method 
which proceeds by inserting a smoothing Fourier multiplier operator $ I=I_N $ (see \eqref{Ioperator}) into the MKG Hamiltonian (Energy) $ H $. The main technical step is showing that $ H[ I \Phi(t)] $ is "almost conserved" - i.e. it varies very slowly in time. Letting $ N \to \infty $  allows control of the size of the solution for long times. 
See also: \cite{bourgain}, \cite{keel1998local}, \cite{kenig2000global}, \cite{colliander2001},  \cite{colliander2001global},  \cite{colliander2002almost}, \cite{colliander2004multilinear}, \cite{colliander2007resonant}, \cite{roy2007adapted}, \cite{dodson2013global}, \cite{hani2012global}. For an introduction to the I-method for the cubic wave equation we refer to \cite{roy2008global}. 

\

However, applying the usual I-method directly to \eqref{YM} is unlikely to work and a new approach is needed for the modified energy. We discuss the reasons below.  

\

To obtain a global in time result for \eqref{YM} the (global or local) choices of gauge are essential. In particular, the energy $ {\bf E}[F](t) $ or modified energies (such as $ {\bf E}[I F](t) $
) should provide control of (global or local) Sobolev norms such as $ \vn{\nabla_{t,x} A_i}_{L^2} $ or $ \vn{I \nabla_{t,x} A_i}_{L^2} $. We must consider:
\begin{enumerate}
\item In general, the curvature $ F_{ij} $ can control only the curl $ \pt_i A_j - \pt_j A_i $ part of $ \nabla_x A_i $. For instance, in the temporal gauge $ F_{0i} = \pt_t A_i $ and $  {\bf E}[F](t) \approx \vn{\nabla_x A^{df} (t)  }_{L^2_x}^2 +  \vn{\pt_t A(t)  }_{L^2_x}^2 $ so the energy gives no Sobolev control of the curl-free part $ \nabla_x A^{cf} $. Setting the Coulomb condition $ \text{div} A_x = 0 $ would make $ A^{cf}=0 $ automatically. However, 
\item Global Coulomb gauges cannot be imposed for large data when $ G $ is non-abelian \cite{klainerman1995finite}, \cite{gribov1978quantization}. This is a fundamental difference to the MKG case \cite{klainerman1994maxwell}, \cite{keel2011global}. Local Coulomb gauges were used instead in \cite{klainerman1995finite} where Uhlenbeck's lemma provided control of local $ H^1 $ norms. 
\item Unless one fixes a global gauge (and does all the analysis in that gauge), setting the local gauges and patching the solutions together using finite speed of propagation (in the temporal gauge) requires the gauge invariance of the energy, \cite{klainerman1995finite}, \cite{oh2019hyperbolic}. However, 
\item The classical I-method modified energy $ {\bf E}[F[IA]](t) $ or $ {\bf E}[I F](t) $ fails to be gauge-invariant or easy to localize, like the original energy. This is due to the presence of the Fourier multiplier $ I $, see \eqref{Ioperator}. 
\end{enumerate}

\

To deal with these considerations we propose a modified energy which:
\begin{enumerate}[label=\Roman*.]
\item Is gauge invariant and fairly easy to localize, addressing (3),(4)
\item Addresses issues (1),(2) through Proposition \ref{LocalGauge1} (which is an Uhlenbeck-type lemma below $ H^1 $) providing gauge selections which give control of local Sobolev norms, see Remark \ref{UhlenbeckType}.
\item Is smoother in time compared to the classical I-method modified energy\footnote{
The structure of the differentiated energy effectively shifts a derivative from high to low frequency (Remark \ref{ComparisonAC})}, slightly improving the numerology compared to the MKG case in \cite{keel2011global}, allowing lower regularities $ \sg > \frac{5}{6} $ rather than $ \sg > \frac{\sqrt{3}}{2} $, see Remark \ref{ComparisonAC}. 

 \end{enumerate}

\

Instead of using the Fourier multiplier $ I $ in the definition, the same degree of spatial smoothing will be realized geometrically using the Yang-Mills heat flow 

\begin{minipage}{0.3\textwidth}
\begin{equation}
F_{si}= \mathbf{D}^{\ell} F_{\ell i}  \label{cYMHF} \tag{cYMHF}
\end{equation} 
    \end{minipage}%
    \begin{minipage}{0.2\textwidth}\centering
    and 
    \end{minipage}%
    \begin{minipage}{0.4\textwidth}
\begin{equation}
F_{s \al}= \mathbf{D}^{\ell} F_{\ell \al}.  \label{dYMHF} \tag{dYMHF}
\end{equation}
    \end{minipage}\vskip1em
    
introduced by Donaldson \cite{donaldson1985anti} and developed extensively by Sung-Jin Oh in \cite{oh2014gauge}, \cite{oh2015finite} and in \cite{oh2017yang}, \cite{oh2020hyperbolic} for the analysis of finite energy \eqref{YM}, along with its dynamic version \eqref{dYMHF}, together with Oh's related notion of caloric gauges for \eqref{YM}, paralleling the caloric gauges for Wave Maps \cite{tao2004geometric} and Schr\"odinger Maps \cite{TaoGauges}, \cite{bejenaru2011global} introduced by Tao. 
The smoothing is realized in a parabolic way, effectively implementing a covariant Littlewood-Paley decomposition, see Remark \ref{RmkSmoothingLP}.

\

As an alternative to II one could address (1),(2) by opting for global in space caloric gauges like in \cite{oh2014gauge}, \cite{oh2015finite} instead of local ones. Unfortunately, that approach requires integration in time which would worsen the numerology and would also be more technically difficult to implement in the present context.

\

If $ A_{\al}:  \mb{R} \times \mb{R}^{3} \times [0,s_0] \to  \mathfrak{g} $ is a regular connection, using the bi-invariant inner product from \eqref{InvMetric}, we denote the energies at each $ s \in [0,s_0] $ by 
\be \label{EnergiesS}
\mathcal{E}(t,s) \defeq  \frac{1}{2} \sum_{\alpha<\beta}  \int_{\mathbb{R}^{3}} (F_{\alpha \beta} , F_{\alpha \beta})(t,s)  \dd x
\ee
\begin{definition}[\bf{Modified energy}]
Let $ N^2 s_0=1 $. We define the following \emph{modified energy}
\be \label{ModifiedEnergy}
\mathcal{IE}(t) \defeq \sup_{s \in [0,s_0]} (N^2 s)^{1-\sigma} \mathcal{E}(t,s) + \int_0^{s_0} (N^2 s)^{1-\sigma}  \mathcal{E}(t,s) \frac{\dd s}{s} 
\ee
\end{definition}

\

The similarity of $ \mathcal{IE}(t) $ with the classical modified energy $  {\bf E}[F[IA]](t) $ is discussed in Remark \ref{RkComparisonModEn} based on the heat flow evolution relation to Littlewood-Paley theory (Remark \ref{RmkSmoothingLP}).

The key to the proof of the main result is the following:

\begin{proposition}[Almost conservation law] \label{AClaw}
Let  $ \sg> \frac{5}{6} $. Let $ A_{t,x} $ be a global regular solution to \eqref{YM} in the temporal gauge. Suppose $ A_{t,x,s} $ solves \eqref{dYMHF} on $ \mb{R} \times \mb{R}^3 \times [0,s_0] $. Let $ t_0 $ be a time such that 
$$
\mathcal{IE}(t_0) \leq 2 \eta^2 , \qquad \qquad \vn{A[t_0]}_{H^{\sg} \times H^{\sg-1}} \leq M_0
$$
Then for any $ t \in [t_0,t_1] $, $ t_1 = t_0+1 $ and any $ s \in [0,s_0] $, $ s_0 = N^{-2} $ one has 
\be \label{deltaEnergy}
(N^2 s)^{1-\sigma} \vm{  \mathcal{E}(t,s) -  \mathcal{E}(t_0,s) }   \ls  \frac{s^{\frac{1}{4}-}}{  (N s^{\frac{1}{2}})^{(1-\sg)}}  \eta^2  \leq \frac{ \eta^2}{N^{\frac{1}{2}-}}
\ee
as well as 
\be  \label{SobolevBound}
 \vn{A[t]}_{H^{\sg} \times H^{\sg-1}} \leq C(M_0).
\ee 
In particular, $ \vm{\mathcal{IE}(t) - \mathcal{IE}(t_0) } \ls \eta^2 / N^{\frac{1}{2}-}  $. 
\end{proposition}

\

The error $ 1/N^{\frac{1}{2}-} $ corresponds to the condition $ \sg> \frac{5}{6} $, which appears to be the limit of this method in the current form. 

\begin{remark} \label{RmkSqSumAC}
A feature of the proof is that we do not prove the almost-conservation law for the whole solution at once (this is due to the fact that we only control the solution in local gauges). Instead, we partition $  \mathcal{E}(t_1,s) -  \mathcal{E}(t_0,s) $ into local pieces (see \eqref{deltaEnergyPartition}) and then we square sum the estimates obtained for the (appropriately extended) local solutions. 
\end{remark}

Finally, let 
\be  \label{wDef}
w_{\al}(s) \defeq \mathbf{D}^{\beta} F_{\al \beta}(s) 
\ee
be the Yang-Mills tension field, which is a measure of the failure of $ A_{\al}(s) $ to satisfy the \eqref{YM} equation for $ s>0 $, since the heat flow and \eqref{YM} do not exactly commute. The gauge transformation \eqref{GaugeTransformation} makes $ w_{\al} \mapsto U w_{\al} U^{-1} $. 

\

\subsection{Overview of the paper and methods} \

We now outline the key results which play a role in the proof and their organization. 

\

Section \ref{SecPreliminaries} introduces the basic definitions, notations, properties and preliminaries. 

Section \ref{SecLWPYM} is devoted to local well-posedness for \eqref{YM}. It begins by addressing $ H^{\sg} \times H^{\sg-1} $ approximations by regular initial data sets satisfying \eqref{Constraint}. Then in Prop. \ref{PropsqSumWs2} we discuss square summability for fractional local Sobolev spaces $ \vn{u}_{H^s}^2 \simeq \sum_j \vn{u}_{W^{s,2}(B_j)}^2 $. We continue by recalling initial data extension results from \cite{oh2019hyperbolic} before reviewing Theorem \ref{LWPHloc} and discussing Corollary \ref{CorLWP}.

\

Section \ref{SecModLWP} develops space-time control of temporal solutions in: hyperbolic $ X^{s,b} $ spaces for the divergence free part $ A^{df} = {\bf P} A $ and in $ X^{r,\tht}_{\tau=0} = H_x^r H^{\tht}_t $ spaces for the curl free part $ A^{cf} = {\bf P}^{\perp} A $. As revealed in  \cite{tao2003local}, small solutions can enjoy $ \frac{1}{4} $ more regularity for $ A^{cf} $ after a suitable change of gauge. Locally in time, by Proposition \ref{ModLWP} one will have $ I A^{df} \in X^{1,\frac{3}{4}+}, \ I A^{cf} \in  H_x^{1+\frac{1}{4}} H^{\frac{1}{2}+}_t $ provided one can somehow make the solution small. 

\

Section \ref{SecYMHF} is reserved for the Yang-Mills heat flow in deTurck's and caloric gauges including: well-posedness in $ I^{-1}H^1 $ and $ H^{\rho} $, control of the modified energy \eqref{ModifiedEnergy} from the initial data and the change between the two gauges as in \cite{oh2014gauge} (deTurck's trick).  

\

A bound on the modified energy $ \mathcal{IE}(t) $ allows us to control higher covariant derivatives of $ F_{\al \beta}(s) $ in $ L^2 $. This is proved in Proposition \ref{CovBdsF1} using covariant energy estimates for parabolic equations. The same is true for $ w_x(s) $, see Section \ref{EnEstSec}. To implement Remark \ref{RmkSqSumAC},
we show that these bounds can be localized while maintaining square summability:

\begin{proposition} \label{PropCoeff} 

Let $ A_{t,x,s} $ be a regular solution to \eqref{YM} and \eqref{dYMHF} on an interval $ J \times \mb{R}^3 \times [0,s_0] $ satisfying $ \quad 
 \mathcal{IE}(t) \leq C  \eta^2 \ll 1 \quad \forall \ t \in J.  $

Let $ ( B_j)_{j \in \mathcal{J}} $ be a finitely overlapping covering of $ \mathbb{R}^3 $ by balls of radius $ \simeq 1 $.  

Then there exist square summable coefficients $ (c_j)_{j \in \mathcal{J}} $ associated to  $ ( B_j)_{j \in \mathcal{J}} $ with
\begin{align}
\label{BallHighCovDer1}
 \vn{  (N s^{\frac{1}{2}})^{1-\sigma} (s^{\frac{1}{2}} \mathbf{D}_x)^m F_{\al \beta}(t,s) }_{L^{\infty}_s L^2_x \cap L^{2}_{\frac{\dd s}{s}} L^2_x ( [0,s_0]\times B_j) } \leq c_j  \\
 \label{BallHighCovDer4}
 \vn{ (N^2 s)^{1-\sigma} s^{\frac{1}{4}} (s^{\frac{1}{2}} \mathbf{D}_x)^m w_x(t,s) }_{L^{\infty}_s L^2_x \cap L^{2}_{\frac{\dd s}{s}} L^2_x ( [0,s_0]\times B_j)} \leq c_j   
 \end{align}
holding for all $ t \in J $, $ j \in \mathcal{J} $, $ 0 \leq m \leq 8 $ and 
\be \label{sqsummcj}
\sum_{j \in \mathcal{J}} c_j^2 \ls \eta^2.  
\ee
\end{proposition}

\

Once these coefficients are defined, section \ref{SecGauge} shows that on each ball we can define local gauges using:

\begin{proposition} \label{LocalGauge1}
Let $ B$ be a ball of radius $ \simeq 1 $ and $ N^2 s_0=1 $. 

(1)
Let $ A_{x,s} $ be a smooth solution to \eqref{cYMHF} in the caloric gauge $ A_s = 0 $ on $ B\times [0,s_0] $. There is a sufficiently small $ \delta > 0 $ such that if the curvature $ F_{ij} $ obeys
\be \label{bddFx}
\sup_{0 \leq m \leq 3} \sup_{s \in [0,s_0]}  (N s^{\frac{1}{2}})^{1-\sg} \vn{ (s^{\frac{1}{2}} \mathbf{D}_x)^m F(s) }_{L^2(B)} \leq \delta 
\ee
then there exists a spatial field $ U: B\to G $,  $ \ U=U(x) $such that the transformation 
\be \label{Uchange}
\tilde{A}_i = U A_i U^{-1} - \partial_i U U^{-1}
\ee 
satisfies $  \tilde{A}_i (s_0) \in \delta H^{1}(B) $ and 
\be \label{localsumbdd1}
\delta^{-1} \tilde{A}_i (\cdot ,0) \in  H^{1}(B) +  N^{\sg-1} H^{\sg}(B) 
\ee

(2) 
Consider a smooth solution $ A_{t,x,s} $ of \eqref{dYMHF} on $ [t_0,t_1] \times B\times [0,s_0] $ in the temporal-caloric gauge $ A_0(t,x,0) = 0, A_s(t,x,s)=0 $ for which the curvature $ F_{\al \beta}(t_0) $ obeys \eqref{bddFx}. Apply Part (1) to $ A_{x}(t_0,\cdot) $ and use $ U=U(x) $ to make the transformation \eqref{Uchange} on $ [t_0,t_1] \times B$. Then in addition to \eqref{localsumbdd1} one also has 
\be \label{localsumbdd2}
\delta^{-1}  \partial_t \tilde{A}_i (t_0,\cdot ,0) \in  L^2(B) +  N^{\sg-1} H^{\sg-1}(B) 
\ee
\end{proposition}

\

\begin{remark} \label{UhlenbeckType}
This Proposition can be thought of as an Uhlenbeck-type lemma for fractional regularities below $ L^2 $. It essentially says that if $ \vn{I F}_{L^2(B)} < \delta $ then one can find gauge-equivalent $ \tilde{A} \sim A, \  \tilde{E} \sim E  $ for which $ \vn{I \tilde{A}}_{H^1(B)} \ls \delta $ and $  \vn{I \tilde{E}}_{L^2(B)} \ls \delta   $. \end{remark}

\

Of course, the multiplier $ I $ is not defined on a domain, so the remark above is simply a heuristic. However, once one has the bounds \eqref{localsumbdd1}, \eqref{localsumbdd2} it is possible to use an extension result (Proposition \ref{Extension}) to extend $ \tilde{A} $ to the whole space $ \mb{R}^3 $ such that $ I \tilde{A} \in \delta H^1, I \pt_t \tilde{A} \in \delta L^2 $ while still maintaing the Gauss equation \eqref{Constraint}. 

\

Once we restrict a solution to a ball $ B $ and perform the extension procedure, one obtains two heat flows: one for the original solution and one for the new one, which coincide at $ s=0 $ on $ B $. 
The need arises to compare them locally for parabolic times $ s>0 $. Due to infinite speed of parabolic propagation, they will not coincide, but one can obtain difference bounds of type 
$  \vn{ \chi [ F(s) - F'(s)] }_{ L^2(B)} \ll s^{M}  $. This is proved by Proposition \ref{PropErrorDif} to which Section \ref{SecErrorSM} is devoted. 
 
 \
 
Eventually, the result is reduced to the space-time bound \eqref{deltaEnergyReduced} which represents a trilinear estimate proved in Proposition \ref{TrilinearProp}. Building upon the decompositions and estimates from Section \ref{SecDecAndEst} (in particular the main Fourier bilinear component of $ w $ identified in \cite{oh2017yang}), the leading term of the trilinear integral is revealed in Section \ref{SecTrilinear} to consist of a null form interaction with favorable frequency weights arising from the "heat-wave commutator" $ w $, see \eqref{YMTrilcomp} and Remark \ref{ComparisonAC}.

\

In the Appendix we formulate the invariant modified energy for Maxwell-Klein-Gordon and show how it can be used to increase the range of regularities for which global well-posedness is known to hold, from $ \sg > \frac{\sqrt{3}}{2} $ in \cite{keel2011global} to $ \sg> \frac{5}{6} $.

\subsection{Reduction of Theorem \ref{MainThm} to the almost conservation law Prop. \ref{AClaw} }  \label{RedMainThm} \

Consider initial data sets of size $ \vn{ (\bar{A}_{x}, \bar{E}_{x})}_{ H^{\sg} \times H^{\sg-1}} \leq M $. Fix an arbitrary time $ T>0 $. By the well-posedness statement in Corollary \ref{CorLWP}, it suffices to prove 
\be \label{CMTbound}
\vn{A[T]}_{H^{\sg} \times H^{\sg-1}} \leq C_{M,T} < \infty  
\ee
for all such solutions which exist on $ [0,T] $. Applying Corollary \ref{CorLWP} again
together with the approximation statement in Proposition \ref{ApproxID} implies that it suffices to show \eqref{CMTbound} for regular solutions $ A \in C^{\infty}_t H^{\infty}_x (\mb{R} \times \mb{R}^3 ) $, which are known to be global. 

We will choose a large number $ N=N(M,T) $ to be specified later, with $  T \ll_M N^{0+} $  as usual with the I-method.   
Rescale the solution by $ A^{\lmd}(t,x) \defeq \lmd^{-1} A(t/\lmd, x/\lmd) $ 
where we choose $ \lmd, N $ 
such that the rescaled initial data satisfies 
\be \label{InDataScale}
\vn{ ( \bar{A}^{\lmd},  \bar{E}^{\lmd}) }_{   \dot{H}^{\sg} \times \dot{H}^{\sg-1}} \ls \lmd^{ \frac{1}{2}- \sg  } M  \ll  \frac{\eta}{ N^{1-\sg}}, \qquad \vn{\bar{A}^{\lmd}}_{  \dot{H}^{\frac{1}{2}+} } \ls \lmd^{0-} M \ll 1,
\ee
so $ \lmd, N $ are related by 
\be \label{lmdNrel}
\lmd = N^{\frac{1-\sg}{\sg-\frac{1}{2}}} \lpr  \frac{C M}{\eta} \rpr^{\frac{1}{\sg-\frac{1}{2}}} 
\ee
and $ \eta \ll 1 $ is a constant chosen small enough to overpower a number of universal constants in different estimates. Now we need to obtain a bound on $  A^{\lmd}[\lmd T] $. 
 
 We extend $ A^{\lmd}_{t,x}(t,x) $ to $  A^{\lmd}_{t,x,s}(t,x,s) $ by the heat flow using Theorem \ref{ThmRegSol}. We obtain a regular solution $ A^{\lmd} \in C^{\infty}_{t,s}  H^{\infty}_x ( [0,\lmd T] \times \mb{R}^3 \times [0,1]) $ to \eqref{dYMHF} in the caloric gauge $ A^{\lmd}_s = 0 $, i.e. 
$$ \pt_s A_{\al}^{\lmd} = D^{\ell} F_{\ell \al}^{\lmd}, \qquad A_{\al}^{\lmd}(t,x,s=0)= A_{\al}^{\lmd}(t,x) $$
Let $ s_0 = N^{-2} $. Denote the energies \eqref{EnergiesS} determined by $ A_{\al}^{\lmd} $ at $ (t,s) $ by $ \mathcal{E}(t,s) $ and consider the modified energy \eqref{ModifiedEnergy}. Begin with the initial data:


By Corollary \ref{CorInitData} with $ \ep \ll  \eta / N^{1-\sg} $  and \eqref{InDataScale} we obtain $ \mathcal{IE}(0) \leq \eta^2 $.


From this and $ O(\lmd T) $ application of Proposition \ref{AClaw} we obtain $ \mathcal{IE}(t) \leq 2 \eta^2 $ for all $ t \in [0,\lmd T ] $ and $ \vn{ A^{\lmd}[\lmd T]}_{H^{\sg} \times H^{\sg-1}} < \infty $ provided that 
\be \label{CondTN}
\lmd T \ll N^{\frac{1}{2}-}, \quad \text{i.e.} \quad T \ll_M N^{\frac{1}{2}- \frac{1-\sg}{\sg-1/2}-} = N^{0+}.
\ee
It is seen that it is possible to choose $ N, \lmd $ satisfying \eqref{InDataScale}, \eqref{lmdNrel}, \eqref{CondTN} if $ \sg > \frac{5}{6} $. Undoing the scaling we obtain \eqref{CMTbound}.

\

\subsection{Proof of the almost conservation Proposition \ref{AClaw}} 

\

By a continuity argument one can assume that $ \mathcal{IE}(t) \leq 3 \eta^2 $
holds for all $ t \in [t_0,t_1] $. 

Without loss of generality we prove \eqref{deltaEnergy} for $ t=t_1 $.

By Theorem \ref{ThmRegSol} we may assume the caloric gauge condition $ A_s(t,x,s) = 0 $.

\

We note that by using Bianchi's identity \eqref{Bianchi} and the Leibniz rule \eqref{Lrule} one can write the differentiated energy at $ s $ as 
$$  \label{DiffEner}
\frac{d}{dt} \mathcal{E}(t,s)=\sum_{\ell=1}^3 \int_{\mathbb{R}^{3}}  (w_\ell , F_{0\ell} )(s)  \dd x
$$
where $ w_\ell $ is given by \eqref{wDef} and therefore 
\be \label{DeltaE}
 \mathcal{E}(t_1,s) -  \mathcal{E}(t_0,s) = \sum_{\ell=1}^3  \int_{t_0}^{t_1} \int_{\mathbb{R}^{3}}  (w_\ell(s), F_{0\ell}(s) )  \dd x \dd t
\ee
The strategy we use consists of localizing this integral and estimating it in appropriate local gauges, crucially using the fact that \eqref{DeltaE} is gauge-invariant. 

\

Let $ ( B_j^{(1)})_{j \in \mathcal{J}} $ be a finitely overlapping covering of $ \mathbb{R}^3 $ with balls of radius $ 1 $ centered at $ (y_j)_{j \in \mathcal{J}} $.  Let $ (\chi_j)_{j \in \mathcal{J}} $ be a spatial partition of unity with $ \supp \chi_j \subset  B_j^{(1)} $. 

Let $ ( B_j)_{j \in \mathcal{J}} $ consist of balls with same centers $ (y_j)_{j \in \mathcal{J}} $ but radius $ 2 $, which are also finitely overlapping. Let $ ( \mathcal{D}_j)_{j \in \mathcal{J}} $ be domains of dependency with bases $ ( B_j)_{j \in \mathcal{J}} $, which are truncated cones 
$$
 \mathcal{D}_j = \{ (t,x) \ | \ t-t_0+\vm{x-y_j} < 2, \ t \in [t_0,t_1] \}
$$
which form a covering of $  [t_0,t_1] \times \mb{R}^3 $. Note that $ [t_0,t_1] \times B_j^{(1)} \subset  \mathcal{D}_j $. 

\

By Proposition \ref{PropCoeff} we define square summable coefficients $ (c_j)_{j \in \mathcal{J}} $ associated to  $ ( B_j)_{j \in \mathcal{J}} $ such that
\eqref{BallHighCovDer1}-\eqref{BallHighCovDer4} and \eqref{sqsummcj} hold. 

\

For each $ j \in \mathcal{J} $ we invoke Propositions \ref{LocalGauge1} on $ B_j $ with $ \delta = c_j $ which provide spatial $ U^{(j)} : B_j \to G $ and corresponding  transformations on $ [t_0,t_1] \times B_j \times [0,s_0] $: 
\be \label{gaugeTransf}
A_{\al}^{(j)} \defeq U^{(j)} A_{\al} U^{(j)-1} - \partial_{\al} U^{(j)} U^{(j)-1}, \qquad F_{\al \beta}^{(j)} = U^{(j)} F_{\al \beta} U^{(j)-1} 
\ee
with  
\be \label{tildeAs0}
\vn{ A^{(j)}(t_0,\cdot,s_0)}_{H^{1}(B_j)} \ls c_j 
\ee 
 and 
$$
\vn{ A^{(j)}(t_0,\cdot,0) }_{H^1(B_j) +  N^{\sg-1} H^{\sg}(B_j)}   + \vn{ \pt_t  A^{(j)}(t_0,\cdot,0) }_{L^2(B_j) +  N^{\sg-1} H^{\sg-1}(B_j) } \ls c_j  
$$ 

Since the $ U^{(j)}(x) $ are independent of $ t $ and $ s $, the $ A_i^{(j)} $'s remain smooth solutions of \eqref{YM} and \eqref{dYMHF} on $ [t_0,t_1] \times B_j \times [0,s_0] $ in the temporal and caloric gauges
$ A_0^{(j)}(t,x,0) = 0, \ A_s^{(j)}(t,x,s) = 0$. 
Due to gauge-invariance the bounds \eqref{BallHighCovDer1}-\eqref{BallHighCovDer4}
hold for $ F_{\al \beta}^{(j)}, w^{(j)} $ too (replacing $ c_j $ by $ C c_j $ if needed) and, moreover
\be  \label{localzwF}
 \chi_j (w_\ell(s), F_{0\ell}(s) ) =  \chi_j ( w^{(j)}_\ell(s), F^{(j)}_{0\ell}(s) ).  
\ee 

\

We now use Proposition \ref{Extension} to extend $ (  A^{(j)}(t_0,\cdot,0), \pt_t  A^{(j)}(t_0,\cdot,0) ) $ from $ B_j $ to a regular $ ( \bar{a}^{(j)},\bar{e}^{(j)} ) $ on $ \mb{R}^3 $
in such a way that the constrain equation \eqref{Constraint}
holds and we have the bound (see Lemma \ref{IEquivSum})
$$
\vn{(I \bar{a}^{(j)}, I \bar{e}^{(j)} ) }_{H^1 \times L^2} \simeq  \vn{\bar{a}^{(j)}}_{ H^1 +  N^{\sg-1} H^{\sg}}  + \vn{\bar{e}^{(j)}}_{L^2 +  N^{\sg-1} H^{\sg-1}} \ls c_j 
$$ 

\

Let $ \bar{A}^{(j)} $ be the regular Yang-Mills solution in the temporal gauge on $ [t_0,t_1] \times \mb{R}^3 $ with initial data $ (  \bar{A}^{(j)}(t_0), \pt_t \bar{A}^{(j)}(t_0) )= (\bar{a}^{(j)},\bar{e}^{(j)} ) $. By Proposition \ref{ModLWP} Part (1):
\be \label{ISobbdatt2}
\vn{I  \bar{A}^{(j)}[t]}_{L^{\infty}_t (H^1 \times L^2)(  [t_0,t_1]\times \mb{R}^3 ) } \ls c_j 
\ee

Due to finite speed of propagation we have that 
\be \label{FiniteSpeed}
A^{(j)}(\cdot,\cdot,0)_{\mkern 1mu \vrule height 2ex\mkern2mu \mathcal{D}_j}  =  \bar{A}^{(j)}_{\mkern 1mu \vrule height 2ex\mkern2mu \mathcal{D}_j} 
\ee 
From this and \eqref{ISobbdatt2} one has 
\be  \label{Atildt1B1}
\vn{ A^{(j)}[t_1]_{\mkern 1mu \vrule height 2ex\mkern2mu s=0 }}_{ (H^1 +  N^{\sg-1} H^{\sg})(B_j^{(1)}) \times (L^2 +  N^{\sg-1} H^{\sg-1})(B_j^{(1)})}  \ls c_j 
\ee

\

We now extend $  \bar{A}^{(j)} $ from $[t_0,t_1] \times \mb{R}^3 $ to $[t_0,t_1] \times \mb{R}^3 \times [0,s_0] $ as the solution of the \eqref{dYMHF} in the caloric gauge $ \bar{A}^{(j)}_s=0 $ given by Theorem \ref{ThmRegSol}. By Proposition \ref{ChangeCaloricDeTurck} and Lemma \ref{IEquivSum}
\be \label{barAjbddts}
\vn{ \bar{A}^{(j)}_i }_{L^{\infty}_t L^{\infty}_s (H^1 +  N^{\sg-1} H^{\sg}) }   \ls c_j 
\ee
From Corollary \ref{CorHFModEn} we control the modified energy of $  \bar{A}^{(j)} $  as  $  \mathcal{IE}[\bar{F}_{\al \beta}^{(j)} ](t) \ls  c_j^2  $ for all $  t \in [t_0,t_1]  $. From Propositions \ref{CovBdsF1} and \ref{CovBdsw1} we control higher derivatives of $ \bar{F}^{(j)} $ and $ \bar{w}^{(j)} $, which restricted to the ball $ B_j $ provide \eqref{BallHighCovDer1}-\eqref{BallHighCovDer4} for $ \bar{F}^{(j)} $ and $ \bar{w}^{(j)} $ as well. These bounds together with \eqref{barAjbddts} and \eqref{FiniteSpeed} allows us to apply Proposition \ref{PropErrorDif} to conclude, for all $ s \in [0,s_0] $, that 
\begin{align}
\label{FDiffA}
(N s^{\frac{1}{2}})^{9(1-\sigma) } \vn{ \chi_j [ F^{(j)} (t,s) - \bar{F}^{(j)}(t,s)] }_{L^{\infty}_t L^2_x} & \ls s^{2} c_j \\
\label{wDiffA}
(N s^{\frac{1}{2}})^{8(1-\sigma) } \vn{ \chi_j [ w_{\ell}^{(j)}(t,s) -  \bar{w}_{\ell}^{(j)}(t,s)  ] }_{L^{\infty}_t L^2_x} & \ls s^{\frac{5}{4}} c_j
\end{align}

\

By \eqref{DeltaE} and \eqref{localzwF}
\be \label{deltaEnergyPartition}
 \mathcal{E}(t_1,s) -  \mathcal{E}(t_0,s) =\sum_{j \in \mathcal{J}}  \int_{t_0}^{t_1} \int_{\mathbb{R}^{3}} \chi_j(x)  ( w^{(j) \ell}(s), F^{(j)}_{0\ell}(s) )   \dd x \dd t
\ee 
Using \eqref{FDiffA}, \eqref{wDiffA} and \eqref{BallHighCovDer1}, \eqref{BallHighCovDer4} with $ m= 0 $ for $ \bar{F}^{(j)} $ and $ w^{(j)} $ to bound 
$$
\chi_j   ( w^{(j)}_{ \ell}, F^{(j)}_{0\ell}- \bar{F}^{(j)}_{0\ell})(s)  \qquad \text{and} \qquad \chi_j ( w^{(j)}_{ \ell} - \bar{w}^{(j)}_{ \ell}, \bar{F}^{(j)}_{0\ell})(s) 
$$
in $ L^1_t L^1_x $, together with \eqref{sqsummcj} allows us to reduce \eqref{deltaEnergy} to summing
\be \label{deltaEnergyReduced}
(N^2 s)^{1-\sigma} \vm{ \int_{t_0}^{t_1} \int_{\mathbb{R}^{3}} \chi_j(x)  ( \bar{w}^{(j) \ell}(s), \bar{F}^{(j)}_{0\ell}(s) )   \dd x \dd t
 }   \ls  \frac{s^{\frac{1}{4}-} }{  (N s^{\frac{1}{2}})^{(1-\sg)}}  c_j^2  
\ee

\

We now use Proposition \ref{ModLWP} Part (2) to place $  \bar{A}^{(j)}(t_0,\cdot,0) $ in the Coulomb gauge (only at $ t=t_0 $), as this results in the improved regularity from \eqref{curlfree} for the curl-free part. We make this transformation on $ [t_0,t_1] \times \mb{R}^3 \times [0,s_0] $ using a spatial $ \bar{U}^{(j)}(x) $, so we maintain a regular caloric solution which at $ s=0 $ is temporal, but now at $ s=0 $ satisfying \eqref{divfree}, \eqref{curlfree} with $ \ep=c_j $. Finally we make another change of gauge \eqref{UchangeCaloricDeTurck2} to deTurck's gauge using Proposition \ref{ChangeCaloricDeTurck} and Remark \ref{RkChangeCaloricDeTurck} which leaves the solution invariant at $ s=0 $. Due to gauge invariance \eqref{deltaEnergyReduced} now follows from Proposition \ref{TrilinearProp}. 


\subsubsection{Proof of the Sobolev bound \eqref{SobolevBound}} \

We apply Proposition \ref{ImprovedWPwAC} to obtain $  \vn{A_i (t_0, \cdot, s_0) }_{H^{\sg}} \leq C_0(M_0) $, as we have \eqref{bddFxN} from Proposition \ref{CovBdsF1}. Without loss of generality we prove \eqref{SobolevBound} for $ t=t_1 $.

Applying Lemma \ref{UtransfBound} for \eqref{gaugeTransf} at $ s=s_0, \ t=t_0$, using \eqref{tildeAs0} we obtain
\be \label{UjXbds}
 \vn{\nabla U^{(j)}}_{ H^{\sg}(B_j)} \ls d_j (1+ C_0(M_0) );  \qquad U^{(j)}, U^{(j)-1} \in C_1(M_0) X_{\sg}(B_j) 
\ee
where we denote 
$$ d_j = c_j +  \vn{A_i (t_0, \cdot, s_0) }_{H^{\sg}(B_j)}, \qquad \sum_j d_j^2 \ls C_0(M_0)^2 +1  $$ 
The last bound follows from Proposition \ref{PropsqSumWs2}. 
From \eqref{Atildt1B1} we have 
$$
\vn{  A^{(j)}(t_1)}_{ H^{\sg}(B_j^{(1)})} + \vn{ \pt_t A^{(j)}(t_1)}_{ H^{\sg-1}(B_j^{(1)})}  \ls c_j 
$$
Combining this with \eqref{UjXbds}, from \eqref{Algebraproducts2} and  \eqref{gaugeTransf} at $ s=0,\ t=t_1 $ we obtain 
$$
\vn{  A(t_1)}_{ H^{\sg}(B_j^{(1)})} + \vn{ \pt_t A(t_1)}_{ H^{\sg-1}(B_j^{(1)})}  \ls_{M_0} c_j + d_j. 
$$
Applying Proposition \ref{PropsqSumWs2} again we conclude $ \vn{A[t_1]}_{H^{\sg} \times H^{\sg-1}} \leq C(M_0) < \infty  $.  \qed

\

\subsection{Heuristic remarks and other works} \ 

\begin{remark}[Smoothing heuristics and connection to Littlewood-Paley theory] \label{RmkSmoothingLP}
The heat flows solutions to \eqref{cYMHF}, \eqref{dYMHF} satisfy the covariant parabolic equations \eqref{0CovPara}. Therefore, in the first approximation one can think of the nonlinear $ F_{\al \beta}(s) $ as roughly $ F_{\al \beta}(s) \approx e^{s \Delta} F_{\al \beta}(0) $. Expanding \eqref{cYMHF} it is seen that in general, the evolution of $ A_i(s) $ is only
degenerately parabolic. With an appropriate choice of gauge (De Turck $ A_s = \pt^{\ell} A_{\ell} $) one can make the $ A_i $ equations genuinely parabolic, specifically \eqref{DeTurck}. Then the approximation $ A_i (s)  \approx e^{s \Delta} A_i (0) $ is still expected. 

The Gaussian multipliers $ e^{s \Delta} $ make the bulk of the frequencies to become concentrated to $ \vm{\xi} \ls s^{-\frac{1}{2}} $. In connection to Littlewood-Paley projections this can be stated as $ e^{s \Delta} \approx P_{\leq k(s)}  $ and  $ \pt_x e^{s \Delta} \approx 2^{k(s)}P_{k(s)}  $ where $ 2^{k(s)} = s^{-\frac{1}{2}} $. Since $ F_{\al \beta} $ contains derivatives, in a certain gauge one heuristically has $  F(s) \approx P_{k(s)} F(0) $. 

Now to decompose $ A_i (0) $ one can set the caloric condition $ A_s =0$, turning \eqref{cYMHF}  into $ \pt_s A_i = \mathbf{D}^{\ell} F_{\ell i} $, which integrates to 
\be \label{CaloricAiRepr}
 A_i (s) =  A_i (s_0) + \int_{s}^{s_0} s \mathbf{D}^{\ell} F_{i \ell} \frac{\dd s}{s} 
\ee
Since $ \dd k \simeq  -\dd s / s $, $ 2^{2k} s= 1$  and taking $ s_0 = N^{-2} $ one can interpret \eqref{CaloricAiRepr} as a continuous Littlewood-Paley decomposition 
\be \label{CaloricAiRepr2}
A (0) \approx P_{\leq N} A (0) + \int_{
\log N}^{\infty} 2^{-k} P_k F(0) \dd k, \qquad P_k  A_i(0) \approx s \mathbf{D}^{\ell} F_{i \ell}(s) 
\ee
\end{remark}

\

\begin{remark}[Comparison of $ \mathcal{IE}(t) $ with the classical I-method modified energy]
\label{RkComparisonModEn}
The I-method suggests applying the Fourier multiplier $ I $  inside the energy and trying to prove almost conservation. 
This is explained in detail in \cite{keel2011global} for the closely related MKG and in \cite{roy2008global} for the simpler cubic wave equation. 

The $ I $ operator is defined in \eqref{Ioperator}. Roughly speaking 
\be \label{Imthapprx1}
 {\bf E}[F[IA]](t) \approx {\bf E}[I F](t) \approx \vn{P_{<N} F}_{L^2}^2 +  N^{2(1-\sg)} \sum_{2^k \geq N}  
 \vn{\vm{D}^{\sg-1}  P_k F}_{L^2}^2 
 \ee

Recall the definition of $ \mathcal{IE}(t) $ and $ \mathcal{E}(t,s) $ in \eqref{ModifiedEnergy}, \eqref{EnergiesS}. In light of Remark \ref{RmkSmoothingLP},
\be \label{Imthapprx2}
\mathcal{E}(t,s_0) \approx \vn{P_{<N} F}_{L^2}^2, \quad \int_0^{s_0} (N^2 s)^{1-\sigma}  \mathcal{E}(t,s) \frac{\dd s}{s} \approx \sum_{2^k \geq N} \Big( \frac{N}{2^k} \Big)^{2(1-\sg)} \vn{ P_k F}_{L^2}^2 
\ee
where we denote $ F=F_{\al \beta}(0) $. Since $ s^{\frac{1}{2}} \approx 2^{-k} \approx \vm{D}^{-1} $ the similarity between \eqref{Imthapprx1} and \eqref{Imthapprx2} is clear, suggesting $ \mathcal{IE}(t) $ could serve as a substitute to $  {\bf E}[F[IA]](t) $. 
\end{remark}

\

\begin{remark}[Comparison of almost conservation laws] \label{ComparisonAC}
We have claimed above in III that $ \mathcal{IE}(t) $ enjoys better almost conservation properties than $ {\bf E}[F[IA]](t) $ or $ H[I \Phi](t) $ - the latter being the Maxwell-Klein-Gordon \eqref{MKG} classical I-method modified energy from \cite{keel2011global} ($ \Phi=(A_x, \phi) $),  while the former would be the YM I-method modified energy in the Coulomb gauge (if it could be implemented for large data). We consider $ {\bf E}[F[IA]](t) $ and $ H[I \Phi](t) $ as essentially the same, since we can ignore the $ [A_i, A_j] $ commutators because they have better regularity. 

The dominant part of $ H[I \Phi](t_1) - H[I \Phi](t_0) $ is proved in \cite[(106),(124)]{keel2011global} to be 
\be \label{MKGTril}
\sum_{k,k_i} \int_{t_0}^{t_1} \int_{\mb{R}^3} \big( \pt_t I \Phi_k, I {\bf N} (\Phi_{k_1}, \Phi_{k_2}) - {\bf N}(I \Phi_{k_1}, I \Phi_{k_2}) \big) \dd x \dd t 
\ee
where $ {\bf N} $ is a null form defined in \eqref{DefNullForm}, \eqref{Qij}. Consider the more difficult case $ N<2^{k_1} < 2^{k_2} \simeq 2^k $, in which the commutator in \eqref{MKGTril} does not provide any cancelation and can be effectively written as 
\be \label{MKGTril2}
 \frac{1}{2^{k_{\min}+k_{\max}}} \lpr \frac{N}{2^k} \rpr^{2(1-\sg)}   \int_{t_0}^{t_1} \int_{\mb{R}^3} \big( \pt_t  \Phi_k, {\bf N} ( \nabla_{x,t} \Phi_{k_1}, \nabla_{x,t} \Phi_{k_2}) \big) \dd x \dd t 
\ee

On the other hand, $ \mathcal{E}(t_1,s) -  \mathcal{E}(t_0,s) $ can be written as \eqref{DeltaE}, of which \eqref{3dfNullTerm} is the leading part. The similarity to \eqref{MKGTril} begins by noting that $ w $ is also an essentially bilinear term with some cancelations, representing a "heat-wave commutator". 

 In \eqref{SpecW}, \eqref{SpecNull} it is revealed that the main contribution to $  \mathcal{IE}(t_1) -  \mathcal{IE}(t_0) $ comes roughly from terms 
\be \label{YMTrilcomp}
(N^2 s)^{1-\sg} \sum_{k,k_i} \int_{t_0}^{t_1} \int_{\mb{R}^3}  \big( \pt_t A_k^{df} , {\bf N} (  \pt_t A^{df}_{k_1}, \pt_t A^{df}_{k_2} ) \big)     
\dd x \dd t \times weight 
\ee
The weight in \eqref{YMTrilcomp} is from \eqref{SpecNull} and in the same case $ N<2^{k_1} < 2^{k_2} \simeq 2^k $ it becomes concentrated at $ 2^k \simeq s^{-\frac{1}{2}} $ turning \eqref{YMTrilcomp} into (with the notation $ \Phi = A^{df} $ )
\be \label{YMTrilcomp2}
 \frac{1}{2^{2k_{\max}}} \lpr \frac{N}{2^k} \rpr^{2(1-\sg)}   \int_{t_0}^{t_1} \int_{\mb{R}^3} \big( \pt_t  \Phi_k, {\bf N} ( \nabla_{x,t} \Phi_{k_1}, \nabla_{x,t} \Phi_{k_2}) \big) \dd x \dd t 
\ee
Comparing \eqref{YMTrilcomp2} with \eqref{MKGTril}-\eqref{MKGTril2} shows that the structure of  $  \mathcal{IE}(t_1) -  \mathcal{IE}(t_0) $ effectively shifts a derivative from a high frequency to low frequency, compared to  $ H[I \Phi](t_1) - H[I \Phi](t_0) $. This translates into a bound of $ O(N^{-\frac{1}{2}}) $ in \eqref{deltaEnergy} compared to $ O(N^{\frac{1}{2}-\sg}) $ in \cite[(45)]{keel2011global}, leading to a result for lower regularities $ \sg > \frac{5}{6} $ rather than $ \sg > \frac{\sqrt{3}}{2} $. 
\end{remark} 

\

\subsubsection{Other works}

Local well-posedness in the Lorenz gauge is given in \cite{selberg2016null}, \cite{tesfahun2015local}. Global existence for critical power Yang-Mills-Higgs equations at the energy regularity is proved in \cite{keel1996global}. Global existence of solutions on globally hyperbolic Lorentzian manifolds, under a higher regularity assumption, is considered in \cite{chrusciel1997global} following the classical approach of \cite{eardley1982global}.

We now mention notable results in $ 4+1 $-dimensions, which is the energy-critical case, beginning with \cite{klainerman1999optimal}. The global result for small energy was established in \cite{krieger2017global}  building upon the high-dimensional case in \cite{krieger2013global}. The large data theory culminated with the proof of the Threshold Conjecture in \cite{oh2021threshold}.


\section{Preliminaries} \label{SecPreliminaries}

\subsubsection{Space-time}
We will work on the Minkowski space $ \mb{R}^{1+3} $ equipped with the Minkowski metric, according to which indices will be raised and lowered. We adopt the Einstein summation convention. Greek letters run over the space-time variables, while latin indices are used only for spatial variables. 

We use $ A[t] $ as a short-hand for the pair $ (A(t), \pt_t A(t) ) $.

\subsubsection{Lie theory} \label{LieDefinitions}

Let $ G $ be a Lie group and $\mathfrak{g}$ its associated Lie algebra. We assume the existence of a bi-invariant inner product $ ( \cdot, \cdot) : \mathfrak{g} \times \mathfrak{g} \to [0,\infty) $. This means 
\be \label{InvMetric} 
([A,B],C)=(A, [B,C]) \quad \text{or equivalently}, \quad (U A U^{-1}, U B U^{-1} )=(A,B)  
\ee
for all $ A, B, C \in \mathfrak{g}$, $ U \in G $. The Leibniz rule
\be \label{Lrule}
\partial_{\al}(B, C)=\left(\mathbf{D}_{\al} B, C\right)+\left(B, \mathbf{D}_{\al} C\right)
\ee 
holds, as well as Bianchi's identity 
\be \label{Bianchi}
\mathbf{D}_{\al} F_{\beta \gamma} + \mathbf{D}_{\beta} F_{ \gamma \al} + \mathbf{D}_{\gamma} F_{  \al \beta} = 0. 
\ee
Covariant derivatives are commuted using 
\be \label{CDComm}
\mathbf{D}_{\al} \mathbf{D}_{\beta}  - \mathbf{D}_{\beta} \mathbf{D}_{\al} = [ F_{  \al \beta}, \cdot ] 
\ee 

The covariant laplacian is denoted by $ \Delta_A \defeq \mathbf{D}^{\ell} \mathbf{D}_{\ell} $. 

Due to \eqref{InvMetric}, the inner-product can be used to define the $ L^p $ norms of $\mathfrak{g}$ -valued functions in a gauge-invariant way. 

For concreteness, it is common to take $ G $ to be a matrix group such as $ SO(n, \mb{R}) $ or $ SU(n, \mb{C}) $, with the associated Lie Algebras $ \mathfrak{so}(n, \mb{R}) $ and $ \mathfrak{su}(n) $. In these cases the Lie bracket is $ [A, B]= AB-BA $ given by matrix multiplication, while the bi-invariant scalar products are $ (A,B)= \tr( A B^{T}) $, respectively $ (A,B)=\tr( A B^{*}) $.

\

\subsubsection{Regular functions} 

A function $ f $ on $ \mb{R}^3 $ (such as an initial data $ \bar{A}_i $) is called \emph{regular} if $ f \in H^{\infty}(\mb{R}^3 ) \defeq \bigcap_{n \geq 0} H^n(\mb{R}^3 ) $. Solutions $ A_{\al}(t,x) $ or  $ A_i(x,s) $ defined on intervals are called regular if $ A_{\al} \in C^{\infty}_t  H^{\infty}_x (I \times \mb{R}^3) $ , respectively  $ A_i \in C^{\infty}_s  H^{\infty}_x ( \mb{R}^3 \times J) $. Similarly for 
$ A_{\al} \in C^{\infty}_{t,s}  H^{\infty}_x ( I \times \mb{R}^3 \times J) $. 
Gauge transformations are called regular if $ \nabla U, \nabla U^{-1}  \in H^{\infty} ( \mb{R}^3) $. 

\subsubsection{Boundedness notations}

We use $ A \ls B $ to denote $ A \leq C B $ and $ A \ll B $ to denote $ A C \leq B $, where $ C $ is a large constant. We use $ A+ $ or $ A- $ to denote $ A+\ep $ and $ A-\ep $ where $ 0 < \ep \ll 1 $ is a large number which can depend on $ \sg>0 $. We denote $ \jb{x}^2 = 1+ \vm{x}^2 $. 

For injective operators $ T $ on $ X $ we denote by $ TX $ the image space with norm $ \vn{T^{-1} u}_X $. The embedding $ X \subseteq Y $ as an estimate denotes $ \vn{u}_Y \ls \vn{u}_X $. The mapping $ T: X \to Y $ as an estimate means $ \vn{Tu}_Y \ls \vn{u}_X $.
We say $ u \in \delta X $ if $ \vn{u}_X \ls \delta $. Product estimates $ X \times Y \to Z $ denote $ \vn{uv}_Z \ls \vn{u}_X \vn{v}_Y $. The space of sums $ X+Y $ is endowed with the norm $ \vn{u}_{X+Y} \defeq \inf \{ \vn{u_1}_X + \vn{u_2}_Y \ | \ u=u_1+u_2   \} $.

\subsubsection{Fourier operators}

The Fourier transform of $ f $ is denoted by $ \hat{f} $ or $ \mathcal{F} f $. Littlewood-Paley projections are denoted by $ P_k $. Whether the frequencies $ \vm{\xi} \ls 1 $ are included in $ P_0 $ or not will be clear from the context. Moreover, $ \tilde{P}_k $ denotes a similar multiplier such that $  \tilde{P}_k P_k = P_k  $. 

Linear and bilinear operators which are invariant to spatial translations are understood either through their spatial kernels, or as multipliers through their Fourier symbols. For instance: the smoothing $ I $ operator is defined by \eqref{Ioperator}, fractional derivative operators $ \vm{D}^{p} $, $ \jb{D}^{p} $ are defined through their symbols $ \vm{\xi}^p $, $ \jb{\xi}^{p} $.  

When the fields $ A, B $ are $\mathfrak{g}$-valued a bilinear operator with symbol $ m^{ij}(\xi_1,\xi_2) $ takes the form 
\be \label{BilMultOp}
\begin{aligned}
\mathcal{M}(A,B)(x) & = \int_{\mb{R}^3  \times \mb{R}^3} e^{i x \cdot (\xi_1+\xi_2)}  m^{ij}(\xi_1,\xi_2)  [ \hat{A}_i(\xi_1),  \hat{B}_j (\xi_2)]  \dd \xi_1 \dd \xi_2  \\ 
& =  \int_{\mb{R}^3  \times \mb{R}^3} K^{ij}(x-y_1, x-y_2) [A_i(y_1),  B_j (y_2)]  
\dd y_1 \dd y_2 
\end{aligned}
\ee
where $ m(\xi_1,\xi_2)= \hat{K} (\xi_1,\xi_2) $. 

A linear or bilinear multiplier operator is called disposable when its kernel is a function or measure with bounded mass. Minkowski's inequality insures that disposable operators
are bounded on translation-invariant normed spaces.

\subsubsection{Leray projections and null forms} \ 
Denote by $ {\bf P} $ and $ {\bf P}^{\perp} $ the projections on divergence free, respectively curl-free vector fields:
\be \label{Leray} 
{\bf P}_j A \defeq \Delta^{-1} \pt^{\ell} (\pt_{\ell} A_j - \pt_j A_{\ell}), \quad {\bf P}_j^{\perp} A \defeq (1-{\bf P}_j) A =  \Delta^{-1} \pt_j \text{div} A. 
\ee
Recall the classical identities (\cite{klainerman1994maxwell}, \cite{klainerman1995finite})
\be
\begin{aligned} \label{Nullformidentities}
{\bf P}_{j}\left(\phi \nabla_{x} \varphi\right)=\Delta^{-1} \nabla^{i} Q_{i j}\left(\phi, \varphi \right) \\
{\bf P} A^{i} \partial_{i} \phi=Q_{i j}\left(\Delta^{-1} \nabla^{i}  A^{j}, \phi\right)
\end{aligned}
\ee
which follow from the definitions by simple computations.

By $ {\bf N} (f,g) $ we denote a null form, meaning a linear combination of 
\be \label{DefNullForm}
\Delta^{-1} \nabla^i Q_{ij} (f, g), \quad \text{and} \quad  Q_{ij} (\Delta^{-1} \nabla^i  f, g) \quad \text{or} \quad 
\ee
where 
\be \label{Qij}
Q_{ij} (f, g) \defeq \pt_i f \pt_j g - \pt_j f \pt_i g 
\ee
In the context of $\mathfrak{g}$-valued fields \eqref{Qij} is replaced by \eqref{QijLie}. 

\subsubsection{The $ I$ multiplier}

Fix a number $ N \gg 1 $ and let $ m(\xi) $ and $ I $ be a smooth positive radial symbol and the associated Fourier multiplier operator such that 
\be \label{Ioperator}
 m(\xi) =
\begin{cases}
1, \qquad \qquad \vm{\xi} \leq N \\
\big( \frac{N}{\vm{\xi}} \big)^{1-\sg} , \quad \vm{\xi} \geq 2N
\end{cases}; 
\qquad \widehat{I f} (\xi) \defeq m(\xi) \hat{f} (\xi) 
\ee
The $ I $ operator smoothens the high frequencies and acts as the identity on relatively lower frequencies. It has an integrable kernel. It is straightforward to verify 

\begin{lemma} \label{IEquivSum}
One has  $ \quad \ I f \in H^1 \ \iff \  f \in H^1 + N^{\sg-1} H^{\sg} \quad $ and 
$$    \qquad \qquad \quad  I g \in L^2 \ \ \iff \ \ g \in L^2 + N^{\sg-1} H^{\sg-1}= L^2 + N^{\sg-1} \dot{H}^{\sg-1} $$
with equivalent norms.  
\end{lemma}

The advantage of these second formulations is that they can be directly localized to a domain.

\subsubsection{Covariant Gagliardo-Nirenberg inequalities and elliptic equations} 

We recall the following \emph{diamagnetic}, or \emph{Kato's}, \emph{inequality}. For a proof see, for example \cite[Lemma 4.2]{oh2015finite} or \cite[Lemma 3.9]{tao2008global}. One has
$$
\vm{ \pt_x \vm{ \varphi } } \leq \vm{\mathbf{D}_{x} \varphi }
$$
for regular $ \mathfrak{g} $-valued functions $ \varphi $ on $ \mb{R}^3 $. As a consequence one has the following covariant Gagliardo-Nirenberg and Sobolev inequalities on $ \mb{R}^3 $

\be  \label{SobGN}
\begin{aligned}
&\|\varphi\|_{L^{p}} \ls \|\varphi\|_{L^{2}}^{1-\al}   \left\|\mathbf{D}_{x} \varphi\right\|_{L^{2}}^{\al} , \qquad \al = 3(2^{-1}-p^{-1}) \\
&\|\varphi\|_{L^{3}} \ls \|\varphi\|_{L^{2}}^{\frac{1}{2}}   \left\|\mathbf{D}_{x} \varphi\right\|_{L^{2}}^{\frac{1}{2}}  \\
&\|\varphi\|_{L^{6}} \ls \left\|\mathbf{D}_{x} \varphi\right\|_{L^{2}} \\
&\|\varphi\|_{L^{\infty}} \ls \left\|\mathbf{D}_{x} \varphi\right\|_{L^{2}}^{\frac{1}{2}}  \left\|\mathbf{D}_{x}^{2} \varphi\right\|_{L^{2}}^{\frac{1}{2}} 
\end{aligned}
\ee

If $ \varphi $ is defined on a unit ball $ B $ instead, one has 

\be  \label{CovGNonB}
\vn{\varphi}_{L^p(B)} \ls \vn{\varphi}_{L^2(B)}^{1-\al} \vn{\mathbf{D}_{x} \varphi }_{L^2(B)}^{\al} + \vn{\varphi}_{L^2(B)}, \qquad \qquad \qquad \frac{1}{p} = \frac{\al}{6}  + \frac{1-\al}{2}.
\ee

\be  \label{CovGNonBinf}
\vn{\varphi}_{L^{\infty}(B)} \ls \big(  \vn{\mathbf{D}_{x} \varphi }_{L^2(B)} + \vn{\varphi}_{L^2(B)} \big)^{\frac{1}{2}} 
 \big(  \vn{\mathbf{D}_{x}^2 \varphi }_{L^2(B)} +  \vn{\mathbf{D}_{x} \varphi }_{L^2(B)}  \big)^{\frac{1}{2}}
+ \vn{\varphi}_{L^2(B)}
\ee 

\

By the argument in \cite[Theorem 4.1]{oh2017yang} one has 
\begin{proposition} \label{EllipticBdd}
Assume $ A \in \dot{H}^{\frac{1}{2}} (\mb{R}^3) $. Then the covariant elliptic equation $ \Delta_A B = F $ is solved by the $  \mathfrak{g} $-valued function $ B $ with the bound 
$$ \vn{B}_{\dot{H}^{\rho}} \ls  \vn{F}_{\dot{H}^{\rho-2 }} , \qquad \qquad \rho \in (0,2).  $$
\end{proposition}

\

\subsection{Sobolev spaces} \ 

The global Sobolev spaces $H^{s} = H^{s} (\mb{R}^n) $ and $\dot{H}^{s} = \dot{H}^{s}  (\mb{R}^n)  $  are defined using the Fourier transform, for any $s \in \mathbb{R}$, by 
\be \label{GlobalSobolev}
 \|u\|_{H^{s}} \defeq \left\|  \jb{\xi}^{s} \hat{u}\right\|_{L^{2}(\mb{R}^n)} , \quad \quad 
\|u\|_{\dot{H}^{s}} \defeq \left\||\xi|^{s} \hat{u}\right\|_{L^{2}(\mb{R}^n)} 
\ee
We denote $ H^{\infty}(\mb{R}^n)  \defeq \bigcap_{m \geq 0} H^m (\mb{R}^n) $. For $ s>0 $ one has $ H^s = L^2 \cap \dot{H}^{s} $ while for $ s<0 $ one has $ H^s = L^2 + \dot{H}^{s}$. By duality one may identify $ H^{-s} $ with $ (H^s)^* $.

\

We next define Sobolev spaces $ W^{s,2}(D) $ (also denoted $ H^s(D) $) when $ D $ is either a ball $  B \subset \mb{R}^n$ of radius $ \simeq 1 $ or  $  D= \mb{R}^n $. When $ s \in \{ 0 ,1 \} $ we use the classical definition. When $ s \in (0,1) $ we use the norm 
\be \label{LocalSobolev}
\vn{u}_{ W^{s,2}(D) }^2 \defeq \vn{u}_{L^2(D)}^2 + \vm{u}_{   \dot{W}^{s,2}(D) }^2
\ee
where $ \vm{\cdot}_{   \dot{W}^{s,2}(D) } $ denotes the Gagliardo seminorm defined by 
\be \label{Gagliardoseminorm}
\vm{u}_{   \dot{W}^{s,2}(D) }^2 \defeq \int_D \int_D \frac{\vm{u(x)-u(y)}^2}{\vm{x-y}^{n+2s}}    \dd x \dd y 
\ee 
Due to \eqref{HsCoincide} we will use the notations $ H^s(D) $ and $ W^{s,2}(D) $ (resp. $ \dot{H}^s(D) $ and $ \dot{W}^{s,2}(D) $) interchangeably for both $ D=B $ and $ D = \mb{R}^n $. 

We denote $ W^{s,2}_0(B) $ the closure of $ C_c^{\infty}(B) $ in the norm $ \vm{\cdot}_{ W^{s,2}(B)}$. Of course $ W^{s,2}_0( \mb{R}^n )= W^{s,2}( \mb{R}^n ) $. 

When $  s \in (-1,0) $ we define $ W^{s,2}(D) $, as a subspace of distributions, by duality  
$$
W^{s,2}(D)  \defeq W^{-s,2}_0(D)^{*}. 
$$

\

We recall some properties from \cite{di2012hitchhiker?s}. For $ D=\mb{R}^n $ the two definitions \eqref{GlobalSobolev}, \eqref{LocalSobolev} coincide 
\be \label{HsCoincide}
 \|u\|_{H^{s}} \simeq \vn{u}_{ W^{s,2}(\mb{R}^n ) }  ,  \quad \quad  \|u\|_{\dot{H}^{s}} \simeq \vm{u}_{   \dot{W}^{s,2}(\mb{R}^n) } 
 \ee

If $ B $ is a ball, then any $ u \in W^{s,2}(B) $ has an extension $ \tilde{u} \in W^{s,2}(\mb{R}^n ) $ with comparable norm 
$$   \vn{u}_{W^{s,2}(B)}   \simeq \vn{\tilde{u}}_{W^{s,2}(\mb{R}^n )} \simeq \inf \{ \vn{\bar{u}}_{W^{s,2}(\mb{R}^n )} \ | \ \  \bar{u} \restriction_{B} = u \}. $$
If $ \psi \in C^{0,1}(B) $, $ 0 \leq \psi \leq 1 $ then 
\be  \label{cutoffWs2}
\vn{\psi u}_{ W^{s,2}(B) } \ls \vn{u}_{ W^{s,2}(B) }, \qquad s \in (0,1). 
\ee 

\subsubsection{Sobolev multiplication laws} \

We recall from \cite{d2012atlas}, \cite{tao2001multilinear} that:
if  $  s_1+s_2 > s+ \frac{3}{2} $,  $ s_1+s_2 \geq 0 $ and $ s \leq \min(s_1,s_2) $ then one has 
\be \label{SobMult}
\| f g \|_{ H^s} \lesssim\|f \|_{H_{x}^{s_1}}\|g\|_{H_{x}^{s_2}}.
\ee
In a similar spirit, one has the homogeneous version \cite[ Lemma 3.2]{oh2014gauge} 
\be \label{SobMultHomog}
\| f g \|_{ \dot{H}^s} \lesssim \|f \|_{\dot{H}_{x}^{s_1}}\|g\|_{\dot{H}_{x}^{s_2}}.
\ee
provided $  s_1+s_2 = s+ \frac{3}{2} $ and $ -s,s_1,s_2 <  \frac{3}{2} $. 

For $ r>3/4 $ we will use the version \cite[Lemma 2.1]{keel2011global}:
\be \label{Sbvone}
\| f g\|_{\dot{H}_{x}^{-1}} \lesssim\|f\|_{H_{x}^{r}}\|g\|_{H_{x}^{r-1}}.
\ee
One also has
\be \label{SobMI}
\vn{I(fg)}_{\dot{H}^{-1}} \ls \vn{If}_{H^1} \vn{I g}_{L^2} 
\ee 
which follows immediately from \eqref{Sbvone} with $ r =\sg $

\subsubsection{Norms for gauge transformations}

We now define some  further norms. Let $ X_{\sg} = \dot{H}^1 \cap \dot{H}^{1+\sg} \cap L^{\infty} (\mb{R}^3) $, and the similar definition on a ball, i.e. 
$$
\vn{U}_{X_{\sg}} \defeq \vn{\nabla U}_{H^{\sg} (\mb{R}^3)} + \vn{U}_{L^{\infty} (\mb{R}^3) }, \quad \vn{U}_{X_{\sg}(B)} \defeq \vn{\nabla U}_{H^{\sg}(B)} + \vn{U}_{L^{\infty}(B) }.
$$
Similarly one defines 
\be  \label{XnormDef}
\vn{V}_X \defeq \vn{I \nabla V}_{H^1 (\mb{R}^3)}+ \vn{V}_{L^{\infty} (\mb{R}^3) }
\ee
One has the following algebra and product estimates
\begin{align}
\label{Algebraproducts1}
& X_{\sg} \times X_{\sg} \to X_{\sg}, \qquad & H^{\sg}  \times X_{\sg} \to H^{\sg}, \qquad &H^{\sg-1}\times X_{\sg} \to H^{\sg-1}
\\
\label{Algebraproducts}
& X  \times X \to X, \qquad & I^{-1} H^1  \times X \to I^{-1} H^1, \qquad &I^{-1} L^2 \times X \to I^{-1} L^2
\end{align}
These are proved using the Littlewood-Paley trichotomy, \eqref{Algebraproducts1} is \cite[Eq. (10)]{tao2003local}. 
By an extension argument, \eqref{Algebraproducts1} holds on domains as well 
\be
\label{Algebraproducts2}
\vn{UV}_{X_{\sg}(B)} \ls \vn{U}_{X_{\sg}(B)} \vn{V}_{X_{\sg}(B)}, \quad \vn{f U}_{ H^{r}(B)} \ls \vn{f}_{ H^{r}(B)} \vn{U}_{ X_{\sg}(B)}
\ee
for $ r \in \{ \sg, \sg-1 \} $. 

\

One may obtain $ X_{\sg}(B) $ control for the gauge transformation between two $ H^{\sg}(B) $ fields.  This is \cite[Lemma 1.2]{uhlenbeck1982connections}  for integer regularities and it easily generalizes.
\begin{lemma} \label{UtransfBound}
Let $ U :B \to G $ be such that $ A_i  =  U \tilde{A}_i U^{-1} - \partial_i U U^{-1} $ on $ B $. If $ M = \vn{A_i}_{ H^{\sg}(B)} + \vn{ \tilde{A}_i}_{ H^{\sg}(B)} $, then $ \vn{U}_{X_{\sg}(B)} - 1 \ls  \vn{\nabla U}_{H^{\sg}(B)}  \ls M (1+M) $. 
\end{lemma}

\begin{proof} 
The $ L^{\infty}(B) $ bound is automatic, so we prove  $ \vn{\pt_i U}_{H^{\sg}(B)} \ls M (1+M) $. Write $ \pt_i U = U \tilde{A}_i - A_i U $, which implies $ \vn{U}_{W^{1,p}(B)} \ls 1+M $ where $ p $ is the Sobolev exponent for $ H^{\sg} $. We use  
 $$
 \vn{V g}_{H^{\sg}(\mb{R}^3) } \ls \vn{V}_{\dot{W}^{1,p} \cap L^{\infty}(\mb{R}^3) } \vn{g}_{H^{\sg}(\mb{R}^3)  } \   \Rightarrow \ \vn{U f}_{H^{\sg}(B)} \ls \vn{U}_{{W^{1,p}(B)}} \vn{f}_{H^{\sg}(B)} 
 $$ 
The first inequality follows from Littewood-Paley theory, while the second follows from the first by an extension argument. Apply this inequality for $ U \tilde{A}_i  $ and $ A_i U $ to obtain  $ \vn{\pt_i U}_{H^{\sg}(B)} \ls M (1+M) $.
\end{proof}

\subsection{$ X^{s,b} $ spaces} We outline the main properties, for more details we refer to \cite{selberg1999multilinear}, \cite{klainerman2002bilinear}. 
The $ X^{s,b} $ spaces associated to the wave equation are defined by the norm
$$
\vn{u}_{X^{s,b}} = \vn{ \jb{\xi}^s \jb{\vm{\tau} - \vm{\xi}}^b \mathcal{F} u (\tau,\xi) }_{L^2_{\tau,\xi}}
$$
where $ \mathcal{F} u $ denotes the space-time Fourier transform on $ \mb{R} \times \mb{R}^{3} $. 

Following \cite{tao2003local} we also use the $ X^{r,\theta}_{\tau=0}= H^r_x H^{\theta}_t $ spaces adapted to solutions of the equation $ \pt_t v = F $, with norm 
$$
\vn{v}_{X^{r,\theta}_{\tau=0}} = \vn{ \jb{\xi}^r \jb{ \tau}^{\theta} \mathcal{F} v (\tau,\xi) }_{L^2_{\tau,\xi}}
$$
For a bounded interval $ J $ and $ X = X^{s,b} $ or $ X^{r,\theta}_{\tau=0} $  one defines the restricted space $ X_J $ consisting of functions or tensors $ u $ defined on $ J \times \mb{R}^3 $ for which the following norm is finite:
$$
\vn{u}_{X_J} = \inf \{  \vn{v}_{X} \ ; \  v_{\mkern 1mu \vrule height 2ex\mkern2mu J \times \mb{R}^3 }=u      \ \} 
$$ 
To say $ u \in X $ on $ J \times \mb{R}^3 $ means $ u \in X_J $, i.e. it has an extension $ v \in X $ to $ \mb{R} \times \mb{R}^3 $.

One has the well-known embeddings
$$
X^{s,\frac{1}{2}+} \subseteq  L^{\infty}_t H^s_x, \qquad X^{r,\frac{1}{2}+}_{\tau=0}  \subseteq  L^{\infty}_t H^r_x.
$$
Moreover, one has the embedding 
\be \label{X00emb}
X^{0,0+ } \subset L^{2+}_t L^2_x  
\ee
which follows by interpolation between $ X^{0,0}=L^2_t L^2_x $ and $  X^{0,\frac{1}{2}+} \subset L^{\infty}_t L^2_x $.

Multiplications by cutoff functions interact well with these spaces:
\begin{lemma} \label{cutoffXsb}
Denote $ \vn{u}_{\mathcal{X}^{1,b} } \defeq \vn{u}_{X^{1,b}} + \vn{\pt_t u}_{X^{0,b}} $ and let $ \chi $ be a bump function. Then, for $ 0 \leq b \leq 1  $ one has
$
\vn{ \chi u}_{\mathcal{X}^{1,b} } \ls \vn{ u}_{\mathcal{X}^{1,b} } $
\end{lemma}

\begin{proof}
It is easily seen that an equivalent norm for $ \mathcal{X}^{1,\tht} $ is 
$$
 \vn{ \jb{\vm{\tau} + \vm{\xi}} \jb{\vm{\tau} - \vm{\xi}}^b \mathcal{F} u (\tau,\xi) }_{L^2_{\tau,\xi}}
$$
When $ b=0 $, respectively $ b=1 $ this is equivalent to 
$$
\vn{u}_{L^2_{t,x}} + \vn{\nabla_{t,x} u}_{L^2_{t,x}}, \quad \text{respectively} \quad \vn{u}_{L^2_{t,x}} + \vn{\nabla_{t,x} u}_{L^2_{t,x}}+ \vn{\Box u}_{L^2_{t,x}}
$$
and then the property follows from the chain rule. The general case follows from these by complex interpolation, see \cite[Lemma 1.4]{grunrock2002new}. 
\end{proof}

\subsubsection{$ X^{s,b} $ and $ X^{r,\theta}_{\tau=0} $ energy estimates}

Suppose $ t_0 \in J $ with $ \vm{J} \ls 1 $ and $ b, \tht \in (\frac{1}{2},1) $, $ s, r \in \mb{R} $. If $ \Box u =F $ then
\be \label{XEnEst1}
\vn{u}_{X^{s,b}_J} + \vn{\pt_t u}_{ X^{s-1,b}_J} \ls \vn{u[t_0]}_{H^s \times H^{s-1}} + \vn{F}_{ X^{s-1,b-1}_J}
\ee
If $ \pt_v = G $ then 
\be \label{XEnEst2}
\vn{v}_{X^{r,\theta}_{\tau=0},J} \ls \vn{v(t_0)}_{H^r} + \vn{G}_{X^{r,\theta-1}_{\tau=0},J}
\ee

\subsubsection{Strichartz estimates} We recall
\be \label{Strichartzz}
X^{s, \frac{1}{2}+} \subseteq L_{t}^{q} L_{x}^{r} 
\ee
when $ q \in (2, \infty] $, $ s \geq 0 $ and 
$  \frac{1}{q}+\frac{1}{r} \leq \frac{1}{2}, \quad \frac{1}{q}+\frac{3}{r} \geq \frac{3}{2}-s. 
$ Moreover, we will use the following Bilinear Strichartz estimate (see e.g. \cite{klainerman1999optimal}):
\be \label{BilStrichartz}
D^{-\left(\frac{1}{2}-\right)}\left(X^{\frac{1}{4} \cdot \frac{1}{2}+} \times X^{\frac{1}{4}, \frac{1}{2}+} \right) \subseteq L_{t}^{2} L_{x}^{2}
\ee


\section{Local well-posedness for Yang-Mills}  \label{SecLWPYM}

\subsection{Approximation by regular initial data sets} \ 

We begin with a $ H^{\sg} \times H^{\sigma-1} $ approximation result, recalling that initial data sets for Yang-Mills have to satisfy the Gauss equation \eqref{Constraint} in order to be admissible. 
The similar statement for $ H^1 \times L^2 $ appears in \cite[Proposition 1.2]{klainerman1995finite}.

\begin{proposition} \label{ApproxID}
Any initial data $ (A_i, E_i) \in H^{\sg} \times H^{\sigma-1} $ satisfying the constraint equation  \eqref{Constraint} can be approximated in $ H^{\sg} \times H^{\sigma-1} $ by a sequence of regular initial data $ (A_i^n, E_i^n) \in H^{\infty} \times H^{\infty}  $  satisfying the constraint equation \eqref{Constraint}. Moreover, $ A_i^n $ can be chosen compactly supported.  
\end{proposition} 

\begin{proof}
Begin with a sequence of smooth compactly supported sequence $ (A^n, \tilde{E}^n) $ converging to $ (A,E) $ in $ H^{\sg} \times H^{\sigma-1} $. Denote the covariant derivatives by $ D^{n} \defeq \pt + [A^n, \cdot] $.  
Using \eqref{Constraint}  and writing, for any $ f \in H^{1-\sg} \dot{H}^1= \dot{H}^1 \cap  \dot{H}^{2-\sg} $ 
$$
\lng D^{n,i} \tilde{E}^n_i  , f \rng = \lng E_i - \tilde{E}^n_i, \partial^i f \rng + \lng [A^n - A, \tilde{E}^n ], f \rng + \lng [A, \tilde{E}^n - E ], f \rng
$$
one obtains $ D^{n,i} \tilde{E}^n_i \to 0 $ in $ H^{\sg-1} \dot{H}^{-1} = \dot{H}^{-1} +  \dot{H}^{\sg-2} $ after using $ H^{\sg} \times  H^{1-\sg} \dot{H}^1 \to H^{1-\sg} $, which is a version of \eqref{SobMult}. 

Define $ E^n_i = \tilde{E}^n_i + D^n_i \phi^n $ which will satisfy \eqref{Constraint} provided we choose 
$$
\Delta_{A^n} \phi^n = -  D^{n,i} \tilde{E}^n_i  
$$
It remains to show $ D^n_i \phi^n \to 0 $ in $ H^{\sigma-1}= L^2 + \dot{H}^{\sigma-1}  $. 
By Proposition \ref{EllipticBdd} we obtain $ \phi^n \to 0 $ in $  \dot{H}^{1} +  \dot{H}^{\sg} $ so it remains to use the following easy version of \eqref{SobMult} 
$$ 
\vn{ [A^n, \phi^n]}_{H^{\sigma-1}} \ls \vn{A^n}_{H^{\sg}} \vn{\phi^n}_{  \dot{H}^{1} +  \dot{H}^{\sg}} \to 0 
$$
\end{proof}

\subsection{Square summability of fractional Sobolev spaces} \ 

We will need the following square-summability property of $ W^{s,2} $ norms.  

\begin{proposition} \label{PropsqSumWs2}
Let $ (B_j)_j $ be a finitely overlapping uniform covering of $ \mb{R}^n $ by balls of radius $ \simeq 1 $ and let  $ s \in (-1,1) $. Then, for any $ u \in  W^{s,2}(\mb{R}^n ) $  one has 
\be \label{sqSumWs2}
\sum_j \vn{u}_{W^{s,2}( B_j)}^2 \simeq \vn{u}_{ W^{s,2}(\mb{R}^n ) }^2.
\ee
\end{proposition}

\begin{proof}
When $ s=0 $ the property is obvious. Suppose first that $ s \in (0,1) $. Clearly, from \eqref{Gagliardoseminorm} 
$$
\sum_j \vm{u}_{\dot{W}^{s,2}(B_j)}^2 \ls \vm{u}_{\dot{W}^{s,2}(\mb{R}^n)  }^2.
$$
Conversely, by splitting the $\mathbb{R}^{n} \times \mathbb{R}^{n}$ integral in \eqref{Gagliardoseminorm}  into regions where there exists $j$ such that both $x, y \in B_{j}$ and regions where $|x-y|>\delta$ 
(where we bound $|u(x)-u(y)| \leq|u(x)|+|u(y)|$ and use the $L^{2}(B_{j})$ norms), we obtain
$$
|u|_{\dot{W}^{s,2}\left(\mathbb{R}^{n}\right)}^{2} \lesssim \sum_{j}\|u\|_{W^{s,2}\left(B_{j}\right)}^{2}.
$$
Now assume $ s\in (-1,0) $. Let $ (\psi_j)_j $ be a partition of unity associated to $ (B_j)_j $. 

Let $ u $ be a tempered distribution and let $ \phi $ be a Schwartz function. Using \eqref{cutoffWs2} and what we already proved
\begin{align*}
| \lng  u, \phi \rng | & \ls \sum_j \vn{u}_{W^{s,2}\left(B_{j}\right)}  \vn{\psi_j \phi}_{W^{-s,2}_0(B_j)}\\
& \ls  \big( \sum_j \vn{u}_{W^{s,2}\left(B_{j}\right)}^2 \big)^{\frac{1}{2}} \vn{\phi}_{ W^{-s,2}(\mb{R}^n) }
\end{align*}
This shows  RHS \eqref{sqSumWs2} $\ls $ LHS \eqref{sqSumWs2}. Conversely, 
\begin{align*}
LHS \eqref{sqSumWs2} ^{\frac{1}{2}} & = \sup_{\vn{c}_{\ell^2} \leq 1} \sum_j c_j  \vn{u}_{W^{s,2}( B_j)} = \sup_{\vn{c}_{\ell^2} \leq 1} \sum_j \sup_{\vn{ \phi_j }_{W^{-s,2}_0(B_j) } \leq 1}   c_j  \lng u, \phi_j \rng \\
& \ls \sup_{\vn{\phi}_{ W^{-s,2}(\mb{R}^n) } \leq 1}
 \lng u, \phi \rng =  \vn{u}_{ W^{s,2}(\mb{R}^n ) } = RHS \eqref{sqSumWs2} ^{\frac{1}{2}}
\end{align*}
\end{proof}

\subsection{Initial data extension} \

We recall the following result from \cite{oh2019hyperbolic}, which we state only for a ball $ B $ of radius $ \simeq 1$. 

\begin{theorem}[\cite{oh2019hyperbolic} - Theorem 5.4.] \label{ExtensionHrho}
Let $ \rho \in (\frac{1}{2},\frac{5}{2}) $ and let $ (\bar{A}_{x}, \bar{E}_{x}) \in H^{\rho}(B) \times H^{\rho-1}(B) $ satisfy the Gauss equation    \eqref{Constraint}. 
If $ \vn{\bar{A}_x}_{H^{\frac{1}{2}}(B)} \leq \ep $ for a sufficiently small $ \ep>0 $, then there exist 
$ (\tilde{A}_{x}, \tilde{E}_{x}) \in H^{\rho}(\mb{R}^3 ) \times H^{\rho-1}(\mb{R}^3) $ which satisfies the Gauss equation \eqref{Constraint}, which coincides with $ (\bar{A}_{x}, \bar{E}_{x}) $ on $ B $ such that 
$$
\vn{ (\tilde{A}_{x}, \tilde{E}_{x}) }_{ H^{\rho} \times H^{\rho-1}} \ls \vn{ (\bar{A}_{x}, \bar{E}_{x})   }_{  H^{\rho}(B) \times H^{\rho-1}(B)  }
$$ 
and such that the map $   (\bar{A}_{x}, \bar{E}_{x})  \mapsto (\tilde{A}_{x}, \tilde{E}_{x}) $ is locally Lipschitz continuous. If $  (\bar{A}_{x}, \bar{E}_{x}) $ is smooth then so is $ (\tilde{A}_{x}, \tilde{E}_{x}) $. 
\end{theorem}

The proof is based from the following solvability result for the inhomogeneous Gauss equation.

\begin{proposition}[\cite{oh2019hyperbolic} - Prop. 4.2.] \label{SolveCovDiv}
Suppose $ a $ satisfies $\|a\|_{\dot{H}^{\frac{1}{2}}(\mathbb{R}^3)} \leq \epsilon_{*}$. Let $ K $ be a convex domain. If $ \epsilon_{*}>0 $ is small enough there exists a solution operator $T_{a}$ for the equation $ \mathbf{D}^{\ell} e_{\ell}=h $
with the following properties:

(1) $ \left\|T_{a} h\right\|_{\dot{W}^{r, p}} \lesssim_{K,r,p} \|h\|_{\dot{W}^{r-1, p}} $ holds for $2 \leq p<\infty$ and $1-\frac{3}{p}< r <\frac{3}{2}$.

(2)  If $h=0$ in $\lambda K$, then $T_{a} h=0$ in $\lambda K$.

(3)  If $ a $ and $h$ are smooth, so is $T_{a} h$.
\end{proposition}

In the proof of the almost conservation law we need the following version of Theorem \ref{ExtensionHrho}. The proof carries over almost verbatim from Theorem 5.4. from \cite{oh2019hyperbolic}. We present the argument for the sake of completeness. 

\begin{proposition} \label{Extension}
Suppose $ (a,e) $ satisfy the constraint equation \eqref{Constraint} on a ball $ B \subset \mathbb{R}^3 $ of radius $ \simeq 1 $ and 
$$
(a,e) \in ( H^{1} +  N^{\sg-1} H^{\sg} ) (B) \times ( L^2 +  N^{\sg-1} H^{\sg-1}) (B)    
$$
Suppose $\|a\|_{H^{\frac{1}{2}}(B)} \leq \epsilon $ for $ \epsilon >0 $ small enough. Then there exist $ (\bar{a}, \bar{e} ) $ solving  \eqref{Constraint} on $ \mathbb{R}^3 $ which coincide with $ (a,e) $ on $ B $ with 
$$
\vn{\bar{a}}_{ H^1 +  N^{\sg-1} H^{\sg}}  + \vn{\bar{e}}_{L^2 +  N^{\sg-1} H^{\sg-1}} \ls \vn{a}_{ H^{1} +  N^{\sg-1} H^{\sg} (B)  } + \vn{e}_{L^2 +  N^{\sg-1} H^{\sg-1} (B)}. 
$$
If $ (a,e) $ are smooth, then $(\bar{a}, \bar{e} ) $ are regular.
\end{proposition}

\begin{proof}
Decompose $ a=a^1 +N^{\sg-1} a^{\sg} $, $ e=e^0+N^{\sg-1} e^{\sg-1} $ where $ a^1 \in H^{1}(B) $, $ a^{\sg} \in  H^{\sg} (B) $, $ e^0 \in L^2(B) $, $ e^{\sg-1} \in  H^{\sg-1} (B) $. One begins by applying a (universal) extension operator to $ a^1, a^{\sg} , e^0, e^{\sg-1} $ obtaining $ \bar{a}^1, \bar{a}^{\sg} , \bar{e}'^{0}, \bar{e}'^{\sg-1} $ on $ \mb{R}^3  $ with comparable norms. Let $ \bar{a}=\bar{a}^1 +N^{\sg-1}  \bar{a}^{\sg} $, $ \bar{e}' = \bar{e}'^{0} +N^{\sg-1} \bar{e}'^{\sg-1}  $. 

Since the constraint equation \eqref{Constraint} is violated in general outside $ B $ by $ (\bar{a}, \bar{e}' ) $, one must correct $ \bar{e}' $ by adding a term. Let  $  h^{-1}= (\mathbf{D}^{\bar{a}})^{\ell} \bar{e}'^{0}_{\ell} \in H^{-1} $ and $ h^{\sg-2}= (\mathbf{D}^{\bar{a}})^{\ell} \bar{e}'^{\sg-1} _{\ell} \in H^{\sg-2}  $ be the errors of the Gauss equations, which are supported outside $ B $. Now apply Proposition \ref{SolveCovDiv} to obtain $ d^0 = - T_{\bar{a} } h^{-1} $ and $ d^{\sg-1} = - T_{\bar{a} } h^{\sg-2} $ which solve $ (\mathbf{D}^{\bar{a}})^{\ell} d^0_{\ell} = - h^{-1} $ and $  (\mathbf{D}^{\bar{a}})^{\ell} d^{\sg-1}_{\ell} = -h^{\sg-2}  $ and satisfy 
$ \vn{d^0}_{L^2} \ls \vn{ \bar{e}'^{0}}_{L^2} $, $ \vn{d^{\sg-1}}_{  H^{\sg-1} }\ls \vn{\bar{e}'^{\sg-1}}_{  H^{\sg-1} } $.

Now let $ d = d^0 + N^{\sg-1} d^{\sg-1} $ and define $ \bar{e} =  \bar{e}' + d $. Then $ (\bar{a},\bar{e}) $ satisfy the conditions. 
\end{proof}

\subsection{Review of Theorem \ref{LWPHloc} and proof of Corollary \ref{CorLWP}} \label{SecProofCorLWP}
 \ 

For the sake of completeness and to set up the notation we review the argument for Theorem \ref{LWPHloc} from \cite[Theorem 1.27]{oh2019hyperbolic}. The key technical tool is their initial data surgery result in Theorem \ref{ExtensionHrho} (\cite[Theorem 5.4]{oh2019hyperbolic}).

We begin with an initial data set $ (\bar{A}_{x}, \bar{E}_{x}) \in H^{\sg}(\mb{R}^3) \times H^{\sg-1}(\mb{R}^3) $ satisfying \eqref{Constraint}.  It suffices to consider $ \frac{3}{4} < \sg < 1 $ and to obtain a solution for positive times. 

Rescale the initial data by $ \bar{A}^{\lmd}(x) = \lmd^{-1} \bar{A}( x/\lmd) $ , $ \bar{E}^{\lmd} (x) = \lmd^{-2} \bar{E}(x/\lmd ) $ so that 
$$
\vn{ ( \bar{A}^{\lmd},  \bar{E}^{\lmd}) }_{   \dot{H}^{\sg} \times \dot{H}^{\sg-1}} \ls \lmd^{ \frac{1}{2}- \sg  }  \vn{ ( \bar{A},  \bar{E}) }_{   \dot{H}^{\sg} \times \dot{H}^{\sg-1}} \ll \ep_{*} 
$$
where $ \ep_{*}>0 $ is chosen small enough to be able to apply Theorem \ref{LWPTao} and Theorem \ref{ExtensionHrho} to obtain modified initial data sets as follows. 

Let  $ (B_j)_j $ be a finitely overlapping uniform covering of $ \mb{R}^3 $ by balls of radius $ 2 $ centered at $ (y_j)_{j} $ such that the domains of dependency $ ( \mathcal{D}_j)_{j} $ with bases $ (B_j)_j $ form a covering of $  [0,1] \times \mb{R}^3 $. 
These are truncated cones 
$ \mathcal{D}_j \defeq \{ (t,x) \ | \ t+\vm{x-y_j} < 2, \ t \in [0,1] \}. $

 One may apply Theorem \ref{ExtensionHrho} for each $ B_j $ to obtain an extension $ ( \bar{A}^{(j)},  \bar{E}^{(j)}) $ to $ \mb{R}^3 $ satisfying the constraint equation \eqref{Constraint} and which coincide with $  ( \bar{A}^{\lmd},  \bar{E}^{\lmd}) $ on $ B_j $ with the bounds
$$
\vn{( \bar{A}^{(j)},  \bar{E}^{(j)})}_{H^{\sg} \times H^{\sg-1}} \ls \vn{ ( \bar{A}^{\lmd},  \bar{E}^{\lmd}) }_{W^{\sg,2} \times  W^{\sg-1,2} (B_j)   } \ll \ep_{*}. 
$$
The map $ ( \bar{A}^{\lmd},  \bar{E}^{\lmd}) \mapsto ( \bar{A}^{(j)},  \bar{E}^{(j)}) $ is locally Lipschitz on these spaces. 

To each $ ( \bar{A}^{(j)},  \bar{E}^{(j)}) $ one applies Theorem \ref{LWPTao} (i.e. Theorem 1.1. in \cite{tao2003local}) to obtain a Yang-Mills solution $ A^{(j)} $ in the temporal gauge on $ [0,1] \times \mb{R}^3 $ satisfying
\be \label{ajtbdAlmd}
\vn{ A^{(j)}[t] }_{ L^{\infty} ( H^{\sg} \times H^{\sg-1})( [0,1] \times \mb{R}^3) } \ls  \vn{ ( \bar{A}^{\lmd},  \bar{E}^{\lmd}) }_{W^{\sg,2} \times  W^{\sg-1,2} (B_j)   }
\ee

One obtains a solution $ A^{\lmd} $ on  $ [0,1] \times \mb{R}^3 $ defined by each $ A^{(j)} $ restricted to $  \mathcal{D}_j $ since any two of these solutions coincide on their common domain as a consequence of finite speed of propagation in the temporal gauge. 

\

We apply the square summability estimate \eqref{sqSumWs2} at the initial time and at an arbitrary time $ t \in [0,1] $, together with \eqref{ajtbdAlmd} obtaining
\begin{align*}
\vn{ A^{\lmd}[t] }_{ H^{\sg} \times H^{\sg-1} }^2 & \simeq \sum_j \vn{ A^{(j)}[t] }_{ W^{\sg,2} \times  W^{\sg-1,2}(  \mathcal{D}_j \cap \{t \} \times \mb{R}^3  ) }^2 \\
& \ls \sum_j  \vn{ ( \bar{A}^{\lmd},  \bar{E}^{\lmd}) }_{W^{\sg,2} \times  W^{\sg-1,2} (B_j)   }^2 \simeq  \vn{ ( \bar{A}^{\lmd},  \bar{E}^{\lmd}) }_{ H^{\sg} \times H^{\sg-1} }^2 . 
\end{align*}
Undoing the scaling one obtains a solution $ A \in C_t H^{\sg} \cap C_t^1 H^{\sg-1}( [0,\lmd^{-1}] \times \mb{R}^3)  $ given by $ A(t,x) = \lmd A^{\lmd}(\lmd t, \lmd x)  $. 

\

If  $ (\bar{A}', \bar{E}') \in H^{\sg} \times H^{\sg-1}(\mb{R}^3) $ is another initial data set satisfying \eqref{Constraint} one has
\be \label{LipschBj}
\vn{ A^{(j)}[\cdot] -  A'^{(j)}[\cdot] }_{ L^{\infty} ( H^{\sg} \times H^{\sg-1})} \ls  \vn{ ( \bar{A}^{\lmd} - \bar{A}'^{\lmd},  \bar{E}^{\lmd}-  \bar{E}'^{\lmd} ) }_{W^{\sg,2} \times  W^{\sg-1,2} (B_j) }
\ee
By repeating the same square summability argument for differences using \eqref{LipschBj} and undoing the scaling one obtains Lipschitz dependence 
$$
\vn{A[t]-A'[t]}_{ H^{\sg} \times H^{\sg-1} } \ls \vn{(\bar{A}, \bar{E}) - (\bar{A}', \bar{E}') }_{ H^{\sg} \times H^{\sg-1} }.
$$ 
The Lipschitz constant depends, of course, on the size of the initial data.


\section{Modified local well-posedness} \label{SecModLWP}

In this section we derive control of space-time norms that will be needed for the trilinear estimate \eqref{deltaEnergyReduced}.

\

Let $ A_i $ be a solution to \eqref{YM} in the temporal gauge $ A_0 = 0 $. 
We use the Leray projections \eqref{Leray} to split $ A=  {\bf P} A+ {\bf P}^{\perp} A $ into a divergence-free part and a curl-free part. From the definition of $ {\bf P}^{\perp} $ and the first equation in \eqref{YMTemporal}, and by applying $ {\bf P} $ to the second equation in \eqref{YMTemporal} one obtains 
\be \label{YMTemporal2}
\begin{aligned}
 \pt_t {\bf P}^{\perp} A & = \Delta^{-1} \nabla [ \pt_t A_j, A_j ] \\
 \Box \ {\bf P} A  & =    - 2 {\bf P} [A_j, \pt_j A ] +  {\bf P} [A_j, \nabla A_j ] + {\bf P} [A, \pt_j A_j]-  {\bf P}  [A_j, [A_j, A ]] 
\end{aligned}
\ee

Following \cite{tao2003local} one proceeds to uncovering the null structure in \eqref{YMTemporal2} as follows. 

Recall the definition of a null form from \eqref{DefNullForm}, \eqref{Qij}. For Lie algebra valued fields, $ {\bf N} (A,B) $ is understood as a linear combination of 
\be \label{DefNullFormLie}
\Delta^{-1} \nabla^i Q_{ij} (A_{\ell}, B_{k}), \quad \text{and} \quad  Q_{ij} (\Delta^{-1} \nabla^i  A_{\ell}, B_{k}) \quad 
\ee
where 
\be \label{QijLie}
Q_{ij} (C,D) \defeq [\pt_i C, \pt_j D] - [\pt_j C, \pt_i D]
\ee

Denote $ A^{df} = {\bf P} A $ and $ A^{cf} = {\bf P}^{\perp} A $. Using the identities \eqref{Nullformidentities} one can write the system \eqref{YMTemporal2} as
\begin{align}
 \label{YMTemporal3}
 \pt_t A^{cf} & =  \nabla^{-1} \mathcal{O}( \pt_t A, A ) \\
 \notag
 \Box  A^{df}  &=  {\bf N} (A^{df},A^{df}) +  \mathcal{O}(A^{df},\nabla A^{cf})  + \mathcal{O}(A^{cf},\nabla A^{df}) + \mathcal{O}(A^{cf},\nabla A^{cf}) +  \mathcal{O} (A^3) 
 \end{align}
where $ \mathcal{O} $ denote (matrix-valued) bilinear or trilinear expressions in the values of the inputs.  

\

Well-posedness in $ I^{-1} H^1 \times I^{-1} L^2 $ holds for this system.

\begin{proposition} \label{ModLWP}
(1) Let $ A_{x}(t,x) $ be a global regular solution to \eqref{YM} in the temporal gauge $ A_0 = 0 $ such that
$$
\vn{I A_x(t_0)}_{H^1} + \vn{I \partial_t A_x(t_0)}_{L^2} \leq \ep
$$
Let $ t_1 \leq t_0+1 $. Then, assuming $ \ep $ is small enough, one has
\be \label{ISobbdatt}
\vn{I A_x[t]}_{L^{\infty}_t (H^1 \times L^2)(  [t_0,t_1]\times \mb{R}^3 ) } \leq C \ep 
\ee
(2) Moreover, there exists a spatial field $ U : \mb{R}^3 \to G $, $ U=U(x) $ such that the change of gauge 
\be \label{gaugechange}
A_i \mapsto U A_i U^{-1} - \partial_i U U^{-1}
\ee
achieves $ \text{div} A_x (t_0)=0 $, the bound \eqref{ISobbdatt} and 
\begin{align}
\label{divfree}
& I A^{df} \in \ep X^{1,\frac{3}{4}+},   & I \pt_t A^{df} \in \ep X^{0,\frac{3}{4}+} \\
\label{curlfree}
& I A^{cf} \in \ep X_{\tau=0}^{1+\frac{1}{4}, \frac{1}{2}+},   & P_k I^2 \pt_t A^{cf} \in \ep L^{\infty}_t H^{\frac{1}{2}}
\end{align} 
hold on $ [t_0,t_1] \times \mb{R}^3 $, where $ A^{df} = {\bf P} A_x $, $ A^{cf} =   {\bf P}^{\perp} A_x $.
\end{proposition}

This proposition follows from the local well-posedness theory developed in \cite{tao2003local}. 

\subsection{Proof of Proposition \ref{ModLWP}} \ 

{\bf Step 1.} 
We begin by discussing the change of gauge that sets $ A^{cf}(t_0)= 0$, i.e. the Coulomb gauge  $ \text{div} A_x (t_0)=0 $, only at the initial time. Without this transformation \eqref{curlfree} could not have more regularity than the initial data. 

The following procedure is adapted here to $ I^{-1} H^1 $ from the $ H^{s} $ case in \cite[Section 3]{tao2003local} (arxiv version), attributed there to Mark Keel. Suppose 
$$
\vn{I A_x(t_0)}_{H^1} \leq \ep, \qquad \vn{I A^{cf}(t_0)}_{H^1} \leq \delta
$$
for $ \delta \leq \ep \ll 1 $. 
Write $ A^{cf}(t_0)= \nabla V $ for a $  \mathfrak{g} $-valued field $ V $ with $ \vn{V}_X \ls \delta $ where the $ X $ norm is defined in \eqref{XnormDef} and obeys the algebra and product estimates \eqref{Algebraproducts}. 
%

Applying \eqref{gaugechange} with $ U \defeq \exp(V) $ one obtains a new temporal (because $ V $ is independent of $ t $) solution $ \tilde{A} $, with initial data written as
$$
\tilde{A}(t_0)= e^V A^{df}(t_0) e^{-V} + \big(  e^V \nabla V - \nabla(e^V)  \big) e^{-V}
$$
and 
$$
e^V A^{df}(t_0) e^{-V}  = A^{df}(t_0) + \big(  e^V A^{df}(t_0) e^{-V} -  A^{df}(t_0) \big) 
$$
Based on these equalities, Taylor expansions and \eqref{Algebraproducts} one obtains that 
$$
\vn{I \tilde{A}(t_0)}_{H^1} \leq \ep+ C \ep \delta , \qquad \vn{I \tilde{A}^{cf}(t_0)}_{H^1} \leq C \ep \delta, \qquad \vn{U-1}_X \ls \delta
$$
Iterating this procedure, one obtains in the limit a new temporal solution, which by abuse of notation we still denote by $ A_x $, such that $ A^{cf}(t_0) = 0 $. The change of gauge \eqref{gaugechange} is done with a $ U $ with $ \vn{U-1}_X \ls \ep $. The new solution has
\be \label{Iindataep}
\vn{I A_x(t_0)}_{H^1} + \vn{I \partial_t A_x(t_0)}_{L^2} \ls \ep
\ee
and \eqref{ISobbdatt} follows using \eqref{Algebraproducts} if we prove 
\be \label{Ienspbdd}
\vn{I A_x[t]}_{L^{\infty}_t (H^1 \times L^2)(  [t_0,t_1]\times \mb{R}^3 ) } \ls \ep 
\ee 
for the new solution. Thus it remains to prove the bounds in part (2) of Prop. \ref{ModLWP}.

\

{\bf Step 2.} 
We now consider a \eqref{YM} solution in the temporal gauge satisfying \eqref{Iindataep} with $ A^{cf}(t_0) = 0 $. The system of equation solved by $ I A $ is written schematically as 
\be \label{SystemIA}
\Box I A^{df} = I \mathcal{N}, \qquad \pt_t I A^{cf} = I \nabla^{-1}(A \pt_t A) 
\ee
$$
\mathcal{N} \defeq {\bf N} (A^{df}, A^{df}) + A^{df} \nabla A^{cf} + A^{cf} \nabla A^{df} +  A^{cf} \nabla  A^{cf} + A^3 
$$
For $ J =[t_0,t_1] $ let 
$$
\vn{I A}_{\mathcal{X}} \defeq \left\| I A^{df} \right\|_{X_J^{1, \frac{3}{4}+}} + \left\| I \pt_t A^{df} \right\|_{X^{0, \frac{3}{4}+}_J} + \left\| I A^{cf} \right\|_{X_{\tau=0, J}^{1+\frac{1}{4}, \frac{1}{2}+}}
$$
One obtains a-priori estimates using the following bounds: 

\begin{align*}
&\left\| I {\bf N} (A_{1}, A_{2}\right\|_{X^{0,- \frac{1}{4}+}} \ls \left\| I A_{1}\right\|_{X^{1, \frac{3}{4}+}} \left\|I A_{2}\right\|_{X^{1, \frac{3}{4}+}} \\
&\left\| I (A_{1} \nabla A_{2} ) \right\|_{ X^{0,- \frac{1}{4}+}} + 
 \left\| I(A_{2} \nabla A_{1}) \right\|_{ X^{0,- \frac{1}{4}+}}\lesssim   
\left\| I A_{1}\right\|_{X^{1, \frac{3}{4}+}} 
 \left\| I A_{2}\right\|_{X_{\tau=0}^{1+\frac{1}{4}, \frac{1}{2}+}}\\
&\left\| I(A_{1} \nabla A_{2} ) \right\|_{ X^{0,- \frac{1}{4}+}} \lesssim 
\left\| I A_{1}\right\|_{X_{\tau=0}^{1+\frac{1}{4}, \frac{1}{2}+}}
\left\| I A_{2}\right\|_{X_{\tau=0}^{1+\frac{1}{4}, \frac{1}{2}+}}\\
&\left\| I(A_{1} A_{2} A_{3} ) \right\|_{ X^{0,- \frac{1}{4}+}} \lesssim \prod_{i=1}^{3} \min \big( \left\| I A_{i}\right\|_{X^{1, \frac{3}{4}+}}, 
 \left\| I A_{i}\right\|_{X_{\tau=0}^{1+\frac{1}{4}, \frac{1}{2}+}}  \big) \\
\end{align*}
and 
\begin{align*}
&\left\| I \nabla^{-1}\left(A_{1} \partial_{t} A_{2}\right)\right\|_{X_{\tau=0}^{1+\frac{1}{4}, -\frac{1}{2}+}} \lesssim 
\left\| I A_{1}\right\|_{X^{1, \frac{3}{4}+}} \left\|I A_{2}\right\|_{X^{1, \frac{3}{4}+}} 
\\
&\left\| I \nabla^{-1}\left(A_{1} \partial_{t} A_{2}\right)\right\|_{X_{\tau=0}^{1+\frac{1}{4}, -\frac{1}{2}+}} +
\left\| I \nabla^{-1}\left(A_{2} \partial_{t} A_{1}\right)\right\|_{X_{\tau=0}^{1+\frac{1}{4}, -\frac{1}{2}+}}  \lesssim \left\| I A_{1}\right\|_{X^{1, \frac{3}{4}+}} 
\left\| I A_{2}\right\|_{X_{\tau=0}^{1+\frac{1}{4}, \frac{1}{2}+}}\\
& \left\| I \nabla^{-1}\left(A_{1} \partial_{t} A_{2}\right)\right\|_{X_{\tau=0}^{1+\frac{1}{4}, -\frac{1}{2}+}}  \lesssim \left\| I A_{1}\right\|_{X_{\tau=0}^{1+\frac{1}{4}, \frac{1}{2}+}}
\left\| I A_{2}\right\|_{X_{\tau=0}^{1+\frac{1}{4}, \frac{1}{2}+}}
\end{align*}
holding for any $  A_{1}, A_{2}, A_{3} $ on $ \mathbf{R}^{3+1} $.

These estimates are stated and proved in \cite[Eq (15)-(21)]{tao2003local} without the $ I $ multipliers and with $ X^{s, \frac{3}{4}+}, X^{s-1,- \frac{1}{4}+}, X_{\tau=0}^{s+\frac{1}{4}, \frac{1}{2}+}, X_{\tau=0}^{s+\frac{1}{4}, -\frac{1}{2}+}  $ instead of 
$ X^{1, \frac{3}{4}+} $, $ X^{0,- \frac{1}{4}+} $, $ X_{\tau=0}^{1+\frac{1}{4}, \frac{1}{2}+} $, $ X_{\tau=0}^{1+\frac{1}{4}, -\frac{1}{2}+}  $ for any $ s>\frac{3}{4} $. The version stated here follows immediately from those by doing basic Fourier decompositions. For details see the interpolation lemma in \cite[Lemma 12.1]{colliander2004multilinear}.

Combining these estimates with the energy-type inequalities \eqref{XEnEst1}, \eqref{XEnEst2} for $ X^{s,b} $, $X_{\tau=0}^{r, \tht} $ spaces and $ A^{cf}(t_0) = 0 $ one obtains
$$
\vn{I A}_{\mathcal{X}} \ls \ep + \vn{I A}_{\mathcal{X}}^2 + \vn{I A}_{\mathcal{X}}^3 
$$
from which one obtains $ \vn{I A}_{\mathcal{X}} \ls \ep $  (if $ \ep$ is small enough) by a continuity argument in $ t_1 $. This proves the first three bounds in \eqref{divfree}, \eqref{curlfree}. In particular, 
$$
\vn{I A^{df}[t]}_{L^{\infty}_t (H^1 \times L^2)} + \vn{I A^{cf}}_{L^{\infty}_t H^{1+\frac{1}{4} }} \ls \ep 
$$
From \eqref{SystemIA}, using the inequality \eqref{SobMI}, for any $ t \in [t_0,t_1] $ one obtains
$$ \vn{I \pt_t A^{cf}}_{L^2} \ls \vn{IA}_{H^1} \big( \vn{I \pt_t  A^{df}}_{L^2}  + \vn{I \pt_t  A^{cf}}_{L^2} \big) \ls \ep^2 + \ep \vn{I \pt_t  A^{cf}}_{L^2}  $$
Absorbing the last term to the LHS completes the proof of \eqref{Ienspbdd}. 

In fact, we can improve the previous inequality at high frequencies -  we obtain the last bound in \eqref{curlfree} by the Littlewood-Paley trichotomy for the equation $ \pt_t A^{cf} = \nabla^{-1}(A \pt_t A) $ from the $ H^1 $ and $ L^2_x $ bounds already obtained in \eqref{Ienspbdd} and using Bernstein's inequality for the low frequency (input or output). \qed

\subsubsection{Improved bounds for $ \pt_t A^{cf} $} 
\begin{proposition}
Let $ A_x $ be the solution from Proposition \ref{ModLWP} Part (2) on $ [t_0,t_1] \times \mb{R}^3 $. Then
\begin{align}
\label{PkDtAcfL2Linfx}
\vn{P_k \pt_t A^{cf}}_{L^2_t L^{\infty}_x} \ls \ep 2^{(\frac{1}{2}+)k} \jb{\frac{2^k}{N}}^{2(1-\sg)}  \\
\label{PkDtAcfL2L2}
\vn{P_k \pt_t A^{cf}}_{L^2_{t} L^{2}_x} \ls   \ep   2^{-\frac{3}{4}k}  \jb{\frac{2^k}{N}}^{2(1-\sg)}   \\
\label{PkDtAL2Linfx}
\vn{P_k \pt_t A}_{L^2_t L^{\infty}_x} \ls \ep 2^{(1+)k} \jb{\frac{2^k}{N}}^{1-\sg} 
\end{align}
\end{proposition}

\begin{proof}

For \eqref{PkDtAcfL2Linfx}, \eqref{PkDtAcfL2L2} we use the equation  
$$ \pt_t A^{cf} = \nabla^{-1}(A^{df} \pt_t A^{df} + A^{df} \pt_t A^{cf}+ A^{cf} \pt_t A^{df} + A^{cf} \pt_t A^{cf})  $$ 
and perform a Littlewood-Paley decomposition of the RHS. 

{\bf Proof of \eqref{PkDtAcfL2Linfx}} 
First consider the case when the output is a high frequency (low-high or high-low cases). When
and at least one of the inputs is $df $ - place that input in $  L^{2+} L^{\infty} $ using  Strichartz \eqref{Strichartzz} and the other term in $ L^{\infty}L^{\infty} $ (using Bernstein and $ L^{\infty} L^2 $). When both inputs are $ cf $ use $ L^{\infty}L^{\infty} $ norms (and Bernstein with $ L^{\infty} L^2 $), based on \eqref{curlfree}.

In the case of high-high-to-low frequency interactions we use Bernstein first. When both inputs are $ cf $ - estimate both inputs in $ L^{\infty} L^2 $ using \eqref{curlfree} and the output in $ L^{2} L^1 $. When both inputs are $ df $ use the bilinear Strichartz estimate \eqref{BilStrichartz} to bound the output in $ L^2 L^2 $. 

Now assume one of the high-high inputs is $ df $ and one is $ cf $. For $ A^{df} \pt_t A^{cf} $  estimate the output in $ L^2 L^2 $, placing the $ df $ term in $ L^{2+} L^{\infty} $ and the $ cf $ term in $ L^{\infty} L^2 $. For $ A^{cf} \pt_t A^{df} $ estimate the output in $ L^2 L^{\frac{8}{5}} $, placing $ A^{cf} $ in $ L^{\infty} L^2 $ and $ \pt_t A^{df} $ in $ L^2 L^8 $ (by interpolating between $ L^{2} L^{\infty} $ and energy - $ L^2 L^2 $). This is done to ensure adequate summation even when $ \sg $ is close to $ \frac{5}{6} $. 

{\bf Proof of \eqref{PkDtAL2Linfx}} This inequality follows from \eqref{PkDtAcfL2Linfx} and from 
\eqref{divfree} together with Strichartz \eqref{Strichartzz} for $ L^{2+} L^{\infty} $ and H\"older in $ t $. 

{\bf Proof of \eqref{PkDtAcfL2L2}} For high-high-to-low frequency interactions we argue exactly the same as for \eqref{PkDtAcfL2Linfx}, only the power obtained from Bernstein's inequality changes by a factor of $ 2^{-\frac{3}{2}k} $. The same is true for low-high and high-low $ A^{cf} \pt_t A^{cf} $ interactions. 

For the remaining high-low interactions place $ A $ (high frequency) in $ L^{\infty} L^2 $ and $ \pt_t A $ (low frequency) in  $ L^{2} L^{\infty} $ using \eqref{PkDtAL2Linfx}. 

For low-high interactions, place $ A^{df} $ (low) in $ L^{2} L^{\infty} $  and $ \pt_t A $  (high) in $ L^{\infty} L^2 $. It remains to deal with $ A^{cf}_{k_1} \pt_t A_k^{df} $ for $ k_1 \leq k $. This term is responsible for the $ 2^{-\frac{3}{4}k} $  factor in \eqref{PkDtAcfL2L2}, rather than $ 2^{-k} $ (which can probably be fully obtained). Place $  A^{cf}_{k_1}$ in $ L^{\infty} L^{12} $ using  \eqref{curlfree} and Bernstein and place $ \pt_t A_k^{df} $ in $ L^{\infty} L^{\frac{12}{5}} $ using Bernstein again. 
\end{proof}


\section{Yang-Mills Heat Flow solutions} \label{SecYMHF}

This section is devoted to the Yang-Mills heat flow. We consider deTurck's and caloric gauges including: well-posedness in $ I^{-1}H^1 $ and $ H^{\rho} $, control of the modified energy \eqref{ModifiedEnergy} from the initial data and the change between the two gauges. 

\subsection{Local well-posedness in DeTurck's gauge} \ 

In the covariant \eqref{cYMHF} equation, setting the DeTurck gauge condition $ A_s = \pt^{\ell} A_{\ell} $ one obtains the parabolic equation for $ A_i $
\be \label{DeTurck}
(\partial_{s} -\Delta) A_{i}=2\left[A^{\ell}, \partial_{\ell} A_{i}\right]-\left[A^{\ell}, \partial_{i} A_{\ell}\right]+\left[A^{\ell},\left[A_{\ell}, A_{i}\right]\right]
\ee

\

In short, this equation can be written schematically as $ (\pt_s - \Delta) A = A \nabla A + A^3  $.

One can separate the linear and nonlinear parts of $A_i $:
\be
A_i(s) = e^{s \Delta} A_i(0) + A^{bil}_i(s)
\ee
where the bilinear term is written schematically as  
\be
A^{bil}_i(s_1) \defeq \int_0^{s_1} e^{(s_1-s)\Delta} ( A \nabla A + A^3 ) \dd s 
\ee 
and will enjoys more favorable estimates, such as \eqref{AbilL2Linf}.

\

This equation is locally well-posed in $ I^{-1} H^1 $:
\begin{proposition} \label{DeTurckProp}
Suppose the initial data satisfies $ I \bar{A}_i \in \ep H^1 $ on $ \mb{R}^3 $ and $ \ep>0 $ is small enough. Then there exists a solution to \eqref{DeTurck} on $ \mb{R}^3 \times [0,1] $ with $ A_i(s=0) =\bar{A}_i $ which satisfies 
\be \label{IAisH1}
 I A_i \in \ep L^{\infty}_{s} H^1( \mb{R}^3 \times [0,1]) \ee
and is the unique limit of regular solutions. 

This provides a solution to \eqref{cYMHF} with $ A_s = \pt^{\ell} A_{\ell} $. 
If the initial data is regular then also the solution is.
Moreover, 
\be \label{AsDTbds}
A_s \in  \ep L^{1}_s L^{\infty}_x, \qquad \nabla A_s \in \ep L^1_s L^3_x , \qquad A_s \in \ep L^1_s H^{\frac{3}{2}+} 
\ee
and for any $ s_1, s_2 \in [0,1] $, one has
\be \label{IDeltaAs}
\int_{s_1}^{s_2} I \Delta A_s \dd s  \in \ep L^2_x 
\ee
\end{proposition}

\begin{proof}
This is proved using standard methods, so we only sketch the proof. One sets up a contraction argument or Picard iteration, 
looking for a fixed point for the functional 
$$
\Phi A_i (s_1) = e^{s_1 \Delta} \bar{A}_i + \int_0^{s_1} e^{(s_1-s) \Delta} ( A \nabla A + A^3 ) \dd s
$$ 
in a suitable space $ S $ where one has the estimates
\begin{align*}
\vn{ \int_0^{s_1} e^{(s_1-s) \Delta} ( A \nabla B  ) \dd s}_S & \ls  \vn{A}_S \vn{B}_S \\
\vn{ \int_0^{s_1} e^{(s_1-s) \Delta} ( A_1 A_2 A_3) ) \dd s}_S & \ls  \vn{A_1}_S \vn{A_2}_S \vn{A_3}_S
\end{align*}
One may use the space
\be \label{Snorm}
\vn{A}_S \defeq \vn{ \vn{ \jb{2^{2k}s}^{10} \vn{P_k I A(s)}_{H^1} }_{L^{\infty}_s} }_{\ell^2_{k \in \mb{N}}}
\ee
to check these estimates using basic Littewood-Paley theory. 

\

One may also use the following parabolic energy estimates, \cite[Proposition 3.12]{oh2014gauge} 
\be
\begin{aligned} \label{ParabolicEnEst}
\vn{ \int_0^{s_1} e^{(s_1-s) \Delta}  \mathcal{N}(s) \dd s}_{\tilde{S}} \ls \vn{s  \mathcal{N}}_{L^1_{\frac{\dd s}{s}} \cap L^2_{\frac{\dd s}{s}}  ( L^2_x ) }  \\
\vn{f}_{\tilde{S}} \defeq \vn{f}_{L^{\infty}_s L^2_x} + \vn{s^{\frac{1}{2} }\pt_x f}_{L^{\infty}_s \cap L^2_{\frac{\dd s}{s}}(L^2_x)} + \vn{s \pt_x^2 f}_{L^2_{\frac{\dd s}{s}} L^2_x} 
\end{aligned}
\ee
used in the proof of  \cite[Proposition 5.2]{oh2014gauge} which is the corresponding well-posedness statement for $ \dot{H}^1 $. Here one writes the equation as 
\be \label{IDxAeq}
(\pt_s - \Delta)I \jb{D_x} A = I \jb{D_x}( A \nabla A + A^3 )
\ee 
One defines $ \vn{A}_S \defeq \vn{I \jb{D_x} A }_{ \tilde{S} } $. 
The needed estimates to check for this case of $ I^{-1} H^1 $, as opposed to $ \dot{H}^1 $, become
\be \label{DTbilest}
\begin{aligned}
\vn{I \jb{D} ( f \nabla g)}_{L^2} & \ls \vn{I f}_{H^{\frac{3}{2}+(1-\sg)+}} \vn{I g}_{H^2} \\
 \vn{I \jb{D} ( f_1 f_2 f_3)}_{L^2} & \ls \vn{I f_1}_{H^{2- \frac{2}{3} \sg}} \vn{I f_2}_{H^{2- \frac{2}{3} \sg}} \vn{I f_3}_{H^{2- \frac{2}{3} \sg}}
\end{aligned}
\ee
which follow easily from \eqref{SobMult} or from Littewood-Paley, by splitting into different cases depending on whether the inputs have frequencies higher or lower than $ N $. 

By similar arguments for the differentiated equation one obtains persistence of regularity. 
We refer to \cite[Proposition 5.2]{oh2014gauge}  for more details. 

The bounds \eqref{AsDTbds} follow from the $  \vn{A}_S $ bound and the definition of $ A_s = \pt^{\ell} A_{\ell}$. 

Denoting by $  \mathcal{N} $ the nonlinear RHS of \eqref{DeTurck}, for our solution, the bounds above show 
\be
 \vn{ I \jb{D_x}  \mathcal{N}}_{L^1_{s}L^2_x  }  \ls \ep^2 
\ee
We use this to prove \eqref{IDeltaAs}. The equation for $ I A_s $ is of the type $ (\pt_s - \Delta)I A_s = I \partial_x \mathcal{N} $. Integrating the equation and taking $ L^2_x $ norms one obtains 
$$
\vn{ \int_{s_1}^{s_2} I \Delta A_s \dd s }_{ L^2_x} \leq \vn{I A_s(s_1)}_{ L^2_x} + \vn{I A_s(s_2)}_{ L^2_x} + \vn{ I \partial_x \mathcal{N} }_{L^1_{s}L^2_x  } \ls \ep. 
$$
\end{proof}
As a consequence one obtains gauge-invariant bounds for the curvature, such as control of the modified energy \eqref{ModifiedEnergy}.

\begin{corollary} \label{CorHFModEn}
Suppose $ A_i(t,x) $ is a regular solution of the \eqref{YM} equation in the temporal gauge on $  [t_0,t_1] \times \mb{R}^3 $ with $ (I A_i,  I F_{0i}) \in \ep L^{\infty}_t ( H^1 \times L^2 )([t_0,t_1] \times \mb{R}^3) $.

Suppose $ A_{\al} (t,x,s) $ solves \eqref{dYMHF} on $ [t_0,t_1]  \times \mb{R}^3 \times [0,s_0]$ in DeTurck's gauge $ A_s = \pt^{\ell} A_{\ell} $, where $ s_0 =N^{-2} $. Then, assuming $ \ep \ll 1 $, for any $  t \in [t_0,t_1]  $ one has 
\be \label{IL2InitDataF}
 \vn{ (N s^{\frac{1}{2}} )^{1 - \sg} F_{\al \beta}(s)}_{L^{\infty}_s L^2_x} + \vn{ (N s^{\frac{1}{2}} )^{1 - \sg} F_{\al \beta}(s)}_{L^2_{\frac{\dd s}{s}} L^2_x}  \ls \ep
\ee
As a consequence, for \eqref{ModifiedEnergy} one has the gauge-invariant bound $  \mathcal{IE}(t) \ls \ep^2  $ for all $  t \in [t_0,t_1]  $.
\end{corollary}
\begin{proof} Fix a time $ t \in [t_0,t_1] $. For $ A_i $ we apply Proposition \ref{DeTurckProp} and so it obeys the estimates established in that proof. 

We begin with $ F_{ij}(s) $. First note that $ \vn{ [A_i,A_j](s)}_{L^2_x} \ls \ep^2 $ since from \eqref{IAisH1} and Sobolev embedding we have $ A_i(s) \in \ep L^4_x $. From the $ \vn{I \jb{D_x} A }_{ \tilde{S} } $ norm and interpolation one obtains $ (N s^{\frac{1}{2}} )^{1 - \sg} \vn{\pt_i A_j (s)}_{L^2_x} \ls \ep $ and the same for $ \pt_j A_i $ which concludes the first term in \eqref{IL2InitDataF} for $ F_{ij} $. For the second term we note that from control of \eqref{Snorm} we have $ \vn{P_k A_i(s)}_{H^1} \ls \al_k \jb{ \frac{2^k}{N}}^{1-\sg} \jb{2^{2k}s}^{-10} $ for all $ k \geq 0 $, $ s \in [0,s_0 ] $ for some sequence with $ \vn{ (\al_k)}_{\ell^2_{k \in \mb{N}}}  \ls \ep $. We bound
$$
\vn{ (N s^{\frac{1}{2}} )^{1 - \sg}  \pt_j A_i(s)}_{L^2_{\frac{\dd s}{s}} L^2_x}^2 \ls \sum_{0 \leq k \leq N} \int_0^{s_0} (N^2 s)^{1 - \sg}  \al_k^2     \frac{\dd s}{s}  + 
\sum_{ k \geq N}   \al_k^2   \int_0^{s_0} \frac{ (2^{2k}s)^{1-\sg}}{ \jb{2^{2k}s}^{20}  }  \frac{\dd s}{s} 
$$
which is $ \ls \ep^2 $. 

Now consider $  F_{0i}(s) $, which solves the covariant equation \eqref{0CovPara}. Expanding it and canceling the $ [A_s, F ]$ term with the $ [  \pt^{\ell} A_{\ell}, F ] $ term one obtains equation \eqref{FeqDeTurck}, which can be written in the form \eqref{BiEq} for $ B= F_{0i} $. As before, the solution is obtained by Picard iteration. One may use the parabolic energy estimate \eqref{ParabolicEnEst} to estimate the solution of the equation 
$$
(\pt_s - \Delta) I B = I ( A \nabla B + \nabla A  B +   A^2 B ), \qquad I B (0)\in \ep L^2
$$
in the norm $ \vn{B}_{I^{-1} \tilde{S}} = \vn{I B }_{ \tilde{S} } $. The estimates we need now are 
\begin{align*}
\vn{ I(A \nabla B)}_{L^2}+ \vn{ I(B \nabla A)}_{L^2 } & \ls \vn{ IA}_{H^{\frac{3}{2}+(1-\sg)+ }} \vn{I B}_{H^{1+} } \\
\vn{ I(A^2 B) }_{L^2 } & \ls  \vn{I A}_{H^{2-\sg}}^2    \vn{I B}_{H^{2-\sg}}  
\end{align*}
which follow from \eqref{SobMult}, \eqref{SobMultHomog} or from Littewood-Paley, by splitting into frequencies higher or lower than $ N $.
 For $ A $ we already have control of the norms from the proof of Proposition \ref{DeTurckProp}. 
This provides $  \vn{I F_{0i} }_{ \tilde{S} } \ls \ep $. 
One obtains control in the norm $ \vn{ \vn{ \jb{2^{2k}s} \vn{P_k I F_{0i}  (s)}_{L^2 } }_{L^{\infty}_s} }_{\ell^2_{k }} $ as well. Then we conclude exactly like for $ F_{ij} $. 
\end{proof}

We record some $ L^{\infty}_x $ bounds separately. 

\begin{lemma}
Suppose $ A_i $ obeys the conditions of Proposition \ref{DeTurckProp} uniformly on an interval $ I \times \mb{R}^3 \times [0,s_0] $. Then 
\be \label{ALinfty}
\vn{A_i (s)}_{L^{\infty}_t L^{\infty}_x} \ls \ep s^{-\frac{1}{4}} / (N s^{\frac{1}{2}})^{1-\sigma} 
\ee
\be \label{AbilL2Linf}
\vn{ A^{bil}_i(s)}_{L^{\infty}_t L^{\infty}_x} \ls  \ep / (N^2 s)^{1-\sigma} .
\ee

\begin{proof}
Fix a time $ t $. The first bound \eqref{ALinfty} follows immediately from the bounds in Proposition \ref{DeTurckProp}. 
Since $ (s_1-s)^{\frac{3}{4}} e^{(s_1-s)\Delta} : L^2_x \to L^{\infty}_x $, we can bound using \eqref{ALinfty} 
\begin{align*}
\vn{ A^{bil}_i(s_1)}_{L^{\infty}_x} & \ls \int_0^{s_1} (s_1-s)^{-\frac{3}{4}} \vn{A}_{L^{\infty}_x}\vn{\nabla A }_{L^2_x}+ \vn{A}_{L^{\infty}_x}^3  \dd s \\
& \ls \int_0^{s_1} (s_1-s)^{-\frac{3}{4}}  s^{-\frac{1}{4}} / (N s^{\frac{1}{2}})^{2(1-\sigma)}  + s^{-\frac{3}{4}} / (N s^{\frac{1}{2}})^{3(1-\sigma)}   \dd s. 
\end{align*}
which proves \eqref{AbilL2Linf} after integration. 
\end{proof}

\end{lemma} 

\begin{remark} \label{rmkWP}
The same argument as in Proposition \ref{DeTurckProp} proves well-posedness of \eqref{DeTurck}  in $ \dot{H}^{\rho} $ for any $ \frac{1}{2} < \rho < 1 $. 
Instead of \eqref{IDxAeq} one considers $ (\pt_s - \Delta) \vm{D}^{\rho} A = \vm{D}^{\rho}( A \nabla A + A^3 ) $ together with the norm  $ \vn{A}_{S^{\rho}} \defeq \vn{\vm{D}^{\rho} A }_{ \tilde{S} } $. The estimates \eqref{DTbilest} are replaced by 
\be \label{Rmkprdest}
\begin{aligned}
\vn{ \vm{D}^{\rho}( f \nabla g)}_{L^2} \ls \vn{ f}_{\dot{H}^{\frac{3}{2}} \cap L^{\infty} } \vn{ g}_{\dot{H}^{1+\rho}} \\
\vn{ \vm{D}^{\rho} ( f_1 f_2 f_3)}_{L^2} \ls  \vn{f_1}_{\dot{H}^{1+\frac{\rho}{3}}}   \vn{f_2}_{\dot{H}^{1+\frac{\rho}{3}}}   \vn{f_3}_{\dot{H}^{1+\frac{\rho}{3}}}  
\end{aligned}
\ee
See \eqref{SobMultHomog}. Control in the norm $ \vn{ \vn{ \jb{2^{2k}s} \vn{P_k  A(s)}_{\dot{H}^{\rho} } }_{L^{\infty}_s} }_{\ell^2_{k }} $ holds as well.
One obtains 
\begin{proposition} \label{DeTurckWP}
Let $ \frac{1}{2} < \rho < 1 $. 
Suppose the initial data satisfies $ \bar{A}_i \in \ep \dot{H}^{\rho}(\mb{R}^3) $ and $ \ep>0 $ is small enough. Then there exists a solution to \eqref{DeTurck} on $ \mb{R}^3 \times [0,1] $ with $ A_i(s=0) =\bar{A}_i $ which is in $ S^{\rho} $ and is the unique limit of regular solutions. 
In particular, one has
\be \label{DeTurckBdd} \sup_{s \in [0,1]} s^{\frac{\rho_1 - \rho}{2}} \vn{ A_i(s)}_{\dot{H}^{\rho_1} } + \vn{s^{\frac{\rho_1 - \rho}{2}} A_i(s)}_{L^2_{\frac{\dd s}{s}} \dot{H}^{\rho_1}}  \ls \ep \qquad \forall  \ \rho_1 \in (\rho, \rho+1].    
\ee  
\end{proposition}
\end{remark}

\

Similarly to Proposition \ref{DeTurckWP} one has a statement for the equation satisfied by $ B_i = F_{0i} $ in DeTurck's gauge: 

\begin{proposition} \label{DeTurckF0i}
Let $ -  \frac{1}{2} < \gamma <0 $. Suppose the connection coefficients $ A_i $ satisfy the conditions in Proposition \ref{DeTurckWP} and Remark \ref{rmkWP} on $ \mb{R}^3 \times [0,1] $. Suppose the initial data satisfies $  \bar{B}_i \in \ep \dot{H}^{\gamma}(\mb{R}^3) $ and $ \ep>0 $ is small enough. Then the schematic equation 
\be \label{BiEq}
(\pt_s - \Delta) B = A \nabla B + B \nabla A  +   A^2 B 
\ee
has a unique solution on $ \mb{R}^3 \times [0,1] $ with $ B_i(s=0) =\bar{B}_i $ and satisfies 
$$ \sup_{s \in [0,1]} s^{\frac{\gamma_1 - \gamma}{2}} \vn{ B_i(s)}_{\dot{H}^{\gamma_1} } + \vn{s^{\frac{\gamma_1 - \gamma}{2}} B_i(s)}_{L^2_{\frac{\dd s}{s}} \dot{H}^{\gamma_1}}  \ls \ep \qquad \forall  \ \gamma_1 \in (\gamma, \gamma+1].    $$  
\end{proposition}

\begin{proof}
As before, the solution is obtained by Picard iteration. One uses the parabolic energy estimate \eqref{ParabolicEnEst} to estimate the solution of the equation 
$$
(\pt_s - \Delta) \vm{D}^{\gamma} B = \vm{D}^{\gamma}( A \nabla B + \nabla A  B +   A^2 B )
$$
in the norm $ \vn{B}_{S^{\gamma}} \defeq \vn{\vm{D}^{\gamma} B }_{ \tilde{S} } $. Recall \eqref{SobMultHomog}. The estimates we need are 
\begin{align*}
\vn{ A \nabla B}_{\dot{H}^{ \gamma }}+ \vn{ B \nabla A}_{\dot{H}^{ \gamma }} & \ls \vn{ A}_{\dot{H}^{\frac{3}{2}} \cap L^{\infty} } \vn{ B}_{\dot{H}^{1+\gamma}} \\
\vn{ A^2 B }_{\dot{H}^{ \gamma }} & \ls  \vn{A}_{\dot{H}^{1}}^2    \vn{B}_{\dot{H}^{1+\gamma}}  
\end{align*}
One obtains control in the norm $ \vn{ \vn{ \jb{2^{2k}s} \vn{P_k  B(s)}_{\dot{H}^{\gamma} } }_{L^{\infty}_s} }_{\ell^2_{k }} $ as well.
\end{proof}

\

Now we can obtain estimates for the curvature at the initial time $ t=0 $.

\begin{corollary} \label{CorInitData}
Suppose $ A_i(t,x) $ is a regular solution to the Yang-Mills equation in the temporal gauge on $  [-1,1] \times \mb{R}^3 $ with initial data $ ( \bar{A}_i,  \bar{E}_i) \in \ep ( \dot{H}^{\sg} \times \dot{H}^{\sg-1} )(\mb{R}^3) $ and $  \bar{A}_i \in  \dot{H}^{\frac{1}{2}+}(\mb{R}^3)  $.  

Suppose $ A_{\al} (t,x,s) $ solves \eqref{dYMHF} on $  [-1,1] \times \mb{R}^3 \times [0,1]$ in DeTurck's gauge $ A_s = \pt^{\ell} A_{\ell} $. Then, assuming $ \ep \ll 1 $, at $ t=0 $ one has 
\be \label{L2InitDataF}
 \vn{ s^{\frac{1 - \sg}{2}} F_{\al \beta}(s)}_{L^{\infty}_s L^2_x} + \vn{s^{\frac{1 - \sg}{2}} F_{\al \beta}(s)}_{L^2_{\frac{\dd s}{s}} L^2_x}  \ls \ep
\ee
As a consequence, at $ t=0 $, one has the gauge-invariant version
$$    \sup_{s \in [0,1]} s^{1 - \sg} \int_{\mathbb{R}^{3}} (F_{\alpha \beta}, F_{\alpha \beta})(s)  \dd x  + \int_0^1 s^{1 - \sg} \int_{\mathbb{R}^{3}} (F_{\alpha \beta}, F_{\alpha \beta})(s)  \dd x  \frac{\dd s}{s} \ls \ep^2.
$$
\end{corollary}
\begin{proof}
Consider $ F_{ij}(s) $ first. For $ \pt_i A_j $ and $ \pt_j A_i $ we use Proposition \eqref{DeTurckWP} with $ \rho = \sigma $ and $ \rho_1 = 1 $. For $ [A_i,A_j ] $ we use that from \eqref{DeTurckBdd} and applying again Proposition \eqref{DeTurckWP} with $ \rho=\frac{1}{2}+ $ one has
$$
s^{\frac{1 - \sg}{2}}  s^{-\frac{\ep'}{2}} \vn{ A_i(s) }_{\dot{H}^{1-\ep'}} \ls \ep,  \qquad \vn{ A_i(s) }_{ \dot{H}^{\frac{1}{2}+} } \ls 1.
$$
 The bounds for $ [A_i,A_j ] $ follow from this with H\"older and Sobolev embeddings 
$ 
\dot{H}^{1-} \times \dot{H}^{\frac{1}{2}+} \subseteq L^{6-} \times L^{3+} \subseteq L^2.
$

Now consider $  F_{0i}(s) $, which solves the covariant equation \eqref{0CovPara}. Expanding it and canceling the $ [A_s, F ]$ term with the $ [  \pt^{\ell} A_{\ell}, F ] $ term one obtains equation \eqref{FeqDeTurck}, which can be written as \eqref{BiEq}. Then \eqref{L2InitDataF} is concluded by  Proposition \ref{DeTurckF0i} with $ \gamma = \sg-1 $ and $ \gamma_1 = 0 $. 
\end{proof}

\

The next goal is to switch to caloric gauges which satisfy $ A_s = 0 $. This is accomplished by a change of gauge (DeTurck's trick) obtained by solving the ODE \eqref{CaloricODE}, see Proposition \ref{ChangeCaloricDeTurck}. 

\begin{lemma} \label{ODEBdds}
Let $ A_i, A_s= \pt^{\ell} A_{\ell} $ be given by Proposition \ref{DeTurckProp}. Let $ U(s) $ be the solution of the following ODE
\be \label{CaloricODE}
\pt_s U = U A_s, \qquad U(s=0)=1 
\ee
Then one has
\be \label{Ubds}
U - 1 \in \ep L^{\infty}_s L^{\infty}_x, \qquad \nabla_x U \in \ep L^{\infty}_s L^{2}_x, \qquad I \Delta U \in \ep L^{\infty}_s L^{2}_x
\ee
and $ U^{-1} $ obeys the same bounds. If $ A_i $ is a regular solution then $ U $ is a regular gauge transformation. 
\end{lemma}

\begin{proof}
Write the integral form 
\be \label{IntForm}
U(s_1) = 1+ \int_0^{s_1} U A_s \dd s 
\ee 
The first two bounds in \eqref{Ubds} follow from the first two bounds in \eqref{AsDTbds} using Gronwall's inequality. Similarly one shows $ U-1 \in \ep L^{\infty}_s H^{\frac{3}{2}+} $. 
 
The third bound in \eqref{Ubds} is more delicate. We use an argument from \cite[Theorem 7]{klainerman1995finite} and \cite[ Appendix B]{oh2014gauge}. One writes
$$
I \Delta U(s_1) = I \int_0^{s_1} \Delta U A_s  + 2 \nabla U \nabla A_s + U \Delta A_s    \dd s 
$$ 
The first two terms in the integral are handled by Gronwall's inequality and by
\begin{align} \label{prodIL2}
\vn{I(fg)}_{L^2} \ls \vn{If}_{L^2}  \vn{g}_{H^{\frac{3}{2}+}} \\
\vn{I(fg)}_{L^2} \ls \vn{I f}_{H^{1}} \vn{g}_{H^{\frac{1}{2}+}} 
\end{align}
for $ f=\Delta U, \ g=A_s $, respectively $ f= \nabla U, \ g=\nabla A_s $, provided we can also bound the third term.

For the third term one plugs in $ U(s) $ from \eqref{IntForm} and by Fubini's theorem write
$$
\int_0^{s_1} I \Delta A_s    \dd s  + I \int_0^{s_1} U(s') A_s(s')   \lpr \int_{s'}^{s_1} \Delta A_s(s)  \dd s  \rpr \dd s'
$$
Bound this in $ L^2_x $ using \eqref{IDeltaAs} for the first integral, then use \eqref{prodIL2} 
for $ g= U(s') A_s(s')  $ and $ f= \int_{s'}^{s_1} \Delta A_s(s)  \dd s \in \ep I^{-1} L^2  $ by \eqref{IDeltaAs}. Note that $ U A_s \in \ep L^1_{s'} H^{\frac{3}{2}+} $ by using the algebra property for $ H^{\frac{3}{2}+} $. 
\end{proof}

\subsection{Caloric gauge solutions} \

Setting the caloric condition $ A_s = 0 $ in the covariant \eqref{cYMHF} and the dynamic \eqref{dYMHF} Yang-Mills heat flows one obtains the initial value problems 

\be  \label{IVPcHF}
\pt_s A_i = \mathbf{D}^{\ell} F_{\ell i}, \qquad A_i(s=0)= \bar{A}_i 
\ee
\be  \label{IVPdHF}
\pt_s A_{\al} = \mathbf{D}^{\ell} F_{\ell \al}, \qquad A_{\al}(s=0)= \bar{A}_{\al}
\ee

\begin{theorem}[Existence and uniqueness for regular solutions, \cite{oh2015finite}, \cite{rade1992yang}] \label{ThmRegSol} \

(1) Consider the initial value problem \eqref{IVPcHF}, i.e. the \eqref{cYMHF} in the caloric gauge $ A_s = 0$, with regular initial data $  \bar{A}_i \in H^{\infty}_x $ (in particular having finite magnetic energy). Then there exists a unique global regular solution $ A_i \in C^{\infty}_s  H^{\infty}_x ( \mb{R}^3 \times [0,\infty)) $. 

(2) Consider the initial value problem \eqref{IVPdHF}, i.e. the \eqref{dYMHF} in the caloric gauge with regular initial data $  \bar{A}_{\al} \in C^{\infty}_t H^{\infty}_x (I \times \mb{R}^3) $ where $ I $ is a time interval. Then there exists a unique regular solution $ A_{\al} \in C^{\infty}_{t,s}  H^{\infty}_x ( I \times \mb{R}^3 \times [0,1]) $. 
\end{theorem}

For part (1) see \cite[Corollary 6.7]{oh2015finite} for the statement and proof in the form stated here, of the similar result due to \cite{rade1992yang} in a different context. 

Part (2) is also from \cite{oh2015finite}. Like in the first part, no smallness assumption is required. We sketch the details of part (2) from \cite[Thm 6.10, Lemma 6.9]{oh2015finite} and \cite[Step 1 in Thm 4.8, Prop. 5.7, Lemma 6.1]{oh2014gauge} for the sake of completeness:

The spatial components $ A_i $ satisfy \eqref{IVPcHF} and are obtained from part (1), so it remains to determine $ A_0 $. This can be done from the equation $ \pt_s A_{0} = \mathbf{D}^{\ell} F_{\ell 0} $ if one can determine what $ F_{\ell 0} $ should be, knowing that they have to solve the linear system with smooth coefficients 
$$
(\pt_s - \Delta_A) B_{\ell} = 2 [F_{\ell}^{\ j}, B_j ], \qquad B_{\ell}(0) = \bar{F}_{\ell 0} 
$$
Proposition 5.7 in \cite{oh2014gauge}  provides a unique regular solution $ B $, which is used to define regular $ A_0 $ from the initial data $ \bar{A}_0 $ by integrating
$ \pt_s A_{0} =  \mathbf{D}^{\ell} B_{\ell} $. Then one obtains curvatures $ F_{\ell 0}= \pt_{\ell} A_0 - \pt_t A_{\ell} + [A_{\ell}, A_0] $ which turn out to be equal to $ B_{\ell} $, as guaranteed by \cite[Lemma 6.1]{oh2014gauge}. Therefore, the solution satisfies \eqref{IVPdHF}. Uniqueness is proved in \cite[Lemma 6.9]{oh2015finite}.

\

\

Now we transfer $ I^{-1} H^1 $ bounds from DeTurck's gauge to the Caloric gauge under a smallness assumption. 

\begin{proposition} \label{ChangeCaloricDeTurck}
Let $ A_i $ be the unique regular solution of \eqref{IVPcHF} obtained in Theorem \ref{ThmRegSol} from an initial data $  \bar{A}_i  $ assumed to be regular and to satisfy $ \vn{I  \bar{A}_i }_{H^1} \leq \ep $. Then, if $ \ep $ is small enough, one has 
\be \label{IAbdd}
\vn{I A_i}_{  L^{\infty}_s H^1 (\mb{R}^3 \times [0,1])} \ls \ep .
\ee
Moreover, $ A_i $ is equivalent, using a regular gauge transformation, to 
\be \label{UchangeCaloricDeTurck}
A_i' = U^{-1} A_i U -  \partial_i U^{-1} U, \qquad  U(s=0)=1 
\ee
the regular solution to \eqref{cYMHF} in DeTurck's gauge $ A'_s = \pt^{\ell} A'_{\ell} $ with initial data $  \bar{A}_i  $ obtained in Proposition \ref{DeTurckProp}.
\end{proposition}

The proof follows the strategy of \cite[Theorem 5.1]{oh2014gauge} and consists of a version of \emph{DeTurck's trick} \cite{deturck1983deforming}, \cite{donaldson1985anti}, see \cite[Remark 5.9]{oh2014gauge}.

\begin{proof}
 Let $ A' $ be the regular solution to \eqref{cYMHF} in DeTurck's gauge $ A'_s = \pt^{\ell} A'_{\ell} $ with initial data $  \bar{A}_i  $ obtained in Proposition \ref{DeTurckProp} and which satisfies $ I A_i' \in \ep L^{\infty}_{s} H^1( \mb{R}^3 \times [0,1]) $. Making the gauge transformation
\be \label{gtransf}
\tilde{A}_i  =  U A_i' U^{-1} - \partial_i U U^{-1}, \qquad \tilde{A}_s  =  U A_s' U^{-1} - \partial_s U U^{-1}
\ee 
on $ \mb{R}^3 \times [0,1] $ one obtains a caloric solution $ \tilde{A}_i, \ \tilde{A}_s=0 $ with $ \tilde{A}_i (s=0) = \bar{A}_i  $ provided $ U(s) $ solves the following ODE pointwise in $ x $ 
$$
\pt_s U = U A_s', \qquad U(s=0)=1 
$$
Since $ A_s' $ is regular, also $ U $, and therefore $  \tilde{A}_i $ are. From the uniqueness part in Theorem \ref{ThmRegSol} we deduce that $ \tilde{A}_i = A_i $.

By Lemma \ref{ODEBdds}, $ U $ and $ U^{-1} $ satisfy the bounds \eqref{Ubds}, and therefore $ U-1 , U^{-1} - 1 \in \ep  L^{\infty}_{s} X $ where $ X $ is defined in \eqref{XnormDef}. 
Using \eqref{gtransf}, $ I A_i' \in \ep L^{\infty}_{s} H^1 $ and the product estimate \eqref{Algebraproducts} one obtains \eqref{IAbdd}. 
\end{proof}

\

\begin{remark} \label{RkChangeCaloricDeTurck}
Suppose $ A_{\al}(t,x,s) $ is the the unique regular solution of \eqref{IVPdHF} obtained in Theorem \ref{ThmRegSol} Part(2) on an interval $ [t_0,t_1] \times \mb{R}^3 \times [0,1] $ from regular and temporal initial data $  \bar{A}_{\al} \in C^{\infty}_t H^{\infty}_x ([t_0,t_1] \times \mb{R}^3) $, $ \bar{A}_{0} = 0 $  satisfying $ \vn{I  \bar{A}_i }_{L^{\infty}_t H^1} \leq \ep $. 
Then we may apply Proposition \ref{ChangeCaloricDeTurck} for each $ t \in [t_0, t_1] $ obtaining a regular $ U(t,x,s) $ with $ U(t,x,s=0) =1 $ and \eqref{IAbdd} holds uniformly in $ t $. In addition to \eqref{UchangeCaloricDeTurck} one has 
\be \label{UchangeCaloricDeTurck2}
A_{\al}' = U^{-1} A_{\al} U -  \partial_{\al} U^{-1} U, \quad F_{\al \beta}' = U^{-1} F_{\al \beta} U,   \qquad  U(t,s=0)=1 
\ee
on $ [t_0,t_1] \times \mb{R}^3 \times [0,1] $ where $ A_{\al}' $ is the solution to \eqref{dYMHF} in DeTurck's gauge, with (temporal) initial data  $ \bar{A}_{\al} $. 
\end{remark}

\

Finally, one has improved well-posedness in Sobolev spaces conditional on the smallness of the (almost conserved) modified energy. 

\begin{proposition} \label{ImprovedWPwAC}
Suppose $ A_i $ is the (caloric) regular solution of the problem \eqref{IVPcHF} from Theorem \ref{ThmRegSol} (1) for regular initial data $  \bar{A}_i $ with $ \vn{\bar{A}_i }_{H^{\sg}} \leq M_0 $ and 
\be \label{bddFxN}
\sup_{0 \leq m \leq 3} \sup_{s \in [0,s_0]}  (N s^{\frac{1}{2}})^{1-\sg} \vn{ (s^{\frac{1}{2}} \mathbf{D}_x)^m F(s) }_{L^2(\mb{R}^3)} \ll 1 
\ee
Then one has $  \vn{A_i (s) }_{H^{\sg}} \leq C(M_0) $ for all $ s \in [0,s_0] $, $s_0=N^{-2}$.
\end{proposition}

\begin{proof}

First note that we have local wellposedness for the equation in DeTurck's gauge \eqref{DeTurck}. This means a (regular) solution $ \bar{A}_i (s) $ to \eqref{DeTurck} exists on $ [0,s^*] $ with $  \bar{A}_i (0) = \bar{A}_i $, $ \vn{\bar{A}_i (s)}_{H^{\sg}} \ls M_0 $ for all $ s \in [0,s^*] $ with $ s^* \simeq M_0^{-p_{\sg}} $ where $ p_{\sg} = \frac{2}{ \sg-1/2} $. 
This is proved by the argument in Proposition \ref{DeTurckProp} and Remark \ref{rmkWP}, one just replaces \eqref{Rmkprdest} by their homogeneous version and one uses the smallness of the interval instead of smallness of the initial data. 

We next consider the ODE \eqref{CaloricODE} for $ \bar{A} $, i.e. $ \pt_s U= U \bar{A}_s $, $ U(0) =1 $ where $ \bar{A}_s = \pt^{\ell} \bar{A}_{\ell} $, which defines the gauge transformation from  DeTurck's gauge to the caloric gauge. The $ A_i $ and $ \bar{A}_i $ are related by
 $$
A_i  =  U \bar{A}_i U^{-1} - \partial_i U U^{-1}, \qquad A_s  =  U \bar{A}_s U^{-1} - \partial_s U U^{-1}=0
$$ 
Now \eqref{Ubds} is replaced by: $ U \in C(M_0) L^{\infty}_s L^{\infty}_x $ and $ \nabla_x U \in  C(M_0) L^{\infty}_s H^{\sg} $ on $ [0,s^*] $, and the same for $ U^{-1} $. We denote by $ C(M_0) $ a constant that can change from line to line. 
By \eqref{Algebraproducts1} one obtains $ \vn{A_i (s)}_{H^{\sg}} \ls C(M_0)$ for  $ s \in [0,s^*] $. 

It remains to extend this bound to $ [0,s_0] $ using \eqref{bddFxN}. Let $ p \in [4,6] $ be the Sobolev exponent such that $ H^{\sg} \subset L^p $. From \eqref{IVPcHF} and the covariant Gagliardo-Nirenberg inequality \eqref{SobGN} one has $ \vn{\pt_s A_i (s)}_{L^p} \ls s^{-1} (N s^{\frac{1}{2}})^{\sg-1} $. Integrating one has 
$$
 \vn{ A_i (s)}_{L^p} \ls  \vn{ A_i (s^*)}_{H^{\sg}} + (N s^{*\frac{1}{2}})^{\sg-1} \ls C(M_0), \qquad s \in [s^*,s_0]. 
$$ 
Now we let $ 1/q = 1/2-1/p $, $ q \in [3,4] $  and bound
\begin{align*}
 \vn{ \pt_s A_i (s)}_{H^1} & \ls \vn{\mathbf{D}_x F(s)}_{L^2} + \vn{\mathbf{D}_x^2 F(s)}_{L^2}+  \vn{ A_i (s)}_{L^p}  \vn{\mathbf{D}_x F(s)}_{L^q}   \\
 & \ls s^{-1} (N s^{\frac{1}{2}})^{\sg-1} + C(M_0) s^{-1} (N s^{\frac{1}{2}})^{\sg-1}
 \end{align*}
Integrating again one has 
$$
\vn{ A_i (s)}_{H^{\sg}} \leq \vn{ A_i (s^*)}_{H^{\sg}} + \int_{s^*}^s \vn{ \pt_s A_i (s')}_{H^1}  \dd s' \ls C(M_0) (N s^{*\frac{1}{2}})^{\sg-1} \ls C(M_0) 
$$
for all $ s \in [s^*,s_0] $. 
 \end{proof}


\section{Energy estimates for the curvature and the tension field} \label{EnEstSec}

The goal of this section is to obtain weighted bounds for higher covariant derivatives for the curvature $ F_{\al \beta}(s) $ and the tension field $ w_x(s) $, first globally in space and then locally - with square summability of the local bounds, proving Proposition \ref{PropCoeff}. All estimates in this section are covariant, independent of the choice of gauge. 

\

Let $ A_{t,x,s} $ be a regular solution to \eqref{dYMHF} on an interval $ I \times \mb{R}^3 \times [0,s_0] $. 

Then the curvature satisfies 
\be \label{0CovPara}
\left(\mathbf{D}_{s}-\mathbf{D}^{\ell} \mathbf{D}_{\ell}\right)  F_{\al \beta} = 2 [F_{\al}^{\ \ell}, F_{\ell \beta}]
\ee
This is obtained from Bianchi's identity \eqref{Bianchi}, commuting the covariant derivatives by \eqref{CDComm} and applying Bianchi's identity again. 

For higher covariant derivatives one computes (see \cite[Section 5]{oh2015finite} for details):
 \be \label{mCovPara}
 \left(\mathbf{D}_{s}-\mathbf{D}^{\ell} \mathbf{D}_{\ell}\right) \mathbf{D}_{x}^{(m)} F=  \sum_{i=0}^{m} \mathcal{O} \lpr \mathbf{D}_{x}^{(i)} F, \mathbf{D}_{x}^{(m-i)} F \rpr
 \ee

Now let $ w_k(s) $ be the Yang-Mills tension field defined in \eqref{wDef}. It satisfies 

\be \label{0CovParaw}
\lpr \mathbf{D}_{s} -\mathbf{D}^{\ell} \mathbf{D}_{\ell} \rpr w_{k}=2\left[F_{k}^{\ell}, w_{\ell}\right]+2\left[F^{0 \ell}, \mathbf{D}_{k} F_{0 \ell}+2 \mathbf{D}_{\ell} F_{k 0}\right]
\ee

\

See \cite[Eq. (8.8)]{oh2017yang}, \cite[Appendix A]{oh2014gauge}.
For higher covariant derivatives one has
\be \label{mCovParaw}
\lpr \mathbf{D}_{s} -\mathbf{D}^{\ell} \mathbf{D}_{\ell}  \rpr  \mathbf{D}_{x}^{(m)}w_{k}=  
\sum_{i=0}^{m} \mathcal{O} \lpr \mathbf{D}_{x}^{(i)} F, \mathbf{D}_{x}^{(m-i)} w_x \rpr
+ \mathcal{O} \lpr \mathbf{D}_{x}^{(i)} F, \mathbf{D}_{x}^{(m+1-i)} F \rpr
\ee

\subsection{Covariant energy estimates and the comparison principle} \label{CovEnEst} \ 

Suppose $ G $ solves the covariant heat equation 
$$
\left(\mathbf{D}_{s}-\mathbf{D}^{\ell} \mathbf{D}_{\ell}\right) G = \mathcal{N} 
$$
on an interval $ \mathbb{R}^3 \times [s_1,s_2] $ and let $ p>0 $. Then one has 
\begin{align*}
\sup_{s \in [s_1,s_2]} s^{2p} \int_{\mb{R}^3} (G,G)(s) \dd x + \int_{s_1}^{s_2} s^{2p}  \int_{\mb{R}^3}  (\mathbf{D}^{\ell} G, \mathbf{D}_{\ell} G)(s) \dd x  \dd s \ls \\
s_1^{2p} \int_{\mb{R}^3} (G,G)(s_1) \dd x + p \int_{s_1}^{s_2} s^{2p-1 } \int_{\mb{R}^3} (G,G)(s) \dd x \dd s + \lpr \int_{s_1}^{s_2} s^p \vn{ \mathcal{N}(s)}_{L^2_x}    \dd s \rpr ^2  
\end{align*}
or, in short, denoting $ J=[s_1,s_2] $,  
\be  \label{EnergyEst}
\vn{s^p G}_{L^{\infty}_{s} L^2_x (J) }  +  \vn{s^{p} s^{\frac{1}{2}} \mathbf{D}_x G}_{L^2_{ \frac{\dd s}{s}} L^2_x (J) } \ls s_1^p \vn{G(s_1)}_{L^2_x} + p \vn{s^p G}_{L^2_{ \frac{\dd s}{s}} L^2_x (J) }    + \vn{s^p \mathcal{N} }_{L^1_{s} L^2_x(J) }  
\ee 
This is obtained by applying the Fundamental theorem of calculus on an interval to the function $ \varphi(s) = \frac{1}{2} s^{2p} \int_{\mb{R}^3} (G,G)(s) \dd x $, using the equation, integrating by parts, applying Cauchy-Schwarz and absorbing a term into the LHS. See \cite[Lemma 4.5]{oh2015finite} for more details. 

\

As a consequence of the diamagnetic-type inequality (\cite[Lemma 4.6]{oh2015finite}, \cite[Lemma 3.9]{tao2008global})
$$
( \pt_s - \Delta) \vm{G} \leq \vm{ \mathcal{N} }
$$
and the Duhamel principle, \cite[Corollary 4.7]{oh2015finite} formulates the comparison principle:

\begin{lemma} \label{Comparison}
The following inequality holds point-wise 
$$
\vm{G(s)} \leq  e^{s \Delta} \vm{G(0)} + \int_0^{s} e^{(s-s') \Delta} \vm{ (\mathbf{D}_{s} -\mathbf{D}^{\ell} \mathbf{D}_{\ell}) G(s') }  \dd s'. 
$$
\end{lemma}

\

We will also use the following estimate for Duhamel terms 

\begin{lemma}[ \cite{oh2015finite}] \label{lemmashur}
For $ p< \frac{3}{4} $ and $ s_0 \leq 1 $ the following estimate holds on $ [0,s_0] \times \mb{R}^3 $:
$$
\vn{s_1^p \int_0^{s_1} e^{(s_1-s) \Delta} \mathcal{N}(s) \dd s }_{L^2_{ \frac{\dd s_1}{s_1}} L^2_x}  
\ls  \vn{s^{p+\frac{1}{4}} \mathcal{N}(s) }_{L^2_{ \frac{\dd s}{s}} L^1_x}
$$
\end{lemma}
See \cite[Lemma 4.8]{oh2015finite}.

\subsection{Higher covariant derivatives bounds for $ F $ and $ w $} \ 

 Denote
 $$
 \al^m(s') \defeq  \vn{  (N s^{\frac{1}{2}})^{1-\sigma} (s^{\frac{1}{2}} \mathbf{D}_x)^m F_{\al \beta}(s) }_{L^{\infty}_s L^2_x ( [0,s'] \times \mb{R}^3) } 
 $$
 $$
 \beta^m(s') \defeq  \vn{  (N s^{\frac{1}{2}})^{1-\sigma} (s^{\frac{1}{2}} \mathbf{D}_x)^m F_{\al \beta}(s) }_{L^{2}_{\frac{\dd s}{s}} L^2_x ([0,s'] \times \mb{R}^3)}  
 $$

\begin{proposition} \label{CovBdsF1}
Let $ A_{t,x,s} $ be a regular solution to \eqref{dYMHF} on an interval $ J \times \mb{R}^3 \times [0,s_0] $. Fix a time $ t \in J $ and denote $ F(s)=F(t,s) $. Let $ \eta \ll 1 $ and suppose 
 $$
  \al^0(s_0) + \beta^0(s_0) \ls \eta  
 $$ 
 Then, for $ N^2 s_0 = 1$ and for all $ 0 \leq m \leq 10 $ one has
 $$
  \al^m(s_0) + \beta^{m+1}(s_0) \ls_m \eta   
 $$
\end{proposition}

\begin{proof}

Denote by $ \mathcal{N}^m $ the RHS of \eqref{mCovPara} and let 
$$
\gamma^m (s') \defeq  \vn{  (N s^{\frac{1}{2}})^{1-\sigma} s^{\frac{m}{2}} \vn{ \mathcal{N}^m(s)}_{L^2_x} }_{L^1_s [0,s']}  
$$
Applying \eqref{EnergyEst} to \eqref{mCovPara} with $ p=\frac{m}{2} + \frac{1-\sg}{2} $ on $ [0,s] $ and multiplying by $ N^{1-\sg} $ one obtains
$$
 \al^m(s) + \beta^{m+1}(s) \ls \beta^{m}(s) + \gamma^m (s) 
 $$
First let $ m=0 $. Using Holder's inequality and \eqref{SobGN} one has
$$
\vn{[F,F](s)}_{L^2_x} \ls \vn{F(s)}_{L^2_x}^{\frac{1}{2}} \vn{\mathbf{D}_x F(s)}_{L^2_x}^{\frac{3}{2}}
$$
which implies that one can bound 
$$
\beta^1 (s) \ls \eta + \gamma^1 (s) \ls \eta + \eta^{\frac{1}{2}} \frac{s^{1/4}}{(N s^{\frac{1}{2}})^{1-\sigma}} \lpr \beta^1 (s) \rpr^{\frac{3}{2}}
$$
Since $ \lim_{s \to 0} \beta^1 (s)=0 $ one can use a continuity argument to obtain $ \beta^1 (s_0) \ls \eta  $. 
Then one proceeds by induction on $ m $, similarly bounding 
$$
\al^m(s) + \beta^{m+1}(s) \ls \beta^{m}(s) + \gamma^m (s) \ls  \eta + \eta^{\frac{1}{2}} \frac{s^{1/4}}{(N s^{\frac{1}{2}})^{1-\sigma}} \lpr \beta^{m+1} (s) \rpr^{\frac{3}{2}}
$$
and concluding like before. 
\end{proof}

Now we turn to the Yang-Mills tension field $ w_x $.

\begin{proposition} \label{CovBdsw1}
Given \eqref{YM} solutions  $ A_{t,x,s} $ and $ F_{\al \beta} $ satisfying the assumptions from Proposition \ref{CovBdsF1}, for $ w_k(s) \defeq \mathbf{D}^{\al} F_{\al k}(s) $ and $ 0 \leq m \leq 8 $ one has 
\be \label{derbdsw}
 \vn{  (N^2 s)^{1-\sigma} s^{\frac{1}{4}} (s^{\frac{1}{2}} \mathbf{D}_x)^m w_x(s) }_{L^{\infty}_s L_x^2   \cap L^2_{\frac{\dd s}{s}} L^2_x ( [0,s_0] \times \mb{R}^3) }  \ls \eta^2
\ee
Moreover,
\be \label{L1bdw}
(N^2 s)^{1-\sigma} \vn{w_x(s)}_{L^1_x} \ls s^{\frac{1}{2}} \eta \qquad \forall s \in [0,s_0].
\ee
\end{proposition}

\begin{proof}
Denote by $ \delta_m $ and $ \rho_m $ the $ L^{\infty}_s L_x^2 $, respectively the $ L^2_{\frac{\dd s}{s}} L^2_x $ norms from \eqref{derbdsw}. Note that $  \mathbf{D}_x^m w(s=0) = 0 $ because $ F $ is assumed to solve the Yang-Mills equation at $ s=0 $. 

\

We begin by showing $ \rho_0 \ls \eta^2 $. 

We use equation \eqref{0CovParaw}, the fact that $ w_x(s=0)=0 $, the Comparison principle in Lemma \ref{Comparison}, Lemma \ref{lemmashur} with $ p=\frac{1}{4}+1-\sg $ and the following estimates 
$$
\vn{ (N^2 s)^{1-\sigma} s^{\frac{1}{2}}   [F, w_x ] }_{L^2_{ \frac{\dd s}{s}} L^1_x} \ls \vn{ (N s^{\frac{1}{2}})^{1-\sigma} F}_{L^{\infty}_s L^2_x} \vn{s^{\frac{1}{4}} (N s^{\frac{1}{2}})^{\sigma-1}}  _{L^{\infty}_s } \rho_0  \ls  \eta s_0^{\frac{1}{4}}  \rho_0 
$$
$$
\vn{ (N^2 s)^{1-\sigma} s^{\frac{1}{2}}   [F,\mathbf{D}_x F] }_{L^2_{ \frac{\dd s}{s}} L^1_x} \ls \vn{ (N s^{\frac{1}{2}})^{1-\sigma} F}_{L^{\infty}_s L^2_x}  \vn{(N s^{\frac{1}{2}})^{1-\sigma} s^{\frac{1}{2}} \mathbf{D}_x F}_{L^2_{ \frac{\dd s}{s}} L^2_x} \ls \eta^2
$$
to obtain 
$$
 \rho_0 \ls \eta s_0^{\frac{1}{4}}  \rho_0   + \eta^2. 
$$
and then we absorb the small $ \rho_0 $ term to the left side since $ \eta \ll 1 $. 

The same estimate for $  [F,\mathbf{D}_x F] $, but with a different factor of $ s $ and bounding in $ L^1_s $, together with the Comparison principle in Lemma \ref{Comparison} will prove \eqref{L1bdw}. Denoting $ M_1 \defeq \vn{w_x}_{L^{\infty}_s L_x^1 [0,s_1] }$, now one absorbs the following term to the LHS:
\be \label{abstrm}
\vn{ [F, w_x ]}_{L^1_s L^1_x} \ls M_1 \int_0^{s_1} \vn{F(s)}_{L^{\infty}_x} \dd s \ls M_1 s_1^{\frac{1}{4}} / (N^2 s)^{1-\sigma} \eta \ll M_1. 
\ee

\

Now we proceed by induction. We assume we have showed $ \rho_i \ls \eta^2 $ for all $ 0 \leq i \leq m $ and we prove 
$$ \delta_m + \rho_{m+1} \ls 
\rho_m + \vn{ (N^2 s)^{1-\sigma} s^{\frac{1}{4}} s^{\frac{m}{2}} \mathcal{N}_m(s) }_{L^1_s L^2_x} 
\ls \eta^2 
$$ 
where $  \mathcal{N}_m(s) $ is the RHS of  \eqref{mCovParaw}. The first inequality is the energy estimate \eqref{EnergyEst}. We will use 
$$
\vn{ \mathbf{D}_{x}^{i} F }_{L^{\infty}_x} \ls \left\|\mathbf{D}^{i+1}_{x} F \right\|_{L_{x}^{2}}^{1 / 2}\left\|\mathbf{D}^{i+2}_{x} F\right\|_{L_{x}^{2}}^{1 / 2}
$$
which follows from \eqref{SobGN}. 
We bound the first component of $  \mathcal{N}_m(s) $ 
$$
\vn{ (N^2 s)^{1-\sigma} s^{\frac{1}{4}} s^{\frac{m}{2}} \vn{ \mathbf{D}_{x}^{i} F }_{L^{\infty}} \vn{\mathbf{D}_{x}^{m-i} w_x }_{L^2}   }_{L^1_s} \ls \eta \vn{s^{\frac{1}{4}} (N s^{\frac{1}{2}})^{\sigma-1}} _{L^{\infty}_s }  \rho_{m-i} \ls \eta^3 s_0^{\frac{1}{4}} 
$$ 
and for the second component we may assume that $ i \leq m+1-i $ and estimate
$$
\vn{ (N^2 s)^{1-\sigma} s^{\frac{1}{4}} s^{\frac{m}{2}} \vn{ \mathbf{D}_{x}^{i} F }_{L^{\infty}} \vn{\mathbf{D}_{x}^{m+1-i} F }_{L^2}   }_{L^1_s} \ls \vn{(N s^{\frac{1}{2}})^{1-\sigma} ( s^{\frac{1}{2}} \mathbf{D}_x)^{i+1} F}^{1/2}_{L^2_{ \frac{\dd s}{s}} L^2_x} \times $$
$$
\vn{(N s^{\frac{1}{2}})^{1-\sigma} ( s^{\frac{1}{2}} \mathbf{D}_x)^{i+2} F}^{1/2}_{L^2_{ \frac{\dd s}{s}} L^2_x} \vn{(N s^{\frac{1}{2}})^{1-\sigma} ( s^{\frac{1}{2}} \mathbf{D}_x)^{m+1-i} F}_{L^2_{ \frac{\dd s}{s}} L^2_x}
\ls \eta^2 
$$
\end{proof}

\subsection{Defining the balls coefficients - Proof of Proposition \ref{PropCoeff}} \

Fix a time $ t \in J $ and denote $ F(s)=F(t,s) $. It is easily seen that the proof is uniform in $ t $ .
If one would only have $ L^{2} $-type norms in \eqref{BallHighCovDer1}, \eqref{BallHighCovDer4}  such as $ L^{2}_{\frac{\dd s}{s}} L^2_x  $ or $ L^2_x $ instead of $ L^{\infty}_s L^2_x $ the square summability would be obvious, because by Proposition \ref{CovBdsF1} we control all higher derivatives of $ F $ on $ [0,s_0]  \times \mb{R}^3 $: 
\be  \label{Highdereqq}
\vn{  (N s^{\frac{1}{2}})^{1-\sigma} (s^{\frac{1}{2}} \mathbf{D}_x)^m F(s) }_{L^{\infty}_s L^2 [0,s_0]}  +
\vn{  (N s^{\frac{1}{2}})^{1-\sigma} (s^{\frac{1}{2}} \mathbf{D}_x)^{m+1} F(s) }_{L^{2}_{\frac{\dd s}{s}} L^2 [0,s_0]} \ls \eta
\ee
and similarly for $ w $, due to Proposition \ref{CovBdsw1} we have  
\be \label{derbdsw2}
 \vn{  (N^2 s)^{1-\sigma} s^{\frac{1}{4}} (s^{\frac{1}{2}} \mathbf{D}_x)^m w_x(s) }_{L^{\infty}_s L_x^2   \cap L^2_{\frac{\dd s}{s}} L^2_x ( [0,s_0] \times \mb{R}^3) }  \ls \eta^2
\ee

To deal with the $ L^{\infty}_s L^2_x $ we use the Energy estimates from section \ref{CovEnEst}. 

Consider a collection of spatial bump functions $ (\chi_j)_{j \in \mathcal{J}} $ with finitely overlapping supports which satisfy $ 1 \ls \sum_{j \in \mathcal{J}} \chi_j^2(x) $. We construct square summable coefficients $ (c_j)_{j \in \mathcal{J}} $ satisfying 
\begin{align}
\label{BumpHighCovDer1}
\al_j^m \defeq  \vn{  \chi_j  (N s^{\frac{1}{2}})^{1-\sigma}(s^{\frac{1}{2}}   \mathbf{D}_x)^m F(s) }_{L^{\infty}_s L^2_x( [0,s_0]\times\mb{R}^3) } \leq c_j  \\
 \label{BumpHighCovDer2}
\beta_j^m \defeq \vn{  \chi_j  (N s^{\frac{1}{2}})^{1-\sigma} (s^{\frac{1}{2}}   \mathbf{D}_x)^m F(s) }_{L^{2}_{\frac{\dd s}{s}} L^2_x ( [0,s_0]\times \mb{R}^3)} \leq c_j   
\end{align}
and
\begin{align}
 \label{BumpHighCovDer3}
\delta_j^m \defeq  \vn{  \chi_j  (N^2 s)^{1-\sigma} s^{\frac{1}{4}} (s^{\frac{1}{2}} \mathbf{D}_x)^m w_x(s) }_{L^{\infty}_s L^2_x( [0,s_0]\times\mb{R}^3) } \leq c_j  \\
 \label{BumpHighCovDer4}
\rho_j^m \defeq \vn{  \chi_j  (N^2 s)^{1-\sigma} s^{\frac{1}{4}} (s^{\frac{1}{2}} \mathbf{D}_x)^m w_x(s) }_{L^{2}_{\frac{\dd s}{s}} L^2_x ( [0,s_0]\times \mb{R}^3)} \leq c_j 
\end{align}
for all $ 0 \leq m \leq 8 $ and then we take each $ \chi_j \equiv 1 $ on $ B_j $ to obtain the conclusion. 

It suffices to show 
\be  \label{sqsumm}
\sum_{j \in \mathcal{J}} (  \al_j^m )^2 \ls \eta^2, \quad \sum_{j \in \mathcal{J}} (  \beta_j^m )^2 \ls \eta^2, \quad
\sum_{j \in \mathcal{J}} (  \delta_j^m )^2 \ls \eta^2  , \quad \sum_{j \in \mathcal{J}} (  \rho^m )^2 \ls \eta^2 . \ee
Note that this is automatically true for the $ \beta_j^m $'s and $ \rho_j^m $'s so it remains to focus on the $ \al_j^m $'s and $ \delta_j^m $'s. 

\

From \eqref{0CovPara} we obtain the parabolic equation of $ \chi_j F $ 
$$
\left(\mathbf{D}_{s}-\mathbf{D}^{\ell} \mathbf{D}_{\ell}\right) \chi_j F = \mathcal{N}_j \defeq 2 \chi_j [F, F] - \Delta \chi_j F - 2 \partial_{\ell} \chi_j  \mathbf{D}^{\ell} F
$$
while from \eqref{0CovParaw} we obtain the equation of  $ \chi_j w_x $ 
$$ 
\lpr \mathbf{D}_{s} -\mathbf{D}^{\ell} \mathbf{D}_{\ell}  \rpr  \chi_j w_x =  \mathcal{N}_j' \defeq 2  \chi_j \left[F, w_{x}\right]+2 \chi_j \left[F, \mathbf{D}_{x} F \right] - \Delta \chi_j w_x - 2 \partial_{\ell} \chi_j  \mathbf{D}^{\ell} w_x
$$

\

By \eqref{EnergyEst} on $ [0,s_0] $, \eqref{SobGN}, \eqref{Highdereqq} and Holder in $ \frac{\dd s}{s} $ we obtain 
\begin{align*}
\al_j^0 & \ls \beta_j^0 + \vn{(N s^{\frac{1}{2}})^{1-\sigma} \mathcal{N}_j(s) }_{L^1_s L^2_x[0,s_0]} \\
& \ls \beta_j^0 + \vn{(N s^{\frac{1}{2}})^{1-\sigma} (\vn{\mathbf{D} F}_{L^2_x}^{\frac{1}{2}} \vn{\mathbf{D}^2 F}_{L^2_x}^{\frac{1}{2}} \vn{\chi_j F}_{L^2_x}  + \vn{\Delta \chi_j F}_{L^2_x}  + \vn{\nabla \chi_j  \mathbf{D} F}_{L^2_x} ) }_{L^1_s [0,s_0]} \\
&\ls \beta_j^0 + \eta   \beta_j^0 \vn{s^{\frac{1}{4}} (N s^{\frac{1}{2}})^{\sigma-1}}_{L^{2}_{\frac{\dd s}{s}}[0,s_0]}    + \tilde{\beta}_j^0 \vn{s}_{L^{2}_{\frac{\dd s}{s}}[0,s_0]}  + \bar{\beta}_j^1 \vn{s^{\frac{1}{2}}}_{L^{2}_{\frac{\dd s}{s}}[0,s_0]}  \\ 
&\ls \beta_j^0 + \tilde{\beta}_j^0 + \bar{\beta}_j^1
\end{align*}
where $ \tilde{\beta}_j^m $ and $ \bar{\beta}_j^m $ (respectively $ \tilde{\rho}_j^m $ and $ \bar{\rho}_j^m $ ) are defined to be the terms in  \eqref{BumpHighCovDer2}  (respectively \eqref{BumpHighCovDer4}) with $ \chi_j $ replaced by $ \Delta \chi_j $ and $ \nabla \chi_j $, which are also square summable, thus proving \eqref{sqsumm} for the $ \al_j^0 $'s. Similarly 
\begin{align*}
\delta_j^0 & \ls \rho_j^0 + \vn{  (N^2 s)^{1-\sigma} s^{\frac{1}{4}} \mathcal{N}_j'(s)   }_{L^1_s L^2_x[0,s_0]} \\
& \ls \rho_j^0 + \vn{  (N^2 s)^{1-\sigma} s^{\frac{1}{4}} ( \vn{\mathbf{D} F}_{L^2_x}^{\frac{1}{2}} \vn{\mathbf{D}^2 F}_{L^2_x}^{\frac{1}{2}} \vn{\chi_j w}_{L^2_x}  +   \vn{\mathbf{D} F}_{L^2_x}^{\frac{1}{2}} \vn{\mathbf{D}^2 F}_{L^2_x}^{\frac{1}{2}} \vn{\chi_j \mathbf{D} F}_{L^2_x}    \\
 & \qquad \ + \vn{\Delta \chi_j w}_{L^2_x}  + \vn{\nabla \chi_j  \mathbf{D} w}_{L^2_x} ) }_{L^1_s [0,s_0]} \\
&\ls \rho_j^0 + \eta   \rho_j^0 \vn{s^{\frac{1}{4}} (N s^{\frac{1}{2}})^{\sigma-1}}_{L^{2}_{\frac{\dd s}{s}}[0,s_0]}  + \eta \beta^1_j + \tilde{\rho}_j^0 \vn{s}_{L^{2}_{\frac{\dd s}{s}}[0,s_0]}  + \bar{\rho}_j^1 \vn{s^{\frac{1}{2}}}_{L^{2}_{\frac{\dd s}{s}}[0,s_0]}  \\ 
&\ls \rho_j^0 + \beta^1_j +  \tilde{\rho}_j^0 + \bar{\rho}_j^1
\end{align*}
This proves \eqref{sqsumm} for $ m=0 $.

\

The same argument holds for $ m \geq 1 $. By \eqref{mCovPara} the equations for higher derivatives of $ F $ are $ \left(\mathbf{D}_{s}-\mathbf{D}^{\ell} \mathbf{D}_{\ell}\right) ( \chi_j \mathbf{D}_{x}^{(m)} F)= \mathcal{N}_j^m $ where 
$$
 \mathcal{N}_j^m \defeq  \chi_j  \sum_{k=0}^{m} \mathcal{O} \lpr \mathbf{D}_{x}^{(k)} F, \mathbf{D}_{x}^{(m-k)} F \rpr  - \Delta \chi_j  \mathbf{D}_{x}^{(m)} F - \nabla  \chi_j \mathbf{D}_{x}^{(m+1)} F
$$
Applying \eqref{EnergyEst} with $ p=\frac{m}{2} + \frac{1-\sg}{2} $, multiplying by $ N^{1-\sg} $ and doing the same manipulations one obtains
$$
\al_j^m \ls \beta_j^m + \eta \sum_{k=0}^m \beta_j^k +  \tilde{\beta}_j^m + \bar{\beta}_j^{m+1}
$$
which is square-summable. Similarly one writes the equation of higher covariant derivatives of $ \chi_j w_x $ following from \eqref{mCovParaw} and estimate 
$$
\delta_j^m \ls \rho_j^m +  \eta \sum_{k=0}^m \rho_j^k  +  \eta \sum_{k=1}^{m+1} \beta_j^k   +   \tilde{\rho}_j^m + \bar{\rho}_j^{m+1}
$$
which concludes the proof.


\section{The local change of gauge: Proof of Proposition \ref{LocalGauge1} } \label{SecGauge}

This section is devoted to the proof of Proposition \ref{LocalGauge1}, which is an Uhlenbeck-type lemma for fractional regularities below $ L^2 $, as mentioned in Remark \ref{UhlenbeckType}. Unlike the previous section, which consisted of only gauge-invariant bounds, the goal here is the opposite: to be as specific as possible with the choice of gauge in order to obtain Sobolev bounds for $ A_i $, $ \pt_t A_i $ from the curvature bounds. 

\

We begin by recalling the classical result of Uhlenbeck \cite{uhlenbeck1982connections}, which played a crucial role in the proof of finite energy global well-posedness \cite{klainerman1995finite} stated in Theorem \ref{GWPKM}. 

\begin{theorem}[Uhlenbeck's lemma \cite{uhlenbeck1982connections}] \label{UhlenbecksL}
Let $ A_x :B\to \mathfrak{g} $ have curvature $ F_{ij} $ bounded in $ L^2(B) $. There exists a $ \delta > 0 $ such that if 
$$ \vn{F}_{L^{\frac{3}{2}}(B)} < \delta $$
then there exists $ U: B\to G $ such that the transformation 
\be \label{Gaugetransf}
\tilde{A}_i = U A_i U^{-1} - \partial_i U U^{-1}
\ee 
satisfies $ \partial^{\ell} \tilde{A}_{\ell} =0 $ in $ B$, $ x^{\ell} \tilde{A}_{\ell} =0 $ on $ \partial B$ and 
$$ \vn{ \tilde{A}}_{H^{1}(B)} \ls \vn{F}_{L^2(B)}. $$
\end{theorem} 

\

After smoothing the curvature $ F_{\al \beta} $ using the Yang-Mills heat flow, we will apply Uhlenbeck's lemma at $ s=s_0 $ to first obtain an $ H^1 $ bound. 

\

By elliptic regularity, one obtains additional bound on $ \tilde{A} $:

\begin{lemma} \label{AiElliptic}
Suppose $ \tilde{A}_i $ is like in Theorem \ref{UhlenbecksL} and assume the (gauge-invariant) bounds 
$$  \vn{ (s_0^{\frac{1}{2}} \tilde{\mathbf{D}}_x)^m \tilde{F} }_{L^2(B)} \leq \delta $$ 
for $ 0 \leq m \leq 2 $ and some $ s_0 \ll 1 $.  Then 
$$ s_0^{\frac{1}{4}} \vn{ \tilde{A}_i}_{L^{\infty}(B)} + s_0^{\frac{3}{4}} \vn{\pt_x \tilde{A}_i}_{L^{\infty}(B)} \ls \delta.     $$ 
\end{lemma}

\begin{proof}
First note that $ s_0^{\frac{1}{2}} \pt_x \tilde{F} \in \delta L^2(B) $. Indeed, by \eqref{CovGNonB} one has $ s_0^{\frac{1}{2}} \tilde{F} \in \delta L^6(B) $ and then $  s_0^{\frac{1}{2}} [ \tilde{A}, \tilde{F} ] \in \delta L^2(B) $.

From $ \partial^{\ell} \tilde{A}_{\ell} =0 $ one obtains the elliptic equation
\be \label{ElleqACol}
\Delta \tilde{A_j}  = \pt^i \tilde{F}_{ij} -  \pt^i [\tilde{A}_i, \tilde{A}_j ].
\ee 
We want to show $ s_0^{\frac{1}{2}} [\tilde{A}, \nabla \tilde{A}] \in  \delta L^2(B) $. We begin with some weaker bounds, like $  [\tilde{A}, \nabla \tilde{A}] \in  \delta L^{\frac{3}{2} } (B) $ which follows from H\"older and Sobolev embedding. Then $ s_0^{\frac{1}{2}} \tilde{A}_j \in \delta W^{2, \frac{3}{2} } (B) \subset \delta L^p(B) $ for all $ p < \infty $. This implies $  s_0^{\frac{1}{2}}  [\tilde{A}, \nabla \tilde{A}] \in  \delta  L^{2-}(B) $ and $ s_0^{\frac{1}{2}} \Delta \tilde{A} \in  \delta  L^{2-}(B) $ from which, together with $ \tilde{A} \in \delta W^{1,2}(B) \subset \delta L^6(B) $ and Gagliardo-Nirenberg inequalities one obtains $ s_0^{\frac{1}{4}+} \tilde{A}  \in \delta L^{\infty}(B) $. Now we obtain $ s_0^{\frac{1}{2}} [\tilde{A}, \nabla \tilde{A}] \in  \delta L^2(B) $ from which $ s_0^{\frac{1}{2}} \Delta \tilde{A} \in  \delta  L^{2}(B) $ and similarly conclude $ s_0^{\frac{1}{4}} \tilde{A}  \in \delta L^{\infty}(B) $.

For the second inequality note that $  s_0 \pt_x^2 \tilde{F} \in \delta L^2(B) $. Differentiating \eqref{ElleqACol} one obtains an elliptic equation for $ \pt_x \tilde{A}_j $. Then we conclude by a similar argument as before, we omit the details. 
\end{proof}

\subsection{Proof of Proposition \ref{LocalGauge1} Part (1) } \ 

{\bf Step 1.} (Change of gauge)
By integrating the equation $ \pt_s A_i = \mathbf{D}^{\ell} F_{\ell i}  $ one obtains 
$$
A_i (s) =  \int_s^{s_0} \mathbf{D}^{\ell} F_{i \ell }(s') \dd s' + A_i(s_0)  , \qquad s \in [0,s_0]. 
$$
Since $ N s^{\frac{1}{2}}_0 = 1 $ we have $  \vn{ F(s_0) }_{L^2(B)} \leq \delta $ and we can apply Uhlenbeck's lemma (Theorem \ref{UhlenbecksL}) at $ s=s_0 $ to obtain $  U: B\to G $ such that after the transformation \eqref{Gaugetransf} one has $  \vn{ \tilde{A}_i (s_0)}_{H^{1}(B)} \ls \delta $. Write
\begin{align*}
\tilde{A}_i (s) & =  \int_s^{s_0} U  \mathbf{D}^{\ell} F_{i \ell } U^{-1} \dd s'      + U  A_i(s_0) U^{-1} - \partial_i U U^{-1} \\ 
& = \int_s^{s_0} \tilde{\mathbf{D}}^{\ell} \tilde{F}_{i \ell }(s') \dd s' +  \tilde{A}_i (s_0)
\end{align*}
For $ \tilde{A}_i (s_0) $ we already have an $ H^{1}(B) $ bound 
so it remains to estimate the integral in $ H^{\sg}(B)  $, and specifically in $ \dot{H}^{\sg}(B)  $, since integrating in $ L^2(B) $ is straightforward. 

\

{\bf Step 2.} (Replace covariant derivatives by ordinary derivatives)

From Lemma \ref{AiElliptic} applied to $  \tilde{A}_i (s_0) $ we obtain $ s_0^{\frac{1}{4}} \tilde{A}_i(s_0), s_0^{\frac{3}{4}} \pt_x \tilde{A}_i (s_0) \in \delta L^{\infty}(B) $.

From this and from $ (N s'^{\frac{1}{2}})^{1-\sg} s'^{\frac{5}{4}}  \tilde{\mathbf{D}}^{\ell} \tilde{F}_{i \ell }(s') \in \delta L^{\infty}(B) $ - which follows from \eqref{CovGNonBinf} one obtains  
\be \label{tldAiinf}
(N s^{\frac{1}{2}})^{1-\sg}  s^{\frac{1}{4}} (s^{\frac{1}{2}} \pt_x)^m \tilde{A}_i(s) \in \delta L^{\infty}(B).
\ee  
for $ m=0 $ and $ s\in [0,s_0] $ . This allows us to obtain 
\be \label{regDerF}
(N s^{\frac{1}{2}})^{1-\sg} \vn{ (s^{\frac{1}{2}} \pt_x)^m \tilde{F}_{ij}(s) }_{L^2(B)} \ls \delta, \qquad s \in  [0,s_0]
\ee
for $ m=1 $ and also \eqref{tldAiinf} for $ m=1 $
which in turn implies \eqref{regDerF} also for $ m=2 $.

\

{\bf Step 3.} (Estimate the integral term) 
We first consider 
$$
\tilde{A}_i^{main}  \defeq \int_0^{s_0} \pt^{\ell} \tilde{F}_{i \ell }(s') \dd s'
$$

Consider an extension $ \bar{F} $ of $ \tilde{F} $ from $ B\times [0,s_0] $ to $ \mb{R}^3 \times [0,s_0] $ such that 
\be \label{regDerFext}
(N s^{\frac{1}{2}})^{1-\sg} \vn{ (s^{\frac{1}{2}} \pt_x)^m \bar{F}(s) }_{L^2( \mb{R}^3)} \ls \delta , \qquad   0 \leq m \leq 2. 
\ee
for $ s \in  [0,s_0] $. This implies 
$$
(N s^{\frac{1}{2}})^{1-\sg} s^{\frac{1}{2}} \vn{P_k \pt_x \bar{F}(s) }_{\dot{H}^{\sg} } \ls \frac{2^{k \sigma}}{ \jb{ s^{\frac{1}{2}} 2^k } }  \al_k 
$$
for some $ \vn{ (\al_k) }_{\ell^2 } \ls \delta $. Indeed, use \eqref{regDerFext} with $ m=1 $ for low frequencies with  $ s 2^{2k} < 1 $ and use \eqref{regDerFext} with $ m=2 $ for high frequencies such that $ s 2^{2k} > 1 $. Now 
\begin{align} \notag
N^{2(1-\sigma)} \vn{\tilde{A}_i^{main}  }_{\dot{H}^{\sg}(B)}^2  & \ls N^{2(1-\sigma)} \sum_k  \int_0^{s_0} \int_0^{s_0} \lng \partial \bar{F}_k(s_1), \partial \bar{F}_k(s_2) \rng_{\dot{H}^{\sigma}} \dd s_1  \dd s_2  \\
\label{intfreqarg}
& \ls \sum_k  \int_0^{s_0} \int_0^{s_0} 
\frac{ (s_1^{\frac{1}{2}} 2^{k})^{\sigma} }{ \lng s_1^{\frac{1}{2}} 2^k \rng } \frac{  (s_2^{\frac{1}{2}} 2^{k})^{\sigma} }{ \lng s_2^{\frac{1}{2}} 2^k \rng }  \al_k^2
\frac{\dd s_1}{s_1}  \frac{\dd s_2 }{s_2} \\
\notag & \ls \sum_k  \al_k^2 \ls \delta^2. 
\end{align}
The remaining term 
$$
\tilde{A}_i^{rem}  \defeq \int_0^{s_0} [ \tilde{A}^{\ell}, \tilde{F}_{i \ell }](s') \dd s'
$$
can in fact be estimated easily in $ N^{-\frac{1}{2}} H^{1}(B) $ placing $ \tilde{A}(s'), \pt_x \tilde{A}(s') $ in $ L^{\infty}(B) $ and using \eqref{regDerF}. 

\subsection{Proof of Proposition \ref{LocalGauge1} Part (2) } \ 

{\bf Step 1.} (Second change of gauge) 

Since $ U=U(x) $ is independent of $ t $ and $ s $, $ \tilde{A} $ remains in the temporal-caloric gauge $ \tilde{A}_0(t,x,0) = 0, \tilde{A}_s(t,x,s)=0 $. In particular 
$$
\pt_t \tilde{A}_i (t_0,0)   = \int_0^{s_0} \pt_t \tilde{\mathbf{D}}^{\ell} \tilde{F}_{i \ell }(t_0,s') \dd s' + \pt_t  \tilde{A}_i (t_0,s_0).
$$
We want to use the bounds
\be \label{bddFxtild}
 (N s^{\frac{1}{2}})^{1-\sg} \vn{ (s^{\frac{1}{2}} \tilde{\mathbf{D}}_x)^m \tilde{F}_{\al \beta}(t_0,s) }_{L^2(B)} \ls \delta 
\ee
to place $ \pt_t  \tilde{A}_i (t_0,s_0) $ in $ L^2(B) $, but $ \tilde{A}_{\al} (\cdot ,s_0) $ is not in the temporal gauge. So we change the gauge to make it so, letting 
$$
\dbtilde{A}_{\al} = V \tilde{A}_{\al} V^{-1} - \partial_{\al} V V^{-1}, \qquad \dbtilde{F}_{\al \beta} = V \tilde{F}_{\al \beta} V^{-1}
$$ 
for $ V=V(t,x) $ which makes $ \dbtilde{A}_{\al}(\cdot,s_0) $ temporal i.e. by  solving the ODE
$$
\pt_t V = V \tilde{A}_{0}(\cdot,s_0), \qquad V(t_0) = Id.
$$
Now we have 
\be \label{dbtildecond}
\dbtilde{A}_{0}(t,s_0)= 0, \qquad   \dbtilde{A}_{s}=0
\ee
and, since $ V(t_0) = Id $,  
\be \label{dbtildecond2}
\dbtilde{A}_{i}(t_0,s) = \tilde{A}_{i}(t_0,s), \qquad  \dbtilde{F}_{\al \beta} (t_0,s) = \tilde{F}_{\al \beta} (t_0,s)
\ee

\

{\bf Step 2.} (Bounds on $  \dbtilde{A}_{0} $) From $  \dbtilde{F}_{s 0} = \dbtilde{\mathbf{D}}^{\ell} \dbtilde{F}_{\ell 0} $ and \eqref{dbtildecond}, \eqref{dbtildecond2} one has
$$
\dbtilde{A}_{0}(t_0,s) = \int_s^{s_0} \tilde{\mathbf{D}}^{\ell} \tilde{F}_{0 \ell} (t_0,s')  \dd s' 
$$
From this and \eqref{CovGNonBinf}, \eqref{bddFxtild} one obtains 
\be \label{A0tildeLinf}
 (N s^{\frac{1}{2}})^{1-\sg} s^{\frac{1}{4}} \vn{ \dbtilde{A}_{0}(t_0,s)}_{L^{\infty}(B)} \ls \delta.
\ee
From Step 2 in Part (1) we have \eqref{tldAiinf} at $ t=t_0 $ for $ m=0,1 $ which allows to pass from \eqref{bddFxtild} to 
\be \label{regDertilddF}
(N s^{\frac{1}{2}})^{1-\sg} \vn{ (s^{\frac{1}{2}} \pt_x)^m \tilde{F}_{\al \beta}(t_0,s) }_{L^2(B)} \ls \delta, \qquad  \ m\leq 2.
\ee
We claim
\be \label{claimdiA0}
\begin{aligned}
\pt_i \dbtilde{A}_{0}(t_0,0) & =  \int_0^{s_0} \pt_i \pt^{\ell} \tilde{F}_{0 \ell} (t_0,s') + \pt_i [ \tilde{A}^{\ell} , \tilde{F}_{0 \ell} ] (t_0,s')  \dd s' \\ 
& \in N^{\sg-1} H^{\sg-1}(B) +  L^2(B). 
\end{aligned}
\ee
Indeed, the second integral is in $ N^{-\frac{1}{2}}  L^2(B) $ using \eqref{tldAiinf} and \eqref{regDertilddF}. 

To bound the first integral we proceed like in Step 3 from Part (1). Extend $ \tilde{F} $ by  $ \bar{F} $ from $ B\times [0,s_0] $ to $ \mb{R}^3 \times [0,s_0] $ such that \eqref{regDertilddF} holds on $ \mb{R}^3 $ for  $ \bar{F} $. Then 
$$
(N s^{\frac{1}{2}})^{1-\sg} s \vn{P_k \pt_x^2 \bar{F}(s) }_{\dot{H}^{\sg-1} } \ls \frac{s^{\frac{1}{2}} 2^{k \sigma} }{ \jb{ s^{\frac{1}{2}} 2^k } }  \al_k
$$
for some $ \vn{ (\al_k) }_{\ell^2 } \ls \delta $. Then  
\begin{align} \label{Fintfreqargg}
N^{2(1-\sigma)} \vn{ \int_0^{s_0} \pt^2 \tilde{F}(t_0,s')  \dd s'   }_{H^{\sg-1}(B)}^2 & \ls \\
\notag
 N^{2(1-\sigma)} \sum_k  \int_0^{s_0} & \int_0^{s_0} \lng \partial^2 \bar{F}_k(s_1), \partial^2 \bar{F}_k(s_2) \rng_{\dot{H}^{\sigma-1}} \dd s_1  \dd s_2 
\end{align}
This is bounded by \eqref{intfreqarg} and thus $ \ls \delta^2 $.

\

{\bf Step 3.} (Estimate $ \pt_t  \dbtilde{A}_{i} $) Using \eqref{dbtildecond} and $  \pt_t  \dbtilde{A}_i (t_0,s_0) =  \dbtilde{F}_{0i} (t_0,s_0)  $ write
$$
\pt_t \dbtilde{A}_i (t_0,0)   = \int_0^{s_0} \dbtilde{\mathbf{D}}_t  \dbtilde{\mathbf{D}}^{\ell} \dbtilde{F}_{i \ell }(t_0,s') - [ \dbtilde{A}_0 , \dbtilde{\mathbf{D}}^{\ell} \dbtilde{F}_{i \ell } ](t_0,s') \dd s' +  \dbtilde{F}_{0i} (t_0,s_0).
$$ 
The last term $ \dbtilde{F}_{0i} (t_0,s_0) $ is in $ \delta L^2(B) $. We rewrite the integral, by commuting and using Bianchi \eqref{Bianchi} and \eqref{dbtildecond2}, as 
\be  \label{inttilde}
\int_0^{s_0} \tilde{\mathbf{D}}_x^2  \tilde{F}_{0 x } + [\tilde{F}_{0 \ell }, \tilde{F}_{\ell i} ] -   [ \dbtilde{A}_0 , \tilde{\mathbf{D}}^{\ell} \tilde{F}_{i \ell } ] \dd s'  , \qquad \quad t=t_0
\ee 

We  bound the integral of $ [\tilde{F}_{0 \ell }, \tilde{F}_{\ell i} ] $ and $ [ \dbtilde{A}_0 , \tilde{\mathbf{D}}^{\ell} \tilde{F}_{i \ell } ] $ in $ N^{-\frac{1}{2}}  L^2(B) $ using \eqref{bddFxtild} together with \eqref{CovGNonBinf}, and \eqref{A0tildeLinf} together with \eqref{bddFxtild}. The same holds for the integral of $ \tilde{\mathbf{D}}_x^2  \tilde{F}_{0 x } - \pt_x^2  \tilde{F}_{0 x } $. The integral of $   \pt_x^2  \tilde{F}_{0 x } $ is estimated exactly by \eqref{Fintfreqargg}. 

This proves $ \pt_t \dbtilde{A}_i (t_0,0) \in L^2(B) + N^{\sg-1} H^{\sg-1}(B) $. 

\

{\bf Step 4.} (Return to $ \tilde{A}_{\al} $) From the temporal condition $  \tilde{A}_0(t,0) = 0 $ and \eqref{dbtildecond2} one has
$$
\pt_t \tilde{A}_i(t_0,0) =  \tilde{F}_{0i} (t_0,0) = \dbtilde{F}_{0i} (t_0,0)= \pt_t \dbtilde{A}_i(t_0,0) - \pt_i \dbtilde{A}_0(t_0,0) + [ \tilde{A}_i, \dbtilde{A}_0](t_0,0)
$$ 
We claim this is in $ L^2(B) + N^{\sg-1} H^{\sg-1}(B) $: the term $ \pt_t \dbtilde{A}_i (t_0,0) $ was treated in Step 3, the term $ \pt_i \dbtilde{A}_0(t_0,0) $ in \eqref{claimdiA0} in Step 2, while for $  [ \tilde{A}_i, \dbtilde{A}_0] $ we place both inputs in $ L^4(B) $ by Sobolev embedding.


\section{Local differences of heat flows which coincide on a ball} \label{SecErrorSM}

The goal of this section is to compare two caloric heat flow solutions which coincide somewhere initially. Due to infinite speed of propagation they will not coincide in the future, but one can obtain quantitative local $ L^2 $ estimates for the difference. 

\

Consider two balls $ B^{(1)} \subset B $ of radius $ \simeq 1 $. 

\begin{proposition} \label{PropErrorDif}
Let $ A_{t,x,s}, A'_{t,x,s}  : [t_0,t_1] \times B \times [0,s_0] \to  \mathfrak{g} $ be two smooth caloric solutions to \eqref{dYMHF} which coincide on a smaller initial cylinder $ [t_0,t_1] \times B^{(1)} \times \{0 \} $. Let $ \chi $ be a spatial bump function supported in the ball $ B^{(1)} $. Let $ N^2 s_0=1 $. We assume, for all $ t \in [t_0,t_1] $ and $ 0 \leq m \leq 3 $, $ \ep \ll 1 $,  the following bounds: 
\begin{align}
\label{Fderbdds1}
 \vn{  (N s^{\frac{1}{2}})^{1-\sigma} (s^{\frac{1}{2}} \mathbf{D}_x)^m F_{\al \beta}(t,s) }_{L^{\infty}_s L^2 \cap L^{2}_{\frac{\dd s}{s}} L^2 ( [0,s_0]\times B) } \ls \ep \\
\label{Fderbdds2}
 \vn{  (N s^{\frac{1}{2}})^{1-\sigma} (s^{\frac{1}{2}} \mathbf{D}_x')^m F_{\al \beta}'(t,s) }_{L^{\infty}_s L^2  \cap L^{2}_{\frac{\dd s}{s}} L^2 ( [0,s_0]\times B) }
 \ls \ep \\
\label{bddAprime}
A'_x \in \ep L^{\infty}_t L^{\infty}_s \lpp H^1( B)+ N^{\sg-1} H^{\sg} ( B) \rpp
\end{align} 
and 
 \begin{align}
 \label{wderbdds1}
  \vn{  (N^2 s)^{1-\sigma} s^{\frac{1}{4}} (s^{\frac{1}{2}} \mathbf{D}_x)^m w_x(t,s) }_{L^{\infty}_s L^2  \cap L^{2}_{\frac{\dd s}{s}} L^2  ( [0,s_0]\times B) }
\ls \ep  \\
 \label{wderbdds2} 
  \vn{  (N^2 s)^{1-\sigma} s^{\frac{1}{4}} (s^{\frac{1}{2}} \mathbf{D}_x')^m w_x'(t,s) }_{L^{\infty}_s L^2  \cap L^{2}_{\frac{\dd s}{s}} L^2  ( [0,s_0]\times B) }
\ls \ep 
\end{align}
Then, for any $ s \in [0,s_0] $, $ t\in [t_0,t_1] $ one has
\begin{align}
\label{FDiff}
(N s^{\frac{1}{2}})^{9(1-\sigma) } \vn{ \chi [ F_{\al \beta}(t,s) - F'_{\al \beta}(t,s)] }_{L^{\infty}_t L^2_x} & \ls s^{2} \ep \\
\label{wDiff}
(N s^{\frac{1}{2}})^{8(1-\sigma) } \vn{ \chi [ w_{\ell}(t,s) -  w'_{\ell}(t,s)  ] }_{L^{\infty}_t L^2_x} & \ls s^{\frac{5}{4}} \ep
\end{align}
\end{proposition}

It will be clear from the proof that one can in fact obtain arbitrary powers of $ s $ in the conclusion. 

Heuristically, these bounds are clear from the decay of the Gaussian kernels on the spatial scales involved. 

\

{\bf Main equations.}
By caloric solutions we mean that they satisfy $ A_s = A'_s = 0 $, which implies
\be \label{ptsA}
\pt_s A_{\ell} = F_{s \ell} =  \mathbf{D}^{j} F_{j \ell} 
\ee
for each $ t \in [t_0,t_1] $. Taking divergence and using the identity $  \mathbf{D}^{\ell}   \mathbf{D}^{j} F_{j \ell} =0$ one has 

\be  \label{ptsDivA}
\pt_s  \pt^{\ell} A_{\ell} = - [A^{\ell}, F_{s \ell} ]
\ee

From \eqref{0CovPara} and $ A_s=0 $ we get the schematic equation 
\be  \label{cutoffF}
\left(\pt_{s}-\mathbf{D}^{\ell} \mathbf{D}_{\ell} \right) \left( \chi F \right) = 2 \chi [F,F] - 
\Delta \chi F - 2 \nabla \chi \cdot  \mathbf{D}_x F 
\ee
and from \eqref{mCovPara}
\be  \label{cutoffDF}
\left(\pt_{s}-\mathbf{D}^{\ell} \mathbf{D}_{\ell} \right) \left( \chi \mathbf{D}_{x} F \right) =  \chi \mathcal{O}(F, \mathbf{D}_{x} F) - 
\Delta \chi \mathbf{D}_{x}F - 2 \nabla \chi \cdot  \mathbf{D}^2_x F 
\ee
while from \eqref{0CovParaw} we obtain
\be  \label{cutoffw}
\left(\pt_{s}-\mathbf{D}^{\ell} \mathbf{D}_{\ell} \right) \left( \chi w \right) = 2 \chi [F,w] +   \chi \mathcal{O}(F, \mathbf{D}_{x} F) -
\Delta \chi w - 2 \nabla \chi \cdot  \mathbf{D}_x w 
\ee

\

{\bf Preliminary bounds.}
From the covariant Gagliardo-Nirenberg inequalities \eqref{CovGNonB}, \eqref{CovGNonBinf} and \eqref{Fderbdds1}, \eqref{Fderbdds2}, for $ p \in [2,\infty] $ and $ m=0,1 $ one has
\be \label{SobF}
(N s^{\frac{1}{2}})^{1-\sg} \big( \vn{ \mathbf{D}_x^m F (s)}_{L^p( B)}  +  \vn{ \mathbf{D}_x^{'m} F' (s)}_{L^p( B )}   \big) \ls \ep s^{-\frac{3}{2} (\frac{1}{2}-\frac{1}{p})} s^{-\frac{m}{2}} 
\ee
Similarly from \eqref{wderbdds1}, \eqref{wderbdds2} 
\be \label{SobW}
(N^2 s)^{1-\sigma}  \big( \vn{ \mathbf{D}_x^m w_x (s)}_{L^p( B )}  +  \vn{ \mathbf{D}_x^{'m} w'_x (s)}_{L^p( B )}   \big) \ls  \ep s^{-\frac{1}{4}}   s^{-\frac{3}{2} (\frac{1}{2}-\frac{1}{p})} s^{-\frac{m}{2}} 
\ee

From \eqref{bddAprime} one has
\be \label{bddAprimeLq}
A'_x \in \ep L^{\infty}_t L^{\infty}_s \lpp L^6( B )+ N^{\sg-1} L^q ( B ) \rpp, \qquad \frac{1}{2}-\frac{1}{q}=\frac{\sg}{3}.
\ee 

\

{\bf Notation.} Below we use the following notations and identities:  
\begin{align}
\notag
&\delta A_x = A_x - A'_x \qquad    \delta \pt^{\ell} A_{\ell}  = \pt^{\ell} A_{\ell} - \pt^{\ell} A'_{\ell} \\
\notag
&  \delta F = F - F'  \qquad   \delta \mathbf{D}_x^m F  = \mathbf{D}_x^m F - (\mathbf{D}_x^{'})^mF' \qquad  \mathbf{D}_x  \delta F =  \mathbf{D}_x ( F-F')\\
\notag
 &  \delta w = w_x - w'_x \qquad  \delta \mathbf{D}_x w=  \mathbf{D}_x w -  \mathbf{D}'_x w' \\
 & \delta \mathbf{D}^{\ell} \mathbf{D}_{\ell} = \mathbf{D}^{\ell} \mathbf{D}_{\ell} - \mathbf{D}^{',\ell} \mathbf{D}_{\ell}' = 2 [\delta A^{\ell} , \pt_\ell \cdot] + [\delta \pt^\ell  A_{\ell} , \cdot] + [A^{\ell}, [A_{\ell}, \cdot ]]- [A^{' \ell}, [A_{\ell}', \cdot ]]
  \label{DiffDx}
\end{align}

\

The following lemma shows that by slightly reducing the support of the cutoff functions we can  improve the exponents of powers of $ s $ for differences of solutions that coincide at the initial time on a domain. 

\begin{lemma} \label{PowersOfS}
Let $ A_{t,x,s}, A'_{t,x,s}  $ be two solutions like in Proposition \ref{PropErrorDif}. Let $ \chi_k, \chi_{k+1} $ be bump functions supported in $ B^{(1)} $ such that $ \chi_k \equiv 1 $ on $ \supp \chi_{k+1} $. Fix $ t \in [t_0,t_1] $. Let $ k \geq 0 $ and suppose that for $ m=0, 1 $, $ p \in [2,4] $ one has
\begin{align}
\label{DiffHyp1}
  (N s^{\frac{1}{2}})^{(k+1)(1-\sigma)} s^{\frac{m}{2}} \vn{ \chi_k \delta  \mathbf{D}_x^m F(s)}_{ L^2 } & \ls \ep s^{\frac{k}{4}} \\
  \label{DiffHyp2}
  \vn{   (N s^{\frac{1}{2}})^{(k+1)(1-\sigma)} s^{\frac{m+1}{2}} \chi_k \delta  \mathbf{D}_x^{m+1} F }_{ L^{2}_{\frac{\dd s}{s}} L^2 [0,s_1] } & \ls \ep s_1^{\frac{k}{4}}\\
  \label{DiffHyp3}
   (N s^{\frac{1}{2}})^{(k+1)(1-\sigma)} \vn{  \chi_k  (A_x - A_x')(s)}_{L^p} & \ls \ep s^{\frac{k+2}{4}} s^{-\frac{3}{2}(\frac{1}{2}-\frac{1}{p})} \\
   \label{DiffHyp4}
    (N s^{\frac{1}{2}})^{(k+2)(1-\sigma)}  \vn{  \chi_k  (\pt^{\ell} A_{\ell} - \pt^{\ell} A'_{\ell})(s)}_{L^2} & \ls \ep s^{\frac{k+1}{4}} 
\end{align}
for any $ s,s_1 \in [0,s_0] $. Then one has 
\begin{align}
\label{DiffCon1}
  (N s^{\frac{1}{2}})^{(k+2)(1-\sigma)} s^{\frac{m}{2}} \vn{ \chi_{k+1}\delta  \mathbf{D}_x^m F(s)}_{ L^2 } & \ls \ep s^{\frac{k+1}{4} } \\
  \label{DiffCon2}
  \vn{   (N s^{\frac{1}{2}})^{(k+2)(1-\sigma)} s^{\frac{m+1}{2}} \chi_{k+1}\delta  \mathbf{D}_x^{m+1} F }_{ L^{2}_{\frac{\dd s}{s}} L^2 [0,s_1] } & \ls \ep s_1^{ \frac{k+1}{4} }\\
  \label{DiffCon3}
   (N s^{\frac{1}{2}})^{(k+2)(1-\sigma)} \vn{  \chi_{k+1} (A_x - A_x')(s)}_{L^p} & \ls \ep s^{ \frac{k+3}{4} } s^{-\frac{3}{2}(\frac{1}{2}-\frac{1}{p})}  \\
 \label{DiffCon4}
    (N s^{\frac{1}{2}})^{(k+3)(1-\sigma)}  \vn{  \chi_{k+1} (\pt^{\ell} A_{\ell} - \pt^{\ell} A'_{\ell})(s)}_{L^2} & \ls \ep s^{\frac{k+2}{4}} 
\end{align}
Moreover, if one assumes $ k \geq 2 $ and 
\be  \label{DiffHyp5}
(N^2 s_1)^{1-\sg} s_1^{\frac{1}{4}} \vn{ \chi_k \delta w(s_1)}_{L^2} + \vn{   (N^2 s)^{1-\sigma} s^{\frac{3}{4}} \chi_k \delta \mathbf{D}_x  w
}_{ L^{2}_{\frac{\dd s}{s}} L^2 [0,s_1] } \ls \frac{\ep s_1^{\frac{k-2}{4}} }{ (N s_1^{\frac{1}{2}})^{(k-2)(1-\sigma)} } 
\ee
 for all $ s_1 \in [0,s_0] $ then one also has
\be \label{DiffCon5}
(N^2 s_1)^{1-\sg} s_1^{\frac{1}{4}} \vn{ \chi_{k+1} \delta w(s_1)}_{L^2} + \vn{   (N^2 s)^{1-\sigma} s^{\frac{3}{4}} \chi_{k+1} \delta \mathbf{D}_x  w
}_{ L^{2}_{\frac{\dd s}{s}} L^2 [0,s_1] } \ls \frac{\ep s_1^{\frac{k-1}{4} } }{ (N s_1^{\frac{1}{2}})^{(k-1)(1-\sigma)} } 
\ee
\end{lemma}

Assuming this Lemma we can can now prove the main estimates of this section.

\subsection{Proof of Proposition \ref{PropErrorDif}}

Fix a time $ t \in [t_0,t_1] $. 
By integrating equation \eqref{ptsA} point-wise and the fact that the solutions coincide at $ s=0 $ on $ B^{(1)} $, we have 
$$
 A_{\ell}(s_1)  - A_{\ell}'(s_1)   =  \int_0^{s_1}  \mathbf{D}^{j} F_{j \ell}(s) -  \mathbf{D}^{'j} F'_{j \ell}(s)   \dd s,   \qquad x \in  B^{(1)}
$$
 which implies, by estimating $  \mathbf{D}^{j} F_{j \ell} $ and $ \mathbf{D}^{'j} F'_{j \ell} $ separately using \eqref{Fderbdds1}, \eqref{Fderbdds2}, \eqref{SobF}, that 
\be \label{diffA1}
 (N s^{\frac{1}{2}})^{1-\sigma} \vn{  A_x(s) - A_x'(s)}_{L^p(  B^{(1)})} \ls \ep s^{\frac{1}{2}} s^{-\frac{3}{2}(\frac{1}{2}-\frac{1}{p})}, \quad p \in [2,4]
\ee
By integrating equation \eqref{ptsDivA}, on $ B^{(1)} $ we have 
$$ 
\pt^{\ell} A_{\ell}(s_1) - \pt^{\ell} A'_{\ell}(s_1)   = -  \int_0^{s_1}  [A_x - A_x', \mathbf{D}_x F   ] +  [ A',   \mathbf{D}_x F - \mathbf{D}'_x F' ]   \dd s  
$$
We bound the integrand in $ L^2_x $. We use \eqref{diffA1} with $ p=2 $ and \eqref{SobF} with $ p=\infty $ for the first term. The second term we bound by $  (L^6+ N^{\sg-1} L^q) \times ( L^3 \cap L^{3/\sg}) \to L^2 $ using \eqref{bddAprimeLq} and \eqref{SobF} with $ p=3 $ and $ p=3/\sg $, for both $  \mathbf{D}_x F  $ and $ \mathbf{D}'_x F' $ separately. 
We obtain 
\be \label{diffA2}
 (N s^{\frac{1}{2}})^{2(1-\sigma)}
\vn{ \pt^{\ell} A_{\ell}(s) - \pt^{\ell} A'_{\ell}(s)}_{L^2(  B^{(1)})} \ls \ep s^{\frac{1}{4}} 
\ee

\

We will apply Lemma \ref{PowersOfS}  several times with different cutoff functions to obtain \eqref{FDiff} and \eqref{wDiff}, each time getting an extra power of $ s^{1/4} $ and slightly reducing the support.  

Given $ \chi $ supported in $ B^{(1)} $, we consider a sequence of bump functions $ \chi_0, \chi_1, \dots , \chi_8 = \chi $ all supported in $ B^{(1)} $ 
such that $ \chi_k \equiv 1 $ on $ \supp \chi_{k+1} $, for each $ k \in \{0, \dots 7 \} $. 

Based on \eqref{Fderbdds1}, \eqref{Fderbdds2}, \eqref{diffA1}, \eqref{diffA2} we first apply Lemma \ref{PowersOfS} with $ k=0 $ and $ \chi_0, \chi_1 $. We obtain \eqref{DiffHyp1}-\eqref{DiffHyp4} with $ k=1 $ for $ \chi_1 $. We repeatedly apply Lemma \ref{PowersOfS} with $ \chi_{k}, \chi_{k+1} $ until we obtain \eqref{FDiff}. Starting from $ k \geq 2 $ we also obtain bounds on $  \delta w $: in the $ k=2 $ step the bound \eqref{DiffHyp5} follows trivially from \eqref{wderbdds1}, \eqref{wderbdds2} by the triangle inequality. When $ k=7 $ we conclude \eqref{wDiff}.

\subsection{Proof of Lemma \ref{PowersOfS} }

For brevity, in the course of this proof we denote $ \tilde{\chi} \defeq \chi_k $ and $ \chi \defeq \chi_{k+1} $. Fix a time $ t \in [t_0,t_1] $. 

\subsubsection{Difference equation for F}
We first prove  \eqref{DiffCon1},  \eqref{DiffCon2} for $ m=0 $.

Subtracting \eqref{cutoffF}  for $ F $ and $ F' $ and using the form of $  \delta \mathbf{D}^{\ell} \mathbf{D}_{\ell} $ from \eqref{DiffDx} we obtain the schematic equation (where we drop some constants and irrelevant signs)  
\begin{align}  
\label{Ndiff1}
\left(\pt_{s}-\mathbf{D}^{\ell} \mathbf{D}_{\ell} \right) \left( \chi \delta F \right) = & \   \chi [\delta F, F] +  \chi [\delta F, F']  \\ 
\label{Ndiff2}
& +  \nabla \chi \cdot \delta \mathbf{D}_x F +  \Delta \chi \delta F \\
\label{Ndiff3}
& + [\delta A^{\ell} , \pt_\ell (\chi F')] + \chi [\delta \pt^\ell  A_{\ell} ,  F'] \\
\label{Ndiff4}
& + \chi [A^{\ell}, [A_{\ell},  F' ]]- \chi [A^{' \ell}, [A_{\ell}',  F' ]]
\end{align}
Note that $ \chi \delta F(0)=0  $. Using the energy estimate \eqref{EnergyEst} with $ p=0 $ we obtain 
\be \label{EnergyDiff}
\vn{ \chi \delta F }_{L^{\infty}_{s} L^2_x [0,s_1] }  +  \vn{ s^{\frac{1}{2}} \mathbf{D}_x  \left( \chi \delta F(s) \right)}_{L^2_{ \frac{\dd s}{s}} L^2_x [0,s_1] } \ls \ep s_1^{\frac{k+1}{4}}   (N s_1^{\frac{1}{2}})^{-(k+2)(1-\sigma)}
\ee
provided we estimate \eqref{Ndiff1} - \eqref{Ndiff4} in $ L^1_s L^2_x [0,s_1] $. Then, note that by bounding $ s^{\frac{1}{2}} \nabla \chi \cdot \delta F(s) $ and $  s^{\frac{1}{2}} \chi [\delta A, F] (s) $ using  \eqref{DiffHyp1}, \eqref{DiffHyp3}  and \eqref{SobF} we would also have 
\be \label{EnergyDiff2}
 \vn{ s^{\frac{1}{2}} \chi \mathbf{D}_x  \left(  \delta F \right)}_{L^2_{ \frac{\dd s}{s}} L^2_x [0,s_1] }  +
  \vn{ s^{\frac{1}{2}} \chi \delta \mathbf{D}_x F }_{L^2_{ \frac{\dd s}{s}} L^2_x [0,s_1] } 
 \ls \ep s_1^{\frac{k+1}{4}}   (N s_1^{\frac{1}{2}})^{-(k+2)(1-\sigma)}
\ee
which will give \eqref{DiffCon1},  \eqref{DiffCon2} for $ m=0 $. 

We estimate the nonlinear terms in $ L^1_s L^2_x [0,s_1] $ as follows: 

For \eqref{Ndiff1} we use 
\eqref{DiffHyp1} with $ m=0 $ and \eqref{SobF} with $ p=\infty, \ m=0 $; 

For \eqref{Ndiff2} we use \eqref{DiffHyp1}, \eqref{DiffHyp2}  with $ m=0 $, recalling that $ \tilde{\chi} \equiv 1 $ on $ \supp \chi $;

For \eqref{Ndiff3} we use \eqref{DiffHyp4}, \eqref{SobF} with $ p=\infty $; to estimate the second term. Then use 
\eqref{DiffHyp3}, \eqref{SobF} with $  m=1,0, \ p=\infty $ to estimate $ \chi [\delta A^{\ell} , \mathbf{D}_{\ell}'  F']  $ and $  [\delta A^{\ell} , \pt_{\ell} \chi  F']  $. 

For \eqref{Ndiff4} and the remainder from \eqref{Ndiff3}, write everything in terms of 
\be \label{trilineardelta}
\mathcal{O}( A', \chi \delta A, F') \quad \text{and} \  \qquad  \mathcal{O}( \chi \delta A,  \tilde{\chi} \delta A, F')
\ee
The first term is estimated as $  (L^6+ N^{\sg-1} L^q) \times ( L^3 \cap L^p) \times L^{\infty} \to L^2 $
using \eqref{bddAprimeLq}, \eqref{DiffHyp3} with $ p=3/\sg $ and \eqref{SobF}. 
The second term is bounded as $ L^4 \times L^4 \times L^{\infty} \to L^2 $ from \eqref{DiffHyp3}, \eqref{SobF}.

\subsubsection{Difference equation for $ \mathbf{D}_x F $}
Similarly one writes the parabolic equation for differences from \eqref{cutoffDF}
\begin{align}  
\label{DNdiff1}
\left(\pt_{s}-\mathbf{D}^{\ell} \mathbf{D}_{\ell} \right) \left( \chi \delta \mathbf{D}_x F \right) = & \   \chi  \mathcal{O}(\delta F, \mathbf{D}_x F) +  \chi \mathcal{O}( F', \delta \mathbf{D}_x F) \\ 
\label{DNdiff2}
& +  \nabla \chi \cdot \delta \mathbf{D}_x^2 F +  \Delta \chi \delta \mathbf{D}_x F \\
\label{DNdiff3}
& + [\delta A^{\ell} , \pt_\ell (\chi  \mathbf{D}'_x F')] + \chi [\delta \pt^\ell  A_{\ell} , \mathbf{D}'_x F'] \\
\label{DNdiff4}
& + \chi [A^{\ell}, [A_{\ell},   \mathbf{D}'_x F' ]]- \chi [A^{' \ell}, [A_{\ell}',   \mathbf{D}'_x F' ]]
\end{align}
Denote the nonlinearity by $ \mathcal{M} $. Using the energy estimate \eqref{EnergyEst} with with $ p=\frac{1}{2} $ one obtains 
\begin{align}  
 \label{EnergyDiffDx1}
\vn{ \chi s^{\frac{1}{2}} \delta \mathbf{D}_x F }_{L^{\infty}_{s} L^2_x [0,s_1] }  +  \vn{ s \mathbf{D}_x  \left( \chi \delta \mathbf{D}_x F(s) \right)}_{L^2_{ \frac{\dd s}{s}} L^2_x [0,s_1] } \ls \\
 \label{EnergyDiffDx2}
\vn{ \chi s^{\frac{1}{2}} \delta \mathbf{D}_x F }_{L^2_{ \frac{\dd s}{s}} L^2_x [0,s_1] } + \int_0^{s_1} s^{\frac{1}{2}} \vn{\mathcal{M}(s)}_{L^2_x} \dd s 
\end{align}
We prove 
\be
\eqref{EnergyDiffDx2} \ls   \ep s_1^{\frac{k+1}{4}}   (N s_1^{\frac{1}{2}})^{-(k+2)(1-\sigma)}
\ee
The first term in \eqref{EnergyDiffDx2} was estimated in \eqref{EnergyDiff2}. 
To estimate $ \mathcal{M} $ in $ L^1_s L^2_x [0,s_1] $ we proceed exactly as in the previous step, with the following difference:  for the terms which have $ \mathbf{D}_x F $ instead of $ F $, when applying \eqref{DiffHyp1}, \eqref{DiffHyp2}, \eqref{SobF}  we use $ m=1 $ instead of $ m=0 $. 

As before, the second term in \eqref{EnergyDiffDx1} can be replaced by $ s \chi \delta \mathbf{D}_x^2 F(s) $, which completes the proof of \eqref{DiffCon1},  \eqref{DiffCon2} also for $ m=1 $.

\subsubsection{Bound on $ A $ and $ div A $}
By integrating equations \eqref{ptsA}, \eqref{ptsDivA} and the fact that the solutions coincide at $ s=0 $ on the support of $ \chi $ we have 
\be
 \chi \left( A_{\ell}(s_1)  - A_{\ell}'(s_1) \right)  = \chi \int_0^{s_1} \delta \mathbf{D}_x F (s)   \dd s   
\ee
\be \label{intDiv}
 \chi  \left( \pt^{\ell} A_{\ell}(s_1) - \pt^{\ell} A'_{\ell}(s_1)  \right)  = - \chi \int_0^{s_1}  [\delta A, \mathbf{D}_x F   ] +  [ A',  \delta \mathbf{D}_x F ]   \dd s  
\ee
By the covariant Gagliardo-Nirenberg inequalities \eqref{SobGN} we have 
\be \label{GNSdiv}
\vn{ \chi \delta \mathbf{D}_x F (s)}_{L^p_x} \ls \vn{ \chi \delta \mathbf{D}_x F (s)}_{L^2_x}^{1-\theta} \vn{ \mathbf{D}_x  \left( \chi \delta \mathbf{D}_x F(s) \right)}_{L^2_x}^{\theta},  \quad \theta = 3( 2^{-1} - p^{-1}) 
\ee
Integrating this and using Holder in $  \frac{\dd s}{s} $ we obtain \eqref{DiffCon3} from \eqref{EnergyDiffDx1}. 

For \eqref{DiffCon4} we estimate the integral \eqref{intDiv} as follows: for the first term use \eqref{DiffCon3} with $ p=2 $ and \eqref{SobF} with  $ p=\infty, \ m=1 $; the second term is estimated as  $  (L^6+ N^{\sg-1} L^q) \times ( L^3 \cap L^p) \to L^2 $
using \eqref{bddAprimeLq} and \eqref{GNSdiv} for $ L^3 $ and $ L^p $ with $ p=3/\sg $, together with \eqref{EnergyDiffDx1}. 

\subsubsection{Bound on $ w $} 

Subtracting \eqref{cutoffw}  for $ w $ and $ w' $ and using the form of $  \delta \mathbf{D}^{\ell} \mathbf{D}_{\ell} $ from \eqref{DiffDx} we obtain the schematic equation
\begin{align}  
\label{Nwdiff1}
\left(\pt_{s}-\mathbf{D}^{\ell} \mathbf{D}_{\ell} \right) \left( \chi \delta w \right) = & \   \chi [\delta F, w] +  \chi [ F, \delta w ]  \\ 
\label{Nwdiff2}
 &+  \chi  \mathcal{O}(\delta F, \mathbf{D}_x F) +  \chi \mathcal{O}( F', \delta \mathbf{D}_x F)  \\ 
\label{Nwdiff3}
& +  \nabla \chi \cdot \delta \mathbf{D}_x w +  \Delta \chi \delta w \\
\label{Nwdiff4}
& + [\delta A^{\ell} , \pt_\ell (\chi w')] + \chi [\delta \pt^\ell  A_{\ell} ,  w'] \\
\label{Nwdiff5}
& + \chi [A^{\ell}, [A_{\ell},  w' ]]- \chi [A^{' \ell}, [A_{\ell}',  w' ]]
\end{align}

 Using the energy estimate \eqref{EnergyEst} with $ p=0 $ we obtain 
\be \label{EnergyDiffW}
\mathcal{T} \defeq \vn{ \chi \delta w }_{L^{\infty}_{s} L^2_x [0,s_1] }  +  \vn{ s^{\frac{1}{2}} \mathbf{D}_x  \left( \chi \delta w \right)}_{L^2_{ \frac{\dd s}{s}} L^2_x [0,s_1] } \ls \ep \mathcal{T} + \frac{\ep s_1^{\frac{k-1}{4}}}{(N s_1^{\frac{1}{2}})^{(k+2)(1-\sigma)}}
\ee
provided we bound \eqref{Nwdiff1} - \eqref{Nwdiff5} in $ L^1_s L^2_x [0,s_1] $ by the RHS of \eqref{EnergyDiffW}. Then we absorb the $ \ep \mathcal{T} $ to the LHS and, replace the $ \mathbf{D}_x  \left( \chi \delta w \right) $ by $  \chi \delta \mathbf{D}_x w $ by bounding $ s^{\frac{1}{2}} \nabla \chi \cdot \delta w $ and $  s^{\frac{1}{2}} \chi [\delta A, w] $ using  \eqref{DiffHyp5}, \eqref{DiffHyp3}  and \eqref{SobW} obtaining 
$$
\vn{ \chi \delta w (s_1)}_{L^2_x} + \vn{ s^{\frac{1}{2}} \chi \delta \mathbf{D}_x w}_{L^2_{ \frac{\dd s}{s}} L^2_x [0,s_1] } \ls  \frac{\ep s_1^{\frac{k-1}{4}}}{(N s_1^{\frac{1}{2}})^{(k+2)(1-\sigma)}}
$$
This is more than enough for \eqref{DiffCon5}. We estimate the nonlinear terms in $ L^1_s L^2_x [0,s_1] $ very similarly to the case of $ \chi \delta F $:
 
 For \eqref{Nwdiff1} we first use 
\eqref{DiffHyp1} with $ m=0 $ and \eqref{SobW} with $ p=\infty, \ m=0 $; then for $ \chi [ F, \delta w ] $ we use  \eqref{SobF} with $ p=\infty, \ m=0 $, which contributes the term $  \ep s_1^{\frac{1}{4}} \mathcal{T} \ll \ep  \mathcal{T} $. 

The terms in \eqref{Nwdiff2} were already estimated in \eqref{DNdiff1}. 

For \eqref{Nwdiff3} we use \eqref{DiffHyp5}, recalling that $ \tilde{\chi} \equiv 1 $ on $ \supp \chi $;

For \eqref{Nwdiff4} we use \eqref{DiffHyp4}, \eqref{SobW} with $ p=\infty $; to estimate the second term. Then use 
\eqref{DiffHyp3}, \eqref{SobW} with $  m=1,0, \ p=\infty $ to estimate $ \chi [\delta A^{\ell} , \mathbf{D}_{\ell}'  w']  $ and $  [\delta A^{\ell} , \pt_{\ell} \chi  w']  $. 

For \eqref{Nwdiff5} and the remainder from \eqref{Nwdiff4}, write everything in terms of 
$$
\mathcal{O}( A', \chi \delta A, w') \quad \text{and} \  \qquad  \mathcal{O}( \chi \delta A,  \tilde{\chi} \delta A, w')
$$ 
and then proceed exactly like in \eqref{trilineardelta}, using \eqref{SobW} instead of \eqref{SobF}.


\section{Decompositions and estimates in DeTurck's gauge} \label{SecDecAndEst}

In this section we consider solutions  $ A_{t,x,s} $ to \eqref{YM} and \eqref{dYMHF} on $  [t_0,t_1] \times \mb{R}^3 \times [0,s_0]$, $ t_1=t_0+1$ in:
\begin{enumerate}
\item The temporal gauge at $ s=0 $, i.e. $ A_0(t,x,0) = 0 \ $.
\item The DeTurck gauge $ A_s = \pt^{\ell} A_{\ell} $ on $  [t_0,t_1] \times \ \mb{R}^3 \times [0,s_0]$. 
\item The Coulomb gauge at $ s=0,\ t=t_0 $, i.e. $  \pt^{\ell} A_{\ell} (t_0,x,0)=0$. 
\end{enumerate}
Let $ \sg > \frac{5}{6} $. We assume $ A_i(t,x,0)$ obeys the bounds in Proposition \ref{ModLWP} part (2) and therefore also \eqref{PkDtAcfL2Linfx} - \eqref{PkDtAL2Linfx}. Then Proposition \ref{DeTurckProp} holds uniformly in $ t \in [t_0,t_1]  $, and as a consequence \eqref{ALinfty}, \eqref{AbilL2Linf}.

 We assume $ \mathcal{IE}(t) \ls \ep^2 $ control of the modified energy \eqref{ModifiedEnergy} for all $ t \in [t_0,t_1] $. In particular, by Proposition \ref{CovBdsF1} and \ref{CovBdsw1}  the curvature $ F_{\al \beta} $ and the tension field $ w_i $ defined in \eqref{wDef} obey the bounds 
\begin{align}
\label{HighCovDerF}
 \vn{  (N s^{\frac{1}{2}})^{1-\sigma} (s^{\frac{1}{2}} \mathbf{D}_x)^m F_{\al \beta}(t,s) }_{L^{\infty}_s L^2_x \cap L^{2}_{\frac{\dd s}{s}} L^2_x ( [0,s_0]\times \mb{R}^3) } \ls \ep \\
 \label{HighCovDerFw}
 \vn{ (N^2 s)^{1-\sigma} s^{\frac{1}{4}} (s^{\frac{1}{2}} \mathbf{D}_x)^m w_x(t,s) }_{L^{\infty}_s L^2_x \cap L^{2}_{\frac{\dd s}{s}} L^2_x ( [0,s_0]\times \mb{R}^3)} \ls \ep  
 \end{align}
for all $ 0 \leq m \leq 8 $ and $ t\in [t_0,t_1] $. 

Under these assumptions we control the norms \eqref{AcfL2Pk} - \eqref{Fsinfty}.

\subsection{Decompositions in DeTurck's gauge} \

\subsubsection{Decomposition of $ F $} \ 

Expanding \eqref{0CovPara} and canceling the $ [A_s, F ]$ term with the $ [  \pt^{\ell} A_{\ell}, F ] $ term one obtains 
\be \label{FeqDeTurck}
(\partial_{s} -\Delta) F_{\al \beta}  = 2 [F_{\al}^{\ \ell}, F_{\ell \beta}] + 2[ A^{\ell}, \pt_{\ell} F_{\al \beta} ] + [A^{\ell}, [A_{\ell},  F_{\al \beta} ]]
\ee
and one decomposes 
\be \label{Fdec} 
F_{\al \beta} (s) =  e^{s \Delta} F_{\al \beta} (0) + F^{bil}_{\al \beta} (s) 
\ee
where the bilinear term 
\be \label{Fbil}
F^{bil}_{\al \beta}(s_1) \defeq  \int_0^{s_1} e^{(s_1-s)\Delta} \lpr 2 [F_{\al}^{\ \ell}, F_{\ell \beta}] + 2[ A^{\ell}, \pt_{\ell} F_{\al \beta} ] + [A^{\ell}, [A_{\ell},  F_{\al \beta} ]] \rpr  (s) \dd s 
\ee
obeys better estimates, such as  

\begin{lemma} Under the assumptions above, one has:
\be \label{FbilL2}
\vn{F^{bil}(s)}_{L^{\infty}_t L^2_x} +  s^{\frac{1}{2}} \vn{\pt_x F^{bil}(s)}_{L^{\infty}_t L^2_x} \ls s^{\frac{1}{4}} / (N^2 s)^{1-\sigma} \ep. 
\ee
\end{lemma}

\begin{proof}
Fix a time $ t \in [t_0,t_1] $. It suffices to bound the three terms in the bracket in \eqref{Fbil} in $ L^2_x $ by $ s^{-\frac{3}{4}} /  (N^2 s)^{1-\sigma} \ep $. For the second estimate in \eqref{FbilL2} one begins with the fact that $ (s_1-s)^{\frac{1}{2}} \pt_x e^{(s_1-s)\Delta} : L^2_x \to L^2_x $. 

For $ [F, F] $ and $ [ A^{\ell}, \mathbf{D}_{\ell} F_{\al \beta} ] $ we use Holder, placing the first term in $ L^{\infty}_x $ and the second in $ L^2_x $, recalling \eqref{ALinfty}. For $ [A,[A,F]] $ we do $ L^{\infty}_x \times L^{\infty}_x \times L^2_x \to L^2_x $.  
\end{proof}

\subsubsection{Decomposition of $ w $} \

Expanding \eqref{0CovParaw}, plugging in \eqref{Fdec} and canceling the $ [A_s, w ]$ term with the $ [  \pt^{\ell} A_{\ell},w ] $ term one obtains 
$$
(\partial_{s} -\Delta) w_k = \mathcal{N}^{(2)} + \mathcal{N}^{(3)} + \mathcal{N}^{self} 
$$
where
\be
\begin{aligned}
\label{NonlinearityW}
\mathcal{N}^{(2)} (s) & \defeq 2 [ e^{s \Delta} \pt_t A^{\ell}(0), e^{s \Delta} (\pt_k \pt_t A_{\ell} - 2 \pt_{\ell} \pt_t A_k )(0)] \\
\mathcal{N}^{(3)} (s) &\defeq [F^{bil}_{0 \ell}, \mathbf{D}_x F_{0x}] + 2 [ e^{s \Delta} \pt_t A^{\ell}(0), \pt_x  F^{bil}_{0 x} +[A_x,F_{0x}] ]  \\
\mathcal{N}^{self} (s) & \defeq 2[F_k^{\ \ell}, w_{\ell}] + 2 [A^{\ell}, \pt_{\ell} w_k] +  [A^{\ell}, [A_{\ell},  w_k ]] 
\end{aligned}
\ee
where $ \mathcal{N}^{(2)} $ contains the dominant part from $ 2\left[F^{0 \ell}, \mathbf{D}_{k} F_{0 \ell}+2 \mathbf{D}_{\ell} F_{k 0}\right] $ (recall the temporal gauge condition implies $ F_{0 \ell}= \pt_t A_{\ell} $ at $ s=0$ ). The term $ \mathcal{N}^{(3)} $ consists of the remainder and $ \mathcal{N}^{self} $ contains the terms involving $ w $.

Since $ w(0)=0 $, this gives rise to the following decomposition which isolates the leading quadratic part of $ w $:  
\be \label{wDec}
w(s) = w^{(2)}(s) + w^{(3)}(s)
\ee
where 
$$
 w^{(2)}(s_1) \defeq \int_0^{s_1} e^{(s_1-s) \Delta}\mathcal{N}^{(2)} (s) \dd s, \quad w^{(3)}(s_1) \defeq \int_0^{s_1} e^{(s_1-s) \Delta} ( \mathcal{N}^{(3)} (s) +\mathcal{N}^{self} (s) )  \dd s
$$
The term $ w^{(2)}(s) $ can be written, following \cite{oh2017yang}, as 
\be \label{w2form}
w^{(2)}_i(s) = \mathbf{W} (  \pt_t A^{\ell}(0), \pt_i \pt_t A_{\ell}(0) - 2 \pt_{\ell} \pt_t A_i (0) )
\ee
where  $\mathbf{W}$ is a symmetric bilinear form with symbol
\be
\begin{aligned}  \label{Wdef}
\mathbf{W}(\xi, \eta, s) &=\int_{0}^{s} e^{-\left(s-s' \right)|\xi+\eta|^{2}} e^{-s' \left(|\xi|^{2}+|\eta|^{2}\right)} \dd s' \\ 
 & =  -\frac{1}{2 \xi \cdot \eta} e^{-s |\xi+\eta|^{2}} (1-e^{2 s(\xi \cdot \eta)} )
\end{aligned}
\ee
which enjoys the following favorable frequency concentration: 

\begin{lemma}[\cite{oh2020hyperbolic} - Lemma 8.2]  \label{LemmaWfreq}
 For any $k, k_{1}, k_{2} \in \mathbb{Z}$ and $s>0$, denoting $ k_{\max } = \max(k,k_1,k_2) $,  the bilinear operator
$$
\left\langle s 2^{2 k}\right\rangle^{10}\left\langle s^{-1} 2^{-2 k_{\max }}\right\rangle 2^{2 k_{\max }} P_{k} \mathbf{W}\left(P_{k_{1}}(\cdot), P_{k_{2}}(\cdot), s\right)
$$
is disposable, i.e. its kernel has bounded mass, see \eqref{BilMultOp}. 
\end{lemma}
 
Using this structure of $\mathbf{W}$, one can obtain more precise estimates for $ w^{(2)} $, such as
\begin{lemma} Under the assumptions above, one has:
\be \label{Pkw2L2}
\vn{P_k w^{(2)}_x(s)}_{L^2_t L^2_x} \ls  \ep 2^{(0+)k} \frac{ (2^k s^{\frac{1}{2}})^{\frac{1}{2}-}}{\jb{2^{2k}s}^8} \frac{1}{ (N^2 s)^{1-\sigma}  } \lpr 1 + \frac{1}{2^{\frac{k}{2}-}} \frac{1}{ (N s^{\frac{1}{2}})^{1-\sigma}} \rpr
\ee
Therefore, after summing in $ k $ one obtains:
\be \label{w2L2}
\vn{w^{(2)}_x(s)}_{L^2_t L^2_x} \ls   \ep  s^{0-} / (N^2 s)^{1-\sigma}
\ee
\end{lemma}

\begin{remark} \label{RmkPkw2L2}
One can decompose the inputs of $ w^{(2)}_x(s) $ into $ \pt_t A^{cf} + \pt_t A^{df} $.
It is clear from the proof that the bound \eqref{Pkw2L2} is true also for the separate components of $ w^{(2)}_x(s) $, in particular for $ P_k w^{(2)}(\pt_t A^{df}, \pt_t A^{df}, s) $. 
\end{remark}

\begin{proof}
Doing a Littlewood-Paley decomposition and invoking Lemma \ref{LemmaWfreq} we need to bound
\be
P_k \mathcal{O}(\pt_t A_{k_1}(0), \pt_t A_{k_2}(0)) \cdot 2^{- k_{\max }} \left\langle s 2^{2 k}\right\rangle^{-10} \left\langle s^{-1} 2^{-2 k_{\max }}\right\rangle^{-1} 
\ee
in the two cases, without loss of generality:
\begin{enumerate}
\item
low-high $ 2^{k_1} \ls 2^k \simeq 2^{k_2} \simeq 2^{k_{\max }} $. Bound $ \pt_t A_{k_1}(0) $ in $ L^2 L^{\infty} $ using \eqref{PkDtAL2Linfx} and $ \pt_t A_{k_2}(0) $ in $ L^{\infty}  L^2 $.
Then one sums in $ k_1 $, obtaining the bound 
$$
\ep 2^{(0+)k}  \frac{ 2^{2k}s}{\jb{2^{2k}s}^{10}} \frac{1}{ (N^2 s)^{1-\sigma}  }. 
$$
\item
 high-high to low $ 2^k \ls 2^{k_1} \simeq 2^{k_2} \simeq 2^{k_{\max }} $. Split both inputs into $ \pt_t A^{cf} + \pt_t A^{df} $.  When at least one input is $ \pt_t A^{cf} $, place that input in $ L^2 L^{\infty} $ using \eqref{PkDtAcfL2Linfx} and the other one in $ L^{\infty}  L^2 $. After summing the high frequencies one obtains the bounds 
 $$
\ep  \frac{1}{2^{\frac{k}{2}-}} \left( \frac{2^k}{N} \right)^{3 (1-\sg)} \frac{1}{(2^{2k}s)^{10}} \qquad \text{or} \qquad  \ep \frac{s^{\frac{1}{4}- }}{ (N s^{\frac{1}{2}})^{3(1-\sigma)}}
 $$
 depending on whether $ s^{-\frac{1}{2}} = k(s) \leq k $ or $ k \leq k(s) $. 
 
 The remaining case $ P_k \mathcal{O}(\pt_t A^{df}_{k_1}(0), \pt_t A^{df}_{k_2}(0)) $ is bounded using the Bilinear Strichartz estimate \eqref{BilStrichartz}. After summation one can obtain either 
$$
\ep 2^{(0+)k}  \frac{ 2^{2k}s}{\jb{2^{2k}s}^{10}} \frac{1}{ (N^2 s)^{1-\sigma}  } \qquad \text{or} \qquad \ep 2^{(0+)k} \frac{ (2^k s^{\frac{1}{2}})^{\frac{1}{2}-}}{\jb{2^{2k}s}^{10}} \frac{1}{ (N^2 s)^{1-\sigma}  } 
$$ 
 depending on whether $ s^{-\frac{1}{2}} = k(s) \leq k $ or $ k \leq k(s) $.   
\end{enumerate}
\end{proof}

\subsection{Estimates for the curvature and the tension field}

\subsubsection{Preparations} For easy reference we record here the bounds on $ [t_0,t_1] \times \mb{R}^3 $ which we will use below. 
These follow from \eqref{curlfree}, \eqref{divfree}, \eqref{Strichartzz}, \eqref{PkDtAL2Linfx},  \eqref{ALinfty}, \eqref{HighCovDerF}, \eqref{SobGN}. 


\begin{align}
\label{AcfL2Pk}
\vn{P_k A^{cf}(0)}_{L^{\infty}_t L^2_x} & \ls \ep 2^{-\frac{5}{4}k} \jb{ \frac{2^k}{N}}^{1-\sg} \\ 
\label{AcfL12}
s^{\frac{1}{8}} \vn{e^{s \Delta} A^{cf}(0)}_{ L^{\infty}_{t,x} }  & \ls \vn{e^{s \Delta} A^{cf}(0)}_{ L^{\infty}_t L^{12}_x} \ls \frac{ \ep}{ (N s^{\frac{1}{2}})^{1-\sigma}} \\
\label{DtPkAcf}
\vn{\pt_t P_k A^{cf}(0)}_{L^{\infty}_t L^2_x} & \ls \ep  2^{-\frac{k}{2}} \jb{ \frac{2^k}{N} }^{2 (1-\sg)} \\
\label{DtPkAdf}
\vn{\pt_t P_k A^{df}(0)}_{L^{\infty}_t L^2_x} & \ls \ep  \jb{ \frac{2^k}{N} }^{1-\sg} \\
\label{PkDtStr}
\vn{P_k \nabla_{t,x} A^{df}(0)}_{L^{2}_t L^{\infty}_x} & \ls 2^{(1+)k}  \jb{ \frac{2^k}{N}}^{1-\sg}  
\end{align}
\begin{align}
\label{esDAdfStr}
\vn{e^{s \Delta} A^{df}(0)}_{L^{2}_t L^{\infty}_x} & \ls s^{0-} / (N s^{\frac{1}{2}})^{1-\sigma} \\
\label{esDL2Linf}
\vn{e^{s \Delta} \pt_t A_i (0)}_{L^{2}_t L^{\infty}_x} & \ls  s^{-\frac{1}{2}+} / (N s^{\frac{1}{2}})^{1-\sigma} \\
\label{ALinftytx}
\vn{A_i (s)}_{L^{\infty}_t  L^{\infty}_x} & \ls s^{-\frac{1}{4}} / (N s^{\frac{1}{2}})^{1-\sigma} \\
\label{Fs2}
\vn{F(s)}_{L^{\infty}_t L^{2}_x} & \ls \ep / (N s^{\frac{1}{2}})^{1-\sigma} \\
 \label{Fsinfty}
\vn{F(s)}_{L^{\infty}_t L^{\infty}_x} & \ls \ep s^{-\frac{3}{4}} / (N s^{\frac{1}{2}})^{1-\sigma} 
\end{align}


\subsubsection{Bounds on the curvature}

\begin{proposition} \label{FbilbdsProp} Under the assumptions above, one has:
\begin{align}
\label{Fbilbd1}
\vn{F^{bil}(s)}_{L^2_t L^{\infty}_x} \ls s^{-\frac{1}{4}+} /  (N s^{\frac{1}{2}})^{3(1-\sigma)} \ep \\
 \label{Fbilbd2}
\vn{F^{bil}(s)}_{L^2_t L^{2}_x} + s^{\frac{1}{2}} \vn{\pt_x F^{bil}(s)}_{L^2_t L^{2}_x}   \ls s^{\frac{1}{2}-\frac{1}{8}} /  (N s^{\frac{1}{2}})^{2(1-\sigma)} \ep
\end{align}
\end{proposition}

\begin{proof}
{\bf Part 1.} We begin with \eqref{Fbilbd1}. 
If $ \mathcal{T} $ is one the three terms in the bracket in \eqref{Fbil}, then we bound either  
$$
\vn{ \int_0^s  e^{(s-s')\Delta} \mathcal{T}(s')  \dd s'  }_{L^2 L^{\infty}}  \quad \text{or} \quad \int_0^s \frac{1}{(s-s')^{\frac{3}{4}}} \vn{ \mathcal{T} (s')}_{L^2 L^{2}} \dd s'
$$
since one has $ (s-s')^{\frac{3}{4}}  e^{(s-s')\Delta} : L^2_x \to L^{\infty}_x $. Let 
\begin{enumerate}
\item $ \mathcal{T} =  [F,F] $, which we place in $ L^2 L^{2} $. For: 

$ [F^{bil},F](s') $ we use \eqref{FbilL2} and \eqref{Fsinfty}, while for 

$ [e^{s' \Delta} \pt_t A_{\ell} (0), F] $, $ [e^{s' \Delta} \pt_x A^{df} (0), F] $ we use \eqref{esDL2Linf} or Strichartz/\eqref{esDAdfStr} and \eqref{Fs2}.



\item $ \mathcal{T} = [ A^{\ell}, \pt_{\ell} F_{\al \beta} ] $ which we separate as follows: 

$ [e^{s' \Delta} A^{df}(0) + A^{bil}(s'), \mathbf{D}_{\ell} F ] $  as $ L^2 L^{\infty} \times  L^{\infty} L^2 \to L^2 L^2 $ by  \eqref{esDAdfStr}, \eqref{AbilL2Linf}. 

$ [e^{s' \Delta} A^{cf}(0), \pt_{x} F^{bil} + e^{s' \Delta} \pt_{x} \pt_{t,x} A^{cf}(0) + e^{s' \Delta} \pt_{x} [A,A](0) ] $ as $   L^{\infty} L^{\infty} \times L^{\infty} L^2 \to L^2 L^2 $ using \eqref{AcfL12}, \eqref{DtPkAcf}, \eqref{FbilL2}.  

$ [e^{s' \Delta} A^{cf}(0),  e^{s' \Delta} \pt_{x} \pt_{t,x} A^{df}(0)] $ which after the $ s' $ integration takes the form 
\be  \label{Wexpterm}
\mathbf{W}(A^{cf}(0), \pt_{x} \pt_{t,x} A^{df}(0)) 
\ee 
where $ \mathbf{W} $ is defined in \eqref{Wdef}. We use Lemma \ref{LemmaWfreq} and the Littlewood Paley trichotomy to bound this term in $ L^2 L^{\infty} $. In the low-high and high-low input frequencies cases we place $ A^{cf} $ in $ L^{\infty} L^{\infty} $ and $ \pt_{t,x}A^{df} $ in $ L^{2} L^{\infty} $ using Strichartz. In the high-high to low case use Bernstein on the output low frequency and $  L^{\infty} L^2  \times L^{2} L^{\infty} \to L^2 L^2 $ with \eqref{AcfL2Pk} and \eqref{PkDtStr}.

\item $ \mathcal{T}=  [A^{\ell}, [A_{\ell},  F ]] $ which is placed in $ L^2 L^{2} $ as follows: 

For $ [e^{s' \Delta} A^{df}(0) + A^{bil}(s'), [A,F](s')] $ as $ L^2 L^{\infty} \times L^{\infty} L^{\infty}  \times L^{\infty} L^2 $ using  \eqref{esDAdfStr}, \eqref{AbilL2Linf}, \eqref{ALinftytx}, \eqref{Fs2}. 

Similarly one deals with $ [ e^{s' \Delta} A^{cf}(0), [ e^{s' \Delta} A^{df}(0) + A^{bil}(s') , F(s') ] ]. $ 

The remaining term $ [e^{s' \Delta} A^{cf}(0), [e^{s' \Delta} A^{cf}(0),F(s')]] $ is bounded as $ L^{\infty} L^{12} \times L^{\infty} L^{12} \times L^{\infty} L^{3} $ using \eqref{AcfL12} and interpolating between \eqref{Fs2} and \eqref{Fsinfty}. 
\end{enumerate}

\

{\bf Part 2.} We continue with \eqref{Fbilbd2}. We bound the bracket from  \eqref{Fbil} in $ L^2 L^2 $. This suffices also for the $ \pt_x F^{bil}_{0 i} $ bound because  $ \pt_x e^{(s_1-s)\Delta} : L^2_x \to (s_1-s)^{-\frac{1}{2}}  L^2_x $.  The majority of terms in \eqref{Fbil} were already estimated in $ L^2 L^2 $ in Part 1. The term that remains to be estimated in $ L^2 L^2 $ is \eqref{Wexpterm}.

Again we use Lemma \ref{LemmaWfreq} and the Littlewood Paley trichotomy. In the high-high to low and high-low cases use $  L^{\infty} L^2  \times L^{2} L^{\infty} \to L^2 L^2 $ with \eqref{AcfL2Pk} and \eqref{PkDtStr}. In the low-high case ($ A^{cf} $ - low frequency) place $ A^{cf}(0) $ in  $ L^{\infty} L^{12+} $ and $  \pt_{t,x} A^{df}(0) $ in $  L^{\infty} L^{\frac{12}{5}-} $ using Bernstein. This last case is responsible for the loss of $ s^{-\frac{1}{8}} $ in \eqref{Fbilbd2}, which can probably be removed. 
\end{proof}

\begin{corollary} \label{CorFbddd} One has:
\be \label{Fbddd}
\vn{F_{\al \beta}(s)}_{L^2_t L^{\infty}_x} \ls \ep s^{-\frac{1}{2}+} /  (N s^{\frac{1}{2}})^{(1-\sigma)} 
\ee 
\end{corollary}

\begin{proof}
From \eqref{Fdec} and the temporal condition $ A_0(s=0)=0 $ one writes
$$
F_{0i} (s) = e^{s \Delta} \pt_t A_i(0) + F^{bil}_{0i} (s) 
$$
and the bound follows from \eqref{esDL2Linf} and \eqref{Fbilbd1}. Similarly 
\begin{align*}
F_{k \ell} (s) & = e^{s \Delta} (\pt_k A_\ell - \pt_{\ell} A_k + [A_{\ell},A_k])(0) + F^{bil}_{k \ell} (s)  \\
& = e^{s \Delta} (\pt_k A_\ell^{df} - \pt_{\ell} A_k^{df} )(0) +  e^{s \Delta} [A_{\ell},A_k](0) + F^{bil}_{k \ell} (s) 
\end{align*}
Now use \eqref{PkDtStr} and  \eqref{Fbilbd1} again. For $ e^{s \Delta} [A_{\ell},A_k](0) $ use either Strichartz when we have $ A^{df} $ or use the better regularity of $ A^{cf} $, together with H\"older and Bernstein. 
\end{proof}

\subsubsection{Bounds on the tension field $w$} \ 

\begin{lemma} \label{wLinfL1}
Under the assumptions above, one has:
\be \label{wLinfL1eq}
\vn{w_x(s)}_{L^{\infty}_t L^1_x} + s^{\frac{1}{2}} \vn{\pt_x w_x(s)}_{L^{\infty}_t L^1_x} \ls s^{\frac{1}{2}} / (N^2 s)^{1-\sigma} \ep
\ee
\end{lemma}

\begin{proof} Fix a time $ t \in [t_0,t_1] $. 
The first inequality was already proved in \eqref{L1bdw}.  

For the second inequality we write schematically 
$$
\pt_x w(s_1) = \int_0^{s_1} \pt_x e^{(s_1-s) \Delta} \lpr [F, \mathbf{D}_{x} F] +  [F, w] + [A, [A,  w ]]+  [A, \pt_{x} w]   \rpr  \dd s
$$
We use the fact that  $ \pt_x e^{(s_1-s)\Delta} : L^1_x \to (s_1-s)^{-\frac{1}{2}}  L^1_x $ so bound the round bracket in $ L^1_x $. The main term is 
$$
 \int_0^{s_1} (s_1-s)^{-\frac{1}{2}} \vn{F}_{L^2_x}   \vn{\mathbf{D}_{x} F}_{L^2_x} \dd s \ls \ep / (N^2 s_1)^{1-\sigma} 
$$
Similarly for $  [F, w]  $ and $ [A, [A,  w ]] $: place $ F(s) $, respectively the $ A(s) $'s, in $ L^{\infty}_x $ and place $ w_x(s) $ in $ L^1_x $ using the first part from \eqref{wLinfL1eq}. 
Finally, for $ [A, \pt_{x} w] $, we place $ A(s) $ in $ L^{\infty}_x $ by \eqref{ALinftytx} and $ \pt_{x} w $ in $ L^1_x $. The resulting term is absorbed to the LHS similarly to \eqref{abstrm}. 
\end{proof}

\begin{lemma} \label{Lw3L2L1} 
Under the assumptions above, one has:
\be \label{w3L2L1}
\vn{w^{(3)}_x(s)}_{L^2_t L^1_x} \ls s^{1-\frac{1}{8}} / (N s^{\frac{1}{2}})^{3(1-\sigma)} \ep 
\ee
\end{lemma}

\begin{proof}
Recall 
$$
w^{(3)}(s_1) \defeq \int_0^{s_1} e^{(s_1-s) \Delta} ( \mathcal{N}^{(3)} (s) +\mathcal{N}^{self} (s) )  \dd s  $$
where $  \mathcal{N}^{(3)} $ and $ \mathcal{N}^{self} $ 
are defined in \eqref{NonlinearityW}. 

We begin with $  \mathcal{N}^{(3)} $:

$ [F^{bil}_{0 \ell}, \mathbf{D}_x F_{0x}] $ is estimated as $ L^2 L^2 \times L^{\infty} L^2 $ using Proposition \ref{FbilbdsProp}. 

$ [ e^{s \Delta} \pt_t A(0), \pt_x  F^{bil}_{0 x} +[A_x,F_{0x}] ]   $ is estimated as $ L^{\infty}  L^2 \times L^2 L^2 $  using Prop. \ref{FbilbdsProp}. 

We continue with $ \mathcal{N}^{self} $: 

$ [F_k^{\ \ell}, w_{\ell}] $ is estimated as $ L^2 L^{\infty} \times L^{\infty} L^1 $ using Corollary \ref{CorFbddd} and Lemma \ref{wLinfL1}. 

$ [e^{s \Delta} A^{df} + A^{bil}, \pt_{x} w] $ seen as $ L^2 L^{\infty} \times L^{\infty} L^1 $ using \eqref{esDAdfStr}, \eqref{AbilL2Linf} and Lemma \ref{wLinfL1}. 

$ [e^{s \Delta} A^{cf}, \pt_{x} w] $ is estimated as $ L^{\infty} L^{12} \times L^{\infty} L^{\frac{12}{11}} $ using \eqref{AcfL12} and 
$$ \vn{\pt_{x} w(s)}_{ L^{\infty} L^{\frac{12}{11}} } \ls \ep s^{-\frac{1}{8}} /  (N s^{\frac{1}{2}})^{2(1-\sigma)}  $$
obtained by interpolating between \eqref{wLinfL1eq} and $ (N^2 s)^{1-\sigma} s^{\frac{3}{4}} \pt_x w(s) \in \ep L^{\infty} L^2  $. 

$ [A^{\ell}, [A_{\ell},  w_k ]]  $ is bounded in several ways, e.g. as $ L^{\infty} L^{\infty} \times L^{\infty} L^{\infty}  \times L^{\infty} L^1 $.
\end{proof}


\section{Trilinear estimates} \label{SecTrilinear}

This section is devoted to the proof of \eqref{deltaEnergyReduced}.

\

We consider a solution which is in:
\begin{enumerate}
\item The temporal gauge at $ s=0 $, i.e. $ A_0(t,x,0) = 0 \ $.
\item The DeTurck gauge $ A_s = \pt^{\ell} A_{\ell} $ on $  [t_0,t_1] \times \ \mb{R}^3 \times [0,s_0]$. 
\item The Coulomb gauge at $ s=0,\ t=t_0 $, i.e. $  \pt^{\ell} A_{\ell} (t_0,x,0)=0$. 
\end{enumerate}

We assume $ A_i(t,x,0)$ obeys the conditions in Prop. \ref{ModLWP} part (2) and Prop. \ref{DeTurckProp} uniformly in $ t $. We assume $ \mathcal{IE}(t) \ls \ep^2 $ and therefore \eqref{HighCovDerF}, \eqref{HighCovDerFw}.

\begin{proposition} \label{TrilinearProp}
Let $ \sg> \frac{5}{6} $ and $ t_1 \leq t_0+1 $. Let $ A_{t,x,s} : [t_0,t_1] \times \mb{R}^3 \times [0,s_0] \to  \mathfrak{g} $ be a regular solution to 
\eqref{YM} at $ s=0 $ and to  \eqref{dYMHF} obeying the gauge conditions $ (1),(2),(3)$ above. Assume the bounds stated above and consequently we assume all the estimates from Section \ref{SecDecAndEst}. 

Then, if $ \ep $ is small enough, one has 
\be \label{TrilwF}
(N^2 s)^{1-\sg}  \vm{  \int_{t_0}^{t_1} \int_{\mathbb{R}^{3}} \chi(x) (w^{\ell}(s), F_{0\ell}(s) )  \dd x \dd t } \ls 
\frac{s^{\frac{1}{4}-} }{  (N s^{\frac{1}{2}})^{(1-\sg)}}  \ep^2
\ee
\end{proposition}

\

After peeling off some terms, the main contribution to \eqref{TrilwF} will come from a null form term. The key estimates needed for this main term are encapsulated by:

\begin{proposition} \label{Coretrilinear}
Denoting $ k_{\max} = \max(k_1,k_2,k_3) $ where $ k_1,k_2,k_3 \geq 0 $ are dyadic frequencies, one has the following inequality for frequency localized functions: 
\be \label{CoreTrilEst}
\vm{ \int_{t_0}^{t_1}  \int_{\mathbb{R}^{3}}   {\bf N} (f_{k_1},g_{k_2}) h_{k_3}  \dd x \dd t } \ls 2^{ (\frac{3}{2}+)  k_{\max} }  \vn{f_{k_1} }_{X^{0,\frac{1}{2}+ }} \vn{g_{k_2}}_{X^{0,\frac{1}{2}+}}  \vn{  h_{k_3}}_{X^{0,\frac{1}{2}+}}
\ee
By translation invariance, the same estimate holds for $  \mathcal{O} {\bf N} $ if $  \mathcal{O} $ is a disposable operator. 
\end{proposition}

Recall the definitions of a null form from \eqref{DefNullForm}, \eqref{Qij}. Moreover, if $ \mathcal{O} $ is a bilinear operator, we denote by $  \mathcal{O}  {\bf N}(f,g) $ a linear combination of 
$$
\Delta^{-1} \nabla^i  \lpr  \mathcal{O}( \pt_i f, \pt_j g) - \mathcal{O}(\pt_j f, \pt_i g ) \rpr  \quad \text{and} \quad   \mathcal{O}(\Delta^{-1} \nabla^i \pt_i f, \pt_j g) - \mathcal{O}(\Delta^{-1} \nabla^i \pt_j f, \pt_i g ) 
$$

\begin{remark}
Proposition \ref{Coretrilinear} states a classical type of multilinear estimate, except for the presence of the sharp time cutoff $ 1_{[t_0,t_1]} $. The proof is essentially contained in \cite{keel2011global}. 
\end{remark}

\subsection{Proof of Proposition \ref{TrilinearProp}} \ 

Splitting $ w(s) $ and $  F_{0\ell}(s) $ by \eqref{Fdec} and \eqref{wDec} we decompose
$$
\int_{t_0}^{t_1}  \int_{\mathbb{R}^{3}} \chi(x) \big(w^{\ell}(s), F_{0\ell}(s) \big)  \dd x \dd t = \mathcal{T}_1 + \mathcal{T}_2 + \mathcal{T}
$$
into two quadrilinear terms and a trilinear term
\begin{align*}
 \mathcal{T}_1 \defeq  \int_{t_0}^{t_1} \int_{\mathbb{R}^{3}} \chi(x) \big(w^{(3)\ell}(s), F_{0\ell}(s) \big)  \dd x \dd t \\
 \mathcal{T}_2 \defeq  \int_{t_0}^{t_1} \int_{\mathbb{R}^{3}} \chi(x) \big(w^{(2)\ell}(s), F_{0\ell}^{bil}(s) \big)  \dd x \dd t \\
 \mathcal{T} \defeq  \int_{t_0}^{t_1} \int_{\mathbb{R}^{3}} \chi(x)  \big(w^{(2) \ell}(s), e^{s \Delta} \pt_t A_{\ell} \big)  \dd x \dd t 
\end{align*}

\

We first estimate the quadrilinear terms by spacetime norms:
\be \label{quadt1}
\vm{ \mathcal{T}_1 } \ls  \int_{t_0}^{t_1} \int_{\mathbb{R}^{3}} \vm{ \chi \big(w^{(3)}(s), F(s) \big)  } \dd x \dd t  \ls \vn{w^{(3)}(s) }_{L^2_t L^1_x} \vn{F(s) }_{L^2_t L^{\infty}_x} \ee
which is bounded by $ s^{\frac{3}{8}-} / (N^2 s)^{2(1-\sigma)} \ep^2 $
using Lemma \ref{Lw3L2L1} and Corollary \ref{CorFbddd},
\be \label{quadt2}
\vm{ \mathcal{T}_2 } \ls  \int_{t_0}^{t_1} \int_{\mathbb{R}^{3}} \vm{ \chi \big(w^{(2)}(s), F^{bil}(s) \big)  } \dd x \dd t  \ls \vn{w^{(2)}(s)}_{L^2_t L^2_x}  \vn{F^{bil}(s) }_{L^2_t L^2_x}
\ee
which is also bounded by $ s^{\frac{3}{8}-} / (N^2 s)^{2(1-\sigma)} \ep^2 $   using \eqref{w2L2} and \eqref{Fbilbd2}. Since $ \sg> \frac{5}{6} $, this is more than enough. 

\

It remains to consider the main term $  \mathcal{T} $. Performing a Littlewood-Paley decomposition we need to bound
$$
 \int_{t_0}^{t_1} \int_{\mathbb{R}^{3}}   \big( P_k w^{(2) \ell}(s), \tilde{P}_k ( \chi e^{s \Delta} \pt_t A_{\ell} ) \big)  \dd x \dd t
$$
To pass the frequency localization onto $  \pt_t A_{\ell} $ write 
$$
\tilde{P}_k ( \chi e^{s \Delta} \pt_t A_{\ell} ) = \tilde{P}_k ( \chi e^{s \Delta} P_k \pt_t A_{\ell} ) +  \tilde{P}_k ( [\chi,P_k] e^{s \Delta} \pt_t A_{\ell} ) 
$$
The commutator term is bounded in \eqref{CommutatorTril}. Then main term is 
$$
 \int_{t_0}^{t_1} \int_{\mathbb{R}^{3}}   \big( P_k w^{(2) \ell}(s),  \chi e^{s \Delta} \pt_t P_k A_{\ell}  \big)  \dd x \dd t.
$$
We split all inputs as $ \pt_t A =  \pt_t A^{df} +  \pt_t A^{cf} $. For the dominant term, with $ 3 $ divergence-free inputs, we further split $ w^{(2)} =  {\bf P}  w^{(2)} + {\bf P}^{\perp} w^{(2)}$ obtaining:

\be  \label{3dfNullTerm}
\int_{t_0}^{t_1} \int_{\mathbb{R}^{3}} \big(P_k {\bf P}  w^{(2)}(\pt_t A^{df}, \pt_t A^{df}, s),  \chi e^{s \Delta}  \pt_t A^{df}_k  \big)  \dd x \dd t 
\ee
which contains a null form, and the following commutator-type error term
\be \label{3dfCommutator} 
\int_{t_0}^{t_1} \int_{\mathbb{R}^{3}} \big(P_k {\bf P}^{\perp}   w^{(2)}(\pt_t A^{df}, \pt_t A^{df}, s), \tilde{P}_k ( \chi e^{s \Delta}  \pt_t A^{df}_k ) \big)  \dd x \dd t 
\ee
Then we have terms with $ 1 $ curl-free input and $ 2 $ divergence-free inputs:
\begin{align} \label{dfdfcf}
 \int_{t_0}^{t_1} \int_{\mathbb{R}^{3}}   \big( P_k w^{(2)}(\pt_t A^{df}, \pt_t A^{df}, s),  \chi e^{s \Delta} \pt_t A_{k}^{cf} \big)  \dd x \dd t \\ 
 \label{cfdfdf}
  \int_{t_0}^{t_1} \int_{\mathbb{R}^{3}}   \big( P_k w^{(2)}(\pt_t A^{cf}, \pt_t A^{df}, s),  \chi e^{s \Delta} \pt_t A_{k}^{df} \big)  \dd x \dd t \\
  \label{dfcfdf}
   \int_{t_0}^{t_1} \int_{\mathbb{R}^{3}}   \big( P_k w^{(2)}(\pt_t A^{df}, \pt_t A^{cf}, s),  \chi e^{s \Delta} \pt_t A_{k}^{df} \big)  \dd x \dd t 
\end{align}
terms with $ 2 $ curl-free inputs and $ 1 $ divergence-free input:
\begin{align} \label{cfcfdf}
 \int_{t_0}^{t_1} \int_{\mathbb{R}^{3}}   \big( P_k w^{(2)}(\pt_t A^{cf}, \pt_t A^{cf}, s),  \chi e^{s \Delta} \pt_t A_{k}^{df} \big)  \dd x \dd t \\ 
  \int_{t_0}^{t_1} \int_{\mathbb{R}^{3}}   \big( P_k w^{(2)}(\pt_t A^{cf}, \pt_t A^{df}, s),  \chi e^{s \Delta} \pt_t A_{k}^{cf} \big)  \dd x \dd t \\
   \int_{t_0}^{t_1} \int_{\mathbb{R}^{3}}   \big( P_k w^{(2)}(\pt_t A^{df}, \pt_t A^{cf}, s),  \chi e^{s \Delta} \pt_t A_{k}^{cf} \big)  \dd x \dd t 
\end{align}
and finally a term with $ 3 $ curl-free inputs
\be \label{cfcfcf}
 \int_{t_0}^{t_1} \int_{\mathbb{R}^{3}}   \big( P_k w^{(2)}(\pt_t A^{cf}, \pt_t A^{cf}, s),  \chi e^{s \Delta} \pt_t A_{k}^{cf} \big)  \dd x \dd t.  
\ee

\

\

\subsubsection{The main term \eqref{3dfNullTerm}} \

We proceed to specifying the null form, recalling \eqref{w2form}. 
From the identities \eqref{Nullformidentities} and Lemma \ref{LemmaWfreq} one can write 
\be \label{SpecW}
P_k {\bf P}  \mathbf{W} (   {\bf P}  \pt_t A^{\ell}_{k_1},  \pt_x \pt_t A^{df}_{k_2,\ell} - 2 \pt_{\ell} \pt_t A_{k_2,x}^{df} )
\ee 
as
\be \label{SpecNull}
\left\langle s 2^{2 k}\right\rangle^{-10}\left\langle s^{-1} 2^{-2 k_{\max }}\right\rangle^{-1} 2^{-2 k_{\max }}  P_k \mathcal{O} {\bf N} (  \pt_t A^{df}_{k_1}, \pt_t A^{df}_{k_2} )  
\ee
where $ \mathcal{O} $ is disposable. 

Since we assume the conditions in Proposition \ref{ModLWP} part (2) we have $ I \pt_t A^{df} \in \ep X^{0,\frac{3}{4}+} $. Additionally, from Lemma \ref{cutoffXsb} 
$$ \chi e^{s \Delta} P_k \pt_t A^{df}  \in \ep  \jb{ 2^k / N }^{1-\sg} X^{0,\frac{3}{4}+} $$  

We apply Prop. \ref{Coretrilinear} to bound the contribution of \eqref{3dfNullTerm} to the LHS of \eqref{TrilwF} by 
$$
(N^2 s)^{1-\sg} 2^{ (-\frac{1}{2}-)  k_{\max} }
\left\langle s 2^{2 k}\right\rangle^{-10}\left\langle s^{-1} 2^{-2 k_{\max }}\right\rangle^{-1} 
 \jb{ \frac{2^{k}} { N} }^{1-\sg}  \jb{ \frac{2^{k_1}} { N} }^{1-\sg}  \jb{ \frac{2^{k_2}} { N} }^{1-\sg}
\ep^3 
 $$
After summation in $ k_1,k_2, k $ the result is $ \ls $  RHS of  \eqref{TrilwF}. \qed

\

We next estimate the terms \eqref{dfdfcf} - \eqref{cfcfcf}, beginning by decomposing them into
\be \label{WLPdec}
(N^2 s)^{1-\sg} \int_{t_0}^{t_1} \int_{\mathbb{R}^{3}}   ( P_k  \mathbf{W} (\pt_t A^{(1)}_{k_1}, \nabla_x \pt_t A^{(2)}_{k_2}, s),  \chi e^{s \Delta} \pt_t A_{k}^{(3)} )  \dd x \dd t 
\ee

\

\subsubsection{The terms \eqref{cfcfdf}-\eqref{cfcfcf} with low-high inputs in $ w $} \

Here we consider the contributions of \eqref{WLPdec}
when $ 2^k \simeq 2^{k_{\max}} $ and at most one of $ A^{(1)}, A^{(2)}, A^{(3)} $ is $ A^{df} $. 
By Lemma \ref{LemmaWfreq} and translation-invariance, without loss of generality it suffices to bound 
 \be \label{LowHiinW}
2^{-k}  \int_{t_0}^{t_1} \int_{\mathbb{R}^{3}}    \pt_t A^{(1)}_{k_1} \cdot \pt_t A^{(2)}_{k_2}  \cdot P_k (\chi e^{s \Delta} \pt_t A_{k}^{(3)} )   \dd x  \dd t  \times \left\langle s 2^{2 k}\right\rangle^{-10}\left\langle s^{-1} 2^{-2 k}\right\rangle^{-1}       (N^2 s)^{1-\sg}
 \ee
when $ 2^{k_1} \ls 2^{k_2} \simeq 2^{k} $.  We bound this term as $ L^2 L^{\infty} \times L^{\infty} L^2 \times L^{\infty} L^2 $,  the low frequency in $ L^2 L^{\infty} $ using \eqref{PkDtAcfL2Linfx}, \eqref{PkDtStr}, \eqref{DtPkAcf}, \eqref{DtPkAdf}. After summing in $ k_1 $, and then in $ k $ using the $ \left\langle s 2^{2 k}\right\rangle^{-10}\left\langle s^{-1} 2^{-2 k}\right\rangle^{-1} $ factors   one obtains a bound of 
$
s^{\frac{1}{2}-} (N s^{\frac{1}{2}})^{-3(1-\sg)}.
$

\subsubsection{The terms \eqref{cfcfdf}-\eqref{cfcfcf} with high-high inputs in $ w $} \

Here we consider the contributions of \eqref{WLPdec}
when $ 2^k \ll  2^{k_1} \simeq 2^{k_2} $ and at most one of $ A^{(1)}, A^{(2)}, A^{(3)} $ is $ A^{df} $. 
By Lemma \ref{LemmaWfreq}, without loss of generality it suffices to bound 
 \be \label{HiHiinW}
2^{-{k_2}}  \int_{t_0}^{t_1} \int_{\mathbb{R}^{3}}   P_k ( \pt_t A^{(1)}_{k_1} \cdot \pt_t A^{(2)}_{k_2})  \chi e^{s \Delta} \pt_t A_{k}^{(3)}   \dd x  \dd t  \times \left\langle s 2^{2 k}\right\rangle^{-10}\left\langle s^{-1} 2^{-2 {k_2}}\right\rangle^{-1}       (N^2 s)^{1-\sg}
 \ee
We use the norms of $ L^{\infty} L^2, L^2 L^2 $ and $ L^2 L^{\infty} $ as follows: a high frequency $ \pt_t A^{cf} $ is always placed in $ L^2 L^2 $ using \eqref{PkDtAcfL2L2}, the term $  \pt_t A^{df} $ (if present) is placed in $ L^2 L^{\infty} $ using \eqref{PkDtStr} (otherwise place the low frequency  $ \pt_t A^{cf} $ in $ L^2 L^{\infty} $ using \eqref{PkDtAcfL2Linfx}) and the remaining $ \pt_t A^{cf} $ term is placed in $ L^{\infty} L^2 $ by \eqref{DtPkAcf}. After summation we obtain either $ s^{\frac{5}{8}-} (N s^{\frac{1}{2}})^{-4(1-\sg)} $ or $ s^{\frac{3}{8}-} (N s^{\frac{1}{2}})^{-(1-\sg)} $. 

\subsubsection{The terms \eqref{dfdfcf}-\eqref{dfcfdf} with low-high inputs in $ w $} \

Now consider the contributions of \eqref{WLPdec}
when $ 2^k \simeq 2^{k_{\max}} $ and two of $ A^{(1)}, A^{(2)}, A^{(3)} $ are $ A^{df} $ and one is $ A^{cf} $. Again it suffices to bound \eqref{LowHiinW} when $ 2^{k_1} \ls 2^{k_2} \simeq 2^{k} $. We place the $ \pt_t A^{cf} $ term in $ L^2 L^2 $ using \eqref{PkDtAcfL2L2}. In the cases of \eqref{dfdfcf} and \eqref{dfcfdf}  when $ \pt_t A^{df} $ is low frequency, place this term in $ L^2 L^{\infty} $ using \eqref{PkDtStr} and the other term in $ L^{\infty} L^2 $. In the case of \eqref{cfdfdf} place the product of the two high frequency df terms in $ L^2 L^2 $ using \eqref{BilStrichartz}. After summing in $ k_1 $, and then in $ k $ using the $ \left\langle s 2^{2 k}\right\rangle^{-10}\left\langle s^{-1} 2^{-2 k}\right\rangle^{-1} $ factors  one obtains a bound of 
$
s^{\frac{3}{8}-} (N s^{\frac{1}{2}})^{-2(1-\sg)}
$
or $s^{\frac{1}{4}-} $.

\subsubsection{The terms \eqref{dfdfcf}-\eqref{dfcfdf} with high-high inputs in $ w $} \

Here we consider the contributions of \eqref{WLPdec}
when $ 2^k \ll  2^{k_1} \simeq 2^{k_2} $ and  two of $ A^{(1)}, A^{(2)}, A^{(3)} $ are $ A^{df} $ and one is $ A^{cf} $. Again it suffices to bound \eqref{HiHiinW} and we place the $ \pt_t A^{cf} $ term in $ L^2 L^2 $ using \eqref{PkDtAcfL2L2}. In the case of  \eqref{dfdfcf} we place the product of the two high frequency df terms in $ L^2 L^2 $ using \eqref{BilStrichartz}. In the cases of \eqref{cfdfdf}, \eqref{dfcfdf}, place the low frequency $ \pt_t A^{df} $ in $ L^2 L^{\infty} $ using \eqref{PkDtStr} and the high frequency $ \pt_t A^{df} $ in $  L^{\infty} L^2$. In all cases, after summations, we obtain a bound slightly better than $ s^{\frac{1}{4}-} (N s^{\frac{1}{2}})^{-(1-\sg)} $.  \qed

\

We continue with the commutator-type terms \eqref{CommutatorTril} and \eqref{3dfCommutator}, but first let us make some preparations. Recall the commutator identity \cite[Lemma 2]{tao2001global}:
\be \label{PkCommutator}
[P_k, \chi] f = 2^{-k} L(\nabla \chi, f)
\ee
where $ L $ is a disposable operator. It's possible to pass $ I $ bounds from $ f $ to $ L( \nabla \chi, f)$:

\begin{lemma} \label{LchifI}
Let $ \chi' $ be a bump function and $ L $ a disposable operator. One has
$$
\vn{P_k L( \chi', f)}_{L^2_x} \ls \jb{\frac{2^k}{N}}^{1-\sg} \vn{I f}_{L^2_x} 
$$
\end{lemma}
\begin{proof}
By translation-invariance and the integrable kernel, it suffices to prove this for $ P_k( \chi' f) $ which we separate as $ P_k( \chi' f_{<k+c}) $ and $  P_k( \chi'_{>k+c-3} f_{>k+c}) $. Normalize $ \vn{I f}_{L^2_x} = 1 $. Then
\begin{align*}
\vn{ P_k( \chi' f) }_{L^2_x} & \ls \vn{f_{<k+c}}_{L^2_x} + \sum_{k_2=k_1+O(1)\geq k+c} \vn{ \chi'_{k_1}}_{L^{\infty}_x } \vn{f_{k_2}}_{L^2_x} \\
& \ls \jb{2^k/N}^{1-\sg} + \sum_{k_2 \geq k+c} 2^{-5 k_2} \jb{2^{k_2}/N}^{1-\sg} \ls \jb{2^k/N}^{1-\sg}.
\end{align*}
\end{proof}

\begin{lemma} \label{CommMD}
Let $ f_k $ be a function localized at frequencies $ \jb{\xi} \simeq 2^k $, $ \chi $ be a bump function and let $ m(\xi) $ be a homogeneous multiplier of degree zero. Then
$$
\vn{P_k [m(D), \chi] f_k}_{L^2_x} \ls 2^{-k} \vn{f_k}_{L^2_x}.  
$$
\end{lemma}
\begin{proof}
Note that
$$
P_k [m(D), \chi] f_k = P_k [m(D), \chi_{<k-3}] f_k + P_k [m(D), P_{[k-3,k+3]} \chi] f_k
$$
For the latter term, bounds the two terms of this commutator separately as
$$
\vn{ P_k [m(D), P_{[k-3,k+3]} \chi] f_k  }_{L^2_x}   \ls \vn{P_{[k-3,k+3]} \chi }_{ L^{\infty}_x} \vn{f_k}_{L^2_x} \leq 2^{-10k} \vn{f_k}_{L^2_x}. 
$$
For the former term, denoting $ T f_k = [m(D), \chi_{<k-3}] f_k $, 
\begin{align*}
\widehat{T f_k}(\xi) & = \int  [m(\xi) - m(\xi-\eta)] \hat{\chi}_{<k-3} (\eta) \hat{f_k}(\xi-\eta)         \dd \eta  \\
& = \int_0^1 \int_{\eta \ll 2^k} \nabla m(\xi-t \eta) \eta \hat{\chi}_{<k-3} (\eta) \hat{f_k}(\xi-\eta) \dd \eta \dd t 
\end{align*}
Apply $ P_k $ and take the $ L^2_{\xi} $ norm, use Minkowski's inequality, take absolute values, bound  
\newline   
$ \vm{ \nabla m(\xi-t \eta)} \ls 2^{-k} $ and use the integrability of $ \vm{ \eta} \hat{\chi}_{<k-3} (\eta) $.
\end{proof}

\subsubsection{Proof of the first commutator bound:}
\be \label{CommutatorTril}
 (N^2 s)^{1-\sg} \sum_k \vm{ \int_{t_0}^{t_1} \int_{\mathbb{R}^{3}}   ( P_k w^{(2) \ell}(s),  \tilde{P}_k ( [\chi,P_k] e^{s \Delta} \pt_t A_{\ell} )  )  \dd x \dd t } 
 \ls \frac{s^{\frac{1}{4}-} }{  (N s^{\frac{1}{2}})^{(1-\sg)}}  \ep^2
\ee

We bound $  P_k w^{(2)}(s)$    in $ L^2_t L^2_x $ using \eqref{Pkw2L2}. The term $  \tilde{P}_k ( [\chi,P_k] e^{s \Delta} \pt_t A_{\ell} ) $ is viewed as  
$$ 2^{-k} \tilde{P}_k L(\nabla \chi, e^{s \Delta} \pt_t A) $$ 
using \eqref{PkCommutator} which we bound using Lemma \ref{LchifI} as 
$$
2^{-k} \vn{ \tilde{P}_k L(\nabla \chi, e^{s \Delta} \pt_t A)}_{L^{\infty}_t L^2_x} \ls \ep 2^{-k} \jb{2^k/N}^{1-\sg}.
$$
After summing in $ k $ we bound the product by the RHS of \eqref{CommutatorTril}. 

\

\subsubsection{The second commutator bound \eqref{3dfCommutator} } We prove that 
$$
 (N^2 s)^{1-\sg} \sum_k 
 \vm{\int_{t_0}^{t_1} \int_{\mathbb{R}^{3}} (P_k {\bf P}^{\perp}   w^{(2)}(\pt_t A^{df}, \pt_t A^{df}, s), \tilde{P}_k ( \chi e^{s \Delta}  \pt_t A^{df}_k ) )  \dd x \dd t }
$$
is $ \ls $ the RHS of \eqref{CommutatorTril}. Since $ {\bf P}^{\perp} \pt_t A^{df} = 0 $ one can write 
$$ {\bf P}^{\perp} \chi e^{s \Delta}  \pt_t A^{df}_k =  [{\bf P}^{\perp}, \chi] e^{s \Delta}  \pt_t A^{df}_k
$$
By Lemma \ref{CommMD} 
$$
\vn{ \tilde{P}_k [{\bf P}^{\perp}, \chi] e^{s \Delta} \pt_t A^{df}_k }_{L^{\infty}_t L^2_x} \ls 
\ep 2^{-k} \jb{2^k/N}^{1-\sg}
$$
By Lemma \ref{Pkw2L2} and Remark \ref{RmkPkw2L2} we obtain an $ L^2_t L^2_x $ bound for $ P_k w^{(2)}(\pt_t A^{df}, \pt_t A^{df}, s) $. The numerology is the same as in the proof of \eqref{CommutatorTril}.  \qed

\

\subsection{Proof of Proposition \ref{Coretrilinear}} \

Integrating by parts if needed, one may assume without loss of generality that 
$ k_1 = \min(k_1,k_2,k_3) $. One may also assume that the Fourier transforms on $ \mb{R}\times \mb{R}^3 $ of the inputs are real and non-negative, as one then estimates
$$ 
\int_{\mb{R}\times \mb{R}^3 } 1_{[t_0,t_1]}(t)   {\bf N} (f_{k_1},g_{k_2}) h_{k_3}  \dd x \dd t 
$$
using Plancharel's identity by
$$
\int  \frac{ \vm{\xi_1 \wedge \xi_2} 
}{\vm{\xi_1} \jb{\tau_1+\tau_2+\tau_3}} 
\hat{f}_{k_1}(\tau_1,\xi_1)  \hat{g}_{k_2}(\tau_2,\xi_2)  \hat{h}_{k_3}(\tau_3,\xi_3) 
$$
since $ \vm{ \hat{1}_{[t_0,t_1]}(\tau) } \ls \jb{\tau}^{-1} $. The integral is taken over the subset of 
$$ (\tau_1,\xi_1, \tau_2,\xi_2,\tau_3,\xi_3) \in (\mb{R}\times \mb{R}^3)^3 $$ with $ \xi_1+\xi_2+\xi_3 = 0 $.  Denote the modulations by $ \lmd_i \defeq \jb{\vm{\xi}_i-\vm{\tau}_i} $. 

Denoting $ a_{\al}(t) $ the Fourier transform of $ \jb{\tau}^{-\al} $, we recall
\be \label{aalphabd}
a_{\al}(t) \in L^p_t, \qquad \forall \ p< (1-\al)^{-1}, \ \al \in (0,1]
\ee
from \cite[Lemma 15.2]{keel2011global}. We also recall 
\be \label{wedgebd}
\vm{\xi_1 \wedge \xi_2} \ls 2^{\frac{k_1}{2}} 2^{\frac{k_2}{2}} 2^{\frac{k_3}{2}} ( \jb{\tau_1+\tau_2+\tau_3} + \lmd_1 +\lmd_2+\lmd_3)^{\frac{1}{2}}
\ee
from \cite[Prop. 1]{klainerman1996estimates} or \cite[Lemma 15.1]{keel2011global}.

Then it suffices to consider three cases and bound \eqref{Est1}, \eqref{Est2}, \eqref{Est3} as follows.

{\bf (1)}  Bound

\be \label{Est1}
2^{k_2-\frac{k_1}{2}}  \int  \frac{1}{\jb{\tau_1+\tau_2+\tau_3}^{\frac{1}{2}}} 
\hat{f}_{k_1}(\tau_1,\xi_1)  \hat{g}_{k_2}(\tau_2,\xi_2)  \hat{h}_{k_3}(\tau_3,\xi_3) 
\ee
by the RHS of \eqref{CoreTrilEst}. Undoing the Fourier transform one proves
$$ 
2^{k_2-\frac{k_1}{2}} \vm{\int_{\mb{R}\times \mb{R}^3 } a_{\frac{1}{2}}(t)  f_{k_1} g_{k_2} h_{k_3} \dd x \dd t } \ls RHS \eqref{CoreTrilEst}
$$
which is bounded as $  L^{2-}_t L^{\infty}_x \times  L^{2+}_t L^{\infty}_x \times L^{\infty}_t L^2_x  \times L^{\infty}_t L^2_x $ using \eqref{aalphabd} and Strichartz estimates \eqref{Strichartzz}. 

{\bf (2)}
When $ \lmd_1 $ dominates in \eqref{wedgebd}, by dyadic decomposition and Cauchy-Schwarz, it suffices to integrate on the subregion $ \lmd_2, \lmd_3 \ls 2^j \simeq \lmd_1 $ and prove
\be \label{Est2}
\begin{aligned}
2^{-\frac{k_1}{2}} \int  \frac{1}{\jb{\tau_1+\tau_2+\tau_3}}  \hat{u}_{k_1,j}(\tau_1,\xi_1) & \hat{g}_{k_2}(\tau_2,\xi_2)   \hat{h}_{k_3}(\tau_3,\xi_3)  \ls  \\
 2^{ (\frac{1}{2}+)  k_{2} }  & \vn{ u_{k_1,j} }_{X^{0,0+ }} \vn{g_{k_2}}_{X^{0,\frac{1}{2}+}}  \vn{  h_{k_3}}_{X^{0,\frac{1}{2}+}}
\end{aligned}
\ee
where we denote $  \hat{u}_{k_1,j}=\lmd_1^{\frac{1}{2}} \hat{f}_{k_1} $. It suffices to prove
$$ 
\vm{\int_{\mb{R}\times \mb{R}^3 } a_1(t) u_{k_1,j} D^{-(\frac{1}{2}-)} (g_{k_2} h_{k_3}) \dd x \dd t } \ls RHS \eqref{Est2}
$$
Use \eqref{aalphabd} to place $ a_1(t) $ in $ L^{\infty-}_t L^{\infty}_x $, the embedding \eqref{X00emb} for $  u_{k_1,j} $  and the bilinear Strichartz estimate \eqref{BilStrichartz} for the product.

{\bf (3)}
In the remaining case, we can assume without loss of generality that $ \lmd_2 $ dominates and integrate on the subregion  $ \lmd_1, \lmd_3 \ls 2^j \simeq \lmd_2 $, proving 
\be \label{Est3}
\begin{aligned}
 \int  \frac{1}{\jb{\tau_1+\tau_2+\tau_3}}  \hat{f}_{k_1}(\tau_1,\xi_1)  \hat{v}_{k_2,j}(\tau_2,\xi_2)  &\hat{h}_{k_3}(\tau_3,\xi_3) \ls  \\
 2^{\frac{k_1}{2}} 2^{ (\frac{1}{2}+)  k_{2} }  & \vn{f_{k_1}}_{X^{0,\frac{1}{2}+ }} \vn{\hat{v}_{k_2,j}}_{ L^2 L^2 }  \vn{  h_{k_3}}_{X^{0,\frac{1}{2}+}}
 \end{aligned}
\ee
where we denote $  \hat{v}_{k_2,j}=\lmd_2^{\frac{1}{2}} \hat{g}_{k_2} $. Undoing the Fourier transform this is reduced to 
$$ 
\vm{\int_{\mb{R}\times \mb{R}^3 } a_1(t)  f_{k_1}  v_{k_2,j}h_{k_3} \dd x \dd t } \ls RHS \eqref{Est3}
$$
which is bounded as $  L^{\infty-}_t L^{\infty}_x \times L^4_t L^4_x \times L^2_t L^2_x  \times L^{4+}_t L^4_x $ using \eqref{aalphabd} and Strichartz estimates \eqref{Strichartzz}. 


\appendix 

\section{Modified energy for Maxwell-Klein-Gordon}

In this appendix we formulate the gauge invariant modified energy for the Maxwell-Klein-Gordon equation, we show that the same favorable balance of derivatives as for Yang-Mills is present when it is differentiated in time (see Remark \ref{ComparisonAC}) and we sketch how this can be used to increase the range of regularities for which global well-posedness is known to hold, from $ \sg > \frac{\sqrt{3}}{2} $ in \cite{keel2011global} to $ \sg> \frac{5}{6} $. 

\

The Yang-Mills equations in the case of the abelian group $ G=SO(2)=U(1) $, $ \mathfrak{g}= \mathfrak{so}(2) $  correspond to Maxwell's equations, where $ A_{\al} $ can be viewed as a real one-form $ A_{\al}: \mb{R}^{1+3} \to \mb{R} $. Coupling with a complex field $ \phi : \mb{R}^{1+3} \to \mb{C} $ one obtains the (massless) Maxwell-Klein-Gordon system \cite{klainerman1994maxwell},  \cite{machedon2004almost}, \cite{keel2011global},  
as the Euler-Lagrange equations for the Lagrangian 
$$
 \frac{1}{2}  \int_{\mb{R}^{1+3}}D_{\al} \phi \overline{D^{\al} \phi} +  \frac{1}{2} F_{\al \beta} F^{\al \beta}  \dd x \dd t 
$$
Here one denotes the covariant derivative $ D_{\al} $ and the curvature $ F_{\al \beta} $ by 
$$
 D_{\al} \phi \defeq (\pt_{\al}+i A_{\al}) \phi, \qquad   F_{\al \beta} \defeq \pt_{\al} A_{\beta} - \pt_{\beta} A_{\al}. 
$$
The system becomes
\be \tag{MKG} \label{MKG}
\begin{aligned}
\pt^{\beta} F_{\al \beta} & = \Im (\phi  \overline{D_{\al} \phi}) \\
\Box_A \phi  & = 0. 
\end{aligned}
\ee
where we denote $\Box_A \defeq D^{\al} D_{\al}$. 
The gauge invariance takes the form
$$
\phi \mapsto e^{i \chi} \phi, \qquad A_{\al} \mapsto A_{\al}- \pt_{\al} \chi, \qquad \chi :   \mb{R}^{1+3} \to \mb{R}
$$
while the energy is
\be \label{MKGEnergy}
H[A,\phi](t) \defeq \frac{1}{2} \int_{ \mb{R}^3} \sum_{\al < \beta } \vm{F_{\al \beta}}^2 +  \sum_{\al} \vm{D_{\al} \phi}^2 \dd x .
\ee
Initial data sets consist of $ (\underline{A}_i, \underline{E}_i, \phi_0, \phi_1) = (A_i, F_{0i}, \phi, D_t \phi)(t=0) $ satisfying the constraint equation $ \pt^i \underline{E}_i =  \Im (\phi_0 \overline{\phi_1})  $. 

\

In \cite{keel2011global} Keel, Roy and Tao prove global well-posedness in the (global) Coulomb gauge $ \pt^i A_i =0 $. In this gauge, $ A_0 $ is determined from the initial data. 
\begin{theorem}[\cite{keel2011global}]
Let $ \sg > \frac{\sqrt{3}}{2} $. The Maxwell-Klein-Gordon equation in the Coulomb gauge is globally well-posed for initial data in $ H^{\sg} \times H^{\sg-1} \times H^{\sg} \times H^{\sg-1} $. 
\end{theorem}
After reducing to global smooth solutions, rescaling and choosing $ N $ such that $ H[IA,I\phi](0) \leq 1 $, their proof proceeds by proving the almost conservation law 
\be \label{ACLawMKGwI}
H[I \Phi](t) = H[I \Phi](t_0) + O(N^{\frac{1}{2}-\sg+}) 
\ee
under the assumption $ H[I \Phi](t_0) \leq 2 $, where one denotes $ \Phi=(A_0,A_x, \phi) $. 

The spacetime control on small intervals is $ \nabla_{t,x}  (I A_x, I \phi) \in X^{0,\sg-} + L^{\infty} L^3 $ where the latter part consists of low frequencies. 

For arbitrary fields \cite{keel2011global}[Lemma 11.1] writes the differentiated Hamiltonian as
\be \label{DiffHam}
\frac{\dd}{\dd t} H[\Phi](t) = - \Re \int_{ \mb{R}^3} D_t \phi \overline{\Box_A \phi} + F_{\mu 0} (\pt^{\al} F_{\al}^{\ \mu} + \Im (\phi \overline{D^{\mu} \phi})  \dd x
\ee
The dominant part of $ H[I \Phi](t_1) - H[I \Phi](t_0) $ is proved in \cite[(106),(124)]{keel2011global} to be \eqref{MKGTril}. 

\

We now formulate the invariant modified energy for \eqref{MKG}. 

The gradient flow for the functional 
$$
\frac{1}{2} \int_{ \mb{R}^3} \sum_{i < j } \vm{F_{ij}}^2 +  \sum_{i} \vm{D_{i} \phi}^2 \dd x .
$$
gives rise to the equations $ \pt_s A_i = \pt^{\ell} F_{\ell i}+  \Im (\phi  \overline{D_{i} \phi}) $ and $ \pt_s \phi=D^{\ell} D_{\ell} \phi $. Adding temporal and parabolic connections $ A_0, A_s $ and writing the system in a gauge-invariant way one obtains a Maxwell-Klein-Gordon Heat Flow
\be \tag{MKG-HF} \label{MKGHF}
\begin{aligned}
& F_{s \al}  = \pt^{\ell} F_{\ell \al } + \Im (\phi  \overline{D_{\al} \phi}) \\
& (D_s - \Delta_A) \phi  = 0. 
\end{aligned}
\ee
The extent of failure of $ A_{\al} (s), \phi(s) $ to satisfy \eqref{MKG} for $ s>0 $ is measured by the tension fields 
$$
v \defeq \Box_A \phi, \qquad  w_{\beta} \defeq \pt^{\al} F_{\al \beta} +  \Im (\phi  \overline{D_{\beta} \phi}). 
$$
Up to lower order terms, the main components of the parabolic equations satisfied by the tension fields can be written as  
\be  \label{TnsEq}
\begin{aligned}
 (\pt_s - \Delta)  {\bf P}_j w & = -2  {\bf P}_j \Im ( \pt_t \phi  \overline{\pt_{x} \pt_t \phi})  + \dots      \\
 (D_s - \Delta_A) v & = 4i \pt_t  {\bf P}^{j} A \pt_j \pt_t \phi + \dots  
\end{aligned}
\ee
The dots consist of trilinear terms or terms with higher regularity, such as those involving $ A^{cf} $ or $ A_0 $. 
Assuming deTurck's gauge $ A_s = \pt^{\ell} A_{\ell} $, expanding and decomposing \eqref{TnsEq} similarly to Section \ref{SecDecAndEst} 
one obtains the leading terms
\be  \label{Decwv}
\begin{aligned}
{\bf P}_j w(s) & = -2  {\bf P}_j  \Im \mathbf{W}( \pt_t \phi , \overline{\pt_{x} \pt_t \phi})(s)  + \dots      \\
 v(s) & = 4i \mathbf{W} ( \pt_t  {\bf P}^{j} A, \pt_j \pt_t \phi)(s)+ \dots  
\end{aligned}
\ee
where $ \mathbf{W} $ is defined by \eqref{Wdef} and the inputs are taken at $ s=0 $. 

\

For a regular global solution of \eqref{MKG} we define the modified energy similarly to \eqref{ModifiedEnergy}. Denote the energies \eqref{MKGEnergy} defined by $ A_{\al} (t, \cdot, s), \phi(t, \cdot,s) $ by $ H(t,s) $. Then, for $ s_0=N^{-2} $, let 
\be \label{MKGModEn}
\mathcal{I H}(t) \defeq \sup_{s \in [0,s_0]} (N^2 s)^{1-\sigma} H(t,s) + \int_0^{s_0} (N^2 s)^{1-\sigma}  H(t,s) \frac{\dd s}{s} 
\ee
The goal is to have $ \vm{\mathcal{IH}(t) - \mathcal{IH}(t_0) } \ll N^{-\frac{1}{2}+}  $ like in Proposition \ref{AClaw} as this can be used to obtain global well-posedness for \eqref{MKG} as in Section \ref{RedMainThm}, with the name numerology. The $ N^{-\frac{1}{2}+} $ bound compared to \eqref{ACLawMKGwI} reduces the minimum initial data regularity from $ \sg > \frac{\sqrt{3}}{2} $ to $ \sg > \frac{5}{6} $.

\

Regarding the choices of gauge there are several options and we avoid making a definitive choice. One may set the temporal gauge $ A_0 =0 $. Working with local gauges and extensions like we did for Yang-Mills is probably not necessary, but can be a natural choice in view of potential extensions to Yang-Mills-Higgs (which generalizes both \eqref{YM} and \eqref{MKG}). This also avoids some low frequency technicalities. On the other hand, after fixing the interval $ [t_0,t_1] $, one could set the global Coulomb gauge at $ s=s_0, t=t_0 $ and obtain an $ H^1+N^{1-\sg} H^{\sg} $ bound at $s=0 $ from the caloric condition like in Section \ref{SecGauge}, but globally instead of locally. 

The key spacetime bound remains $ \nabla_{t,x}  (IA_x^{df}, I\phi) \in X^{0,\frac{3}{4}+} $, perhaps with a low frequency component. 

To estimate the variation of the modified energy one writes similarly to \eqref{DeltaE}, using \eqref{DiffHam}
\be \label{IntDiffHam}
H(t_1,s)- H(t_0,s)=  - \Re  \int_{t_0}^{t_1} \int_{ \mb{R}^3} D_t \phi \overline{v} + F_{j 0} w_j  \dd x \dd t 
\ee
As \eqref{quadt1}, \eqref{quadt2} show for quadrilinear terms and \eqref{dfdfcf}-\eqref{cfcfcf} show for terms involving $ A^{cf} $, in the case of temporal gauge + DeTurck gauge, 
all these terms can be fairly easily bounded using spacetime Lebesgue norms. In the case of the Coulomb gauge, sections 12 and 14 in \cite{keel2011global} show that such terms are not problematic. 

Up to error terms, \eqref{IntDiffHam} is written, from \eqref{Decwv} and the identities \eqref{Nullformidentities}, as 
\be
 \Re  \int_{t_0}^{t_1} \int_{ \mb{R}^3} 4i \pt_t \phi \mathbf{W} {\bf N} ( \pt_t  A^{df}, \overline{\pt_t \phi})(s)  + 2 \pt_t A^{df} \Im \mathbf{W}  {\bf N} ( \pt_t \phi , \overline{\pt_t \phi})(s)  \dd x \dd t  + \dots 
\ee
Schematically, these terms are exactly like \eqref{3dfNullTerm} together with \eqref{SpecNull} and are estimated using Proposition \ref{Coretrilinear}.


\nocite{*}
\bibliography{biblio}
\bibliographystyle{abbrv}


\end{document}